\documentclass[preprint,12pt,authoryear]{elsarticle}

\usepackage{lineno}
\usepackage{hyperref}
\usepackage{epstopdf}
\usepackage{multirow}
\usepackage{microtype}
\usepackage{graphicx}
\usepackage{subfigure}
\usepackage{booktabs}
\usepackage{amsmath}
\usepackage{amssymb}
\usepackage{mathtools}
\usepackage{amsthm}
\usepackage{color}
\usepackage{algorithm}
\usepackage{algorithmic}

\theoremstyle{plain}
\newtheorem{theorem}{Theorem}

\newtheorem{lemma}{Lemma}
\newtheorem{corollary}{Corollary}
\theoremstyle{definition}
\newtheorem{definition}{Definition}
\newtheorem{assumption}{Assumption}
\newtheorem{remark}{Remark}

\journal{ }

\begin{document}

\begin{frontmatter}

%% Title, authors and addresses

%% use the tnoteref command within \title for footnotes;
%% use the tnotetext command for theassociated footnote;
%% use the fnref command within \author or \affiliation for footnotes;
%% use the fntext command for theassociated footnote;
%% use the corref command within \author for corresponding author footnotes;
%% use the cortext command for theassociated footnote;
%% use the ead command for the email address,
%% and the form \ead[url] for the home page:
%% \title{Title\tnoteref{label1}}
%% \tnotetext[label1]{}
%% \author{Name\corref{cor1}\fnref{label2}}
%% \ead{email address}
%% \ead[url]{home page}
%% \fntext[label2]{}
%% \cortext[cor1]{}
%% \affiliation{organization={},
%%            addressline={}, 
%%            city={},
%%            postcode={}, 
%%            state={},
%%            country={}}
%% \fntext[label3]{}

\title{Stochastic ADMM with batch size adaptation for nonconvex nonsmooth optimization} %% Article title

\author{Jiachen Jin} %% Author name
\ead{jinjiachen@nudt.edu.cn}

\author{Kangkang Deng}
\ead{freedeng1208@gmail.com}

\author{Boyu Wang}
\ead{wangboyu20@nudt.edu.cn}

\author{Hongxia Wang\corref{cor1}}
\cortext[cor1]{Corresponding author.}
\ead{wanghongxia@nudt.edu.cn}

%% Author affiliation
\affiliation{organization={College of Science, National University of Defense Technology},%Department and Organization
            %addressline={}, 
            city={Changsha},
            postcode={410073}, 
            state={Hunan},
            country={China}}

%% Abstract
\begin{abstract}
%% Text of abstract
Stochastic alternating direction method of multipliers (SADMM) is a popular method for solving nonconvex nonsmooth optimization in various applications. However, it typically requires an empirical selection of the static batch size for gradient estimation, resulting in a challenging trade-off between variance reduction and computational cost.
This paper proposes adaptive batch size SADMM, a practical method that dynamically adjusts the batch size based on accumulated differences along the optimization path. We develop a simple convergence analysis to handle the dependence of batch size adaptation that matches the best-known complexity with flexible parameter choices.
We further extend this adaptive scheme to reduce the overall complexity of the popular variance-reduced methods, SVRG-ADMM and SPIDER-ADMM.
Numerical results validate the effectiveness of our proposed methods.
\end{abstract}

%%Graphical abstract
%\begin{graphicalabstract}
%%\includegraphics{grabs}
%\end{graphicalabstract}

%%Research highlights
% \begin{highlights}
% \item An adaptive batch size scheme for SADMM is proposed to balance speed and accuracy.

% \item The adaptation is extended to variance-reduced SADMM for optimized computational cost.

% \item Convergence analysis handles historical dependencies, matching best-known complexity.

% \item The algorithms are tested for fused logistic regression and graph-guided lasso.
% \end{highlights}

%% Keywords
\begin{keyword}
Stochastic ADMM \sep batch size adaptation \sep variance reduction \sep nonconvex nonsmooth optimization
\end{keyword}

\end{frontmatter}

%% Add \usepackage{lineno} before \begin{document} and uncomment 
%% following line to enable line numbers
%\linenumbers

%% main text
%%

\section{Introduction}\label{sec1}
This paper considers the problem of empirical risk (finite-sum) minimization, defined as:
		\begin{equation}\label{p}
			\min_{x,y}f(x)+g(y),\ \text{s.t.}~Ax+By=c,
		\end{equation}
		where $A\in\mathbb{R}^{m\times d_1}$, $B\in\mathbb{R}^{m\times d_2}$ and $c\in \mathbb{R}^{m}$. $f(x)=\frac{1}{n}\sum_{i=1}^{n}f_i(x):\mathbb{R}^{d_1}\to\mathbb{R}$ is a smooth function (with potentially very large $n$), and $g(y):\mathbb{R}^{d_2}\to\mathbb{R}$ is nonsmooth. Neither $f$ nor $g$ is necessarily convex. 
		This model covers a broad range of applications in science and engineering. A common instance is:
			\[\min_{x,y}f(x)+g(Ax).\]
			A key challenge in solving this problem is handling the nonconvexity of $f$ and the nonsmoothness of $g$ simultaneously. By introducing an auxiliary variable $y = Ax$, the problem is reformulated as a special case of \eqref{p}:
			\begin{equation}\label{op}
				\min_{x,y} f(x)+g(y),\ \text{s.t.}~Ax-y=0.
		\end{equation}
		Problem \eqref{op} arises frequently in fields ranging from statistics to machine learning, such as the graph-guided fused lasso, the reconstruction task in computed tomography and the training of deep neural networks with structural constraints, where $f$ typically represents a loss function and $g$ a structural regularization.

		The alternating direction method of multipliers (ADMM) \citep{GM76} is well-suited for such problems due to its variable-splitting capability, making it efficient for large-scale optimization problem \eqref{p} \citep{BPCPE11,NLRPJ15,ZGQYXZ24}. Its involves updating the primal variables $x$ and $y$ alternately, followed by a dual update, as follows:
		\begin{equation}\label{ADMM}
			\left\{
			\begin{aligned}
				&y_{k+1}=\arg\min_y \{\mathcal{L}_\beta(x_k,y,\lambda_k)\},\\
				&x_{k+1}=\arg\min_x \{\mathcal{L}_\beta(x,y_{k+1},\lambda_k)\},\\
				&\lambda_{k+1}=\lambda_k-\beta(Ax_{k+1}+By_{k+1}-c).
			\end{aligned}\right.
		\end{equation}
		Here $\mathcal{L}_\beta(\cdot)$ is the augmented Lagrangian function (ALF) of problem \eqref{p}:
		\[\mathcal{L}_\beta(x,y,\lambda)=f(x)+g(y)-\lambda^\top (Ax+By-c)+\frac{\beta}{2}\|Ax+By-c\|^2,\] 
		where $\lambda\in \mathbb{R}^{m}$ is the Lagrange multiplier and $\beta>0$ is a penalty parameter. 
		For nonconvex problems \eqref{op}, the convergence properties of the classic ADMM have been studied in \citep{guo2017IJCMconvergence}, with extensions to multi-block separable optimization in \citep{guo2017FMCconvergence}. Further analysis of problems \eqref{op} with nonseparable objectives was provided in \citep{guo2018convergence}. For general problems \eqref{p}, \cite{wu2017symmetric} proved the convergence of symmetric ADMM, which differs from the ADMM \eqref{ADMM} through the addition of an intermediate update of the multipliers. Combining both the Nesterov's acceleration technique and the relaxation scheme, \cite{bai2024accelerated} present a two-stage accelerated symmetric ADMM. \cite{WYZ19} studied the convergence of the classic ADMM under relaxer assumptions, in its multi-block and original cyclic update forms.  \cite{bai2025inexact} developed an inexact ADMM with an expansion line search step, which solves each subproblem inexactly to an adaptive accuracy and allows a larger range of dual stepsize. A comprehensive survey of the ADMM and its variants can be found in \citep{han2022survey}.
		
		Although the ADMM \eqref{ADMM}  has been widely studied in theory, it requires full-batch gradient evaluations of $f$ in its $x$-subproblem is therefore becomes computationally prohibitive for large $n$ in large-scale optimization problems \eqref{p}. Stochastic ADMM (SADMM) \citep{WB12,OHTG13,Suzuki13} addresses this issue by incorporating stochastic gradient descent (SGD) \citep{RM51} with the linearized ADMM framework \citep{YY13,OCLP15}.
		The updates in the classic SADMM for variables $y$ and $\lambda$ are consistent with the ADMM \eqref{ADMM}, while the $x$-subproblem uses a stochastic gradient approximation $\nabla f_{\mathcal{I}_k}(x_k)$ with a fixed batch size $M_k:=|\mathcal{I}_k|\equiv \mathcal{O}(\epsilon^{-1})$ to get an $\epsilon$-stationary point.
		Its iterative scheme for problem \eqref{p} is
		\begin{equation}\label{SADMM}
			\left\{
			\begin{aligned}
				&y_{k+1}=\arg\min_y \{\mathcal{L}_\beta(x_k,y,\lambda_k)\},\\
				&x_{k+1}=\arg\min_x \{\tilde{\mathcal{L}}_\beta(x,y_{k+1},\lambda_k,x_k,\nabla f_{\mathcal{I}_k}(x_k))\},\\
				&\lambda_{k+1}=\lambda_k-\beta(Ax_{k+1}+By_{k+1}-c).
			\end{aligned}\right.
		\end{equation}
		Here the approximation function is
		\[\begin{aligned}
			\tilde{\mathcal{L}}_\beta(x,y_{k+1},\lambda_k,x_k,\nabla f(x_k))
			=&f(x_k)+\nabla f(x_k)^\top (x-x_k)+\frac{1}{2\eta}\|x-x_k\|^2_G\nonumber\\
			+g(y_{k+1})&-\lambda_k^\top(Ax+By_{k+1}-c)+\frac{\beta}{2}\|Ax+By_{k+1}-c\|^2,%\label{LALF}
		\end{aligned}\]
		where $\eta>0$ and $G\succ 0$ is a positive definite matrix, usually selected as $G =rI_{d_1}- \beta\eta A^\top A$ with $r\geq \beta\eta \|A^\top A\|+1$, to avoid costly matrix inversions $(\frac{G}{\eta} +\beta A^\top A)^{-1}$. 
		Subsequent advances integrated variance reduction (VR) techniques and momentum acceleration, leading to faster SADMM variants for convex problems \citep{ZK16IJCAI,LSLKJL20,BHZ22,HGZ22}. Furthermore, VR-based extensions such as SVRG-ADMM \citep{HCL16,ZK16} and SPIDER-ADMM \citep{HCH19} have been developed for nonconvex settings, along with unified frameworks for nonsmooth nonconvex problem \eqref{op} \citep{BLZ21}. These VR-type methods use periodic full-batch gradient evaluations in outer-loop to correct the gradient variance to achieve better theoretical results. More recently, extensions such as convex ASVRG-ADMM \citep{LSLKJL20} have been developed for nonconvex objectives \citep{ZWBS24}. \citet{LHGL25} studied an inertial stochastic generalized ADMM framework with a class of VR gradient estimators.  Other innovations include zeroth-order estimators \citep{HGPH24} and inexact updates \citep{HHGZ23,ZBWW24}.
		
		In practice, existing stochastic ADMM methods rely heavily on static batch sizes that researchers must predefine. This leads to a practical trade-off: small batches reduce the cost per iteration but increase gradient variance, which impairs convergence, while large batches reduce variance but incur high computational overhead. In fact, small batches with gradient noise in the early stages can help escape sharp minima, thereby improving generalization in deep learning \citep{keskar2017large}. Recently, adaptive batch size strategies have been successfully applied to nonconvex SGD methods for solving unconstrained problems. These schemes use either predefined schedules \citep{FS12,ZYF18} or dynamic adjustments based on gradient variance \citep{DYJG16,DYJG17}, learning rates \citep{DNG17,HLT19}, or training loss \citep{SS19,LLK24}. For VR-SGD methods (e.g., SVRG and SPIDER), VR gradient estimators provide little benefit in the early stages of optimization when stochastic gradients exhibit directional coherence \citep{BRH16}. The adaptive scheme scales outer-loop batch sizes exponentially \citep{BAVSK15,LJ20} or based on historical gradient norms \citep{JWWZZL20,HG22}. Although batch size adaptation has shown promise in nonconvex SGD for unconstrained problems, such advances have yet to be extended to stochastic ADMM. This has left adaptive optimization for nonconvex constrained problems largely unexplored.

\subsection{Contributions}
		This paper introduces a simple scheme for solving constrained problems using nonconvex SADMM, balancing the computational burden with variance reduction. Specifically, this work does the following.
		
		We propose a practical SADMM (Algorithm \ref{alg1}) that scales the mini-batch size inversely with history differences. The scheme starts with mini batch sizes to accelerate initial convergence and gradually increases batch sizes as the iterations stabilize to ensure accurate, near-optimum precision.
		
		We generalize our adaptation strategy to VR variants and develop adaptive batch-size SVRG/SPIDER-ADMM (Algorithms \ref{alg2} and \ref{alg3}). These methods adjust full-batch gradients at each outer iteration based on the average differences from inner iterations, reducing redundant computations while maintaining the benefits of variance reduction.

		The presence of history differences in batch size updates makes it more challenging to establish a provable convergence guarantee than the existing analysis allows. Our analysis shows that, under flexible parameter choices, the worst-case complexity of the proposed adaptive methods matches the best-known results (Theorem \ref{AdsCR1}).
		
		Experiments on fused logistic regression and graph-guided regularized Lasso problems demonstrate the practical advantages of batch size adaptation compared to their static batch size counterparts.
		
		\subsection{Outline}
		The rest of this paper is organized as follows. Section \ref{sec2} presents some preliminary results that will be used throughout this paper. Section \ref{sec3} describes the algorithm in detail and establishes the convergence properties. In Section \ref{sec4}, we extend the batch size adaptation to two popular VR variants, followed by numerical results in Section \ref{sec5}. Section \ref{sec6} concludes the paper.

\section{Preliminaries}\label{sec2}
		First, we give some notation used in this paper. $\|\cdot\|$ denotes the Euclidean norm of a vector or the spectral norm of a matrix. $\varsigma_A$ denotes the smallest eigenvalues of matrix $A^\top A$. $\zeta_{\min}$ and $\zeta_{\max}$ denote the smallest and largest eigenvalues of positive definite matrix $G\succ 0$. Let $\|x\|^2_G =x^\top Gx$, then $\zeta_{\min}\|x\|^{2}\leq\|x\|^{2}_{G}\leq\zeta_{\max}\|x\|^{2}$. $A^+$ denotes the generalized inverse of matrix $A$. $I_n\in\mathbb{R}^{n\times n}$ is an identity matrix of size $n$. For any $x,~y\in\mathbb{R}^{n}$, $\langle x,y\rangle=x^\top y$. $\partial g$ denotes the subgradient of $g$. For a nonempty closed set $C$, ${\rm dist}(x,C) = \inf_{y\in C} \|x-y\|$ denotes the distance from $x$ to $C$.
		
		Then, we define the $\epsilon$-stationary point of problem \eqref{p}. The Lagrangian function of \eqref{p} is defined as $\mathcal{L}(w)=f(x)+g(y)-\langle\lambda, Ax+By-c\rangle$.
		\begin{definition}\label{stationary}
			For $w:=(x,y,\lambda)$, denote
			\[
			\partial\mathcal{L}(w):=\begin{bmatrix}
				\nabla f(x)-A^\top\lambda\\
				\partial g(y)-B^\top\lambda\\
				Ax+By-c
			\end{bmatrix}.
			\]
			Given $\epsilon>0$, the point $w^*:=(x^*,y^*,\lambda^*)$ is called an $\epsilon$-stationary point of problem \eqref{p} if it holds that
			\[\mathbb{E}[{\rm dist}(0,\partial \mathcal{L}(w^*))^2]\leq \epsilon.\]
		\end{definition}
		
		Next, we give some standard assumptions regarding problem \eqref{p} as follows:
		\begin{assumption}\label{nograb}
			For a smooth function $f(x)$, there exists a noisy estimation to $\nabla f(x)$, and the noisy estimation $\nabla f(x,\xi)$ satisfies $\mathbb{E}[\nabla f(x,\xi)]=\nabla f(x)$ and $\mathbb{E}\|\nabla f(x,\xi)-\nabla f (x)\|^2\leq \sigma^2$, where the expectation is taken with respect to the random variable $\xi$ and $\sigma> 0$.
		\end{assumption}
		
		From Assumption \ref{nograb}, we have $\mathbb{E}[\nabla f_\mathcal{I}(x)]=\nabla f(x)$ and $\mathbb{E}\|\nabla f_\mathcal{I}(x)-\nabla f (x)\|^2\leq \frac{\sigma^2}{M}$ for the size $M$ of a mini-batch $\mathcal{I}$ and the stochastic gradient $\nabla f_\mathcal{I}(x)=\frac{1}{M}\sum_{i=1}^{M}\nabla f(x,\xi_i)$.

		\begin{assumption}\label{Lsm}
			$\nabla f(x)$ is Lipschitz continuous with the constant $L > 0$, i.e.,
			\[\|\nabla f(x)-\nabla f(y)\|\leq L \|x-y\|,~\forall~x,y\in \mathbb{R}^{d_1}.\]
		\end{assumption}
		\begin{assumption}\label{fucor}
				$A$ is a full column rank matrix.
		\end{assumption}
		\begin{assumption}\label{lb}
			$f(x)$ and $g(y)$ are all lower bounded, and let $f^*=\inf_x f(x)>-\infty$ and $g^*=\inf_y g(y)>-\infty$.
		\end{assumption}
		\begin{assumption}\label{grab}
			The gradient of loss function $f(x)$ is bounded, i.e., there exists a constant $\mu>0$ such that for all $x$, it follows $\|\nabla f(x)\|\leq \mu$.
		\end{assumption}
		
		Assumption \ref{nograb} is a standard condition in stochastic optimization known as ``finite variance'' \citep{GLZ16,ZBWW24}.
		Assumption \ref{Lsm} is a common regularity condition in nonconvex optimization \citep{GL13}. 
		Assumptions \ref{fucor}-\ref{grab} have been used in nonconvex ADMM \citep{HLR16,HC18,JLMZ19}. 
		Note that Assumptions \ref{Lsm}, \ref{lb} and \ref{grab} hold for common objective functions in machine learning, such as sigmoid loss $f(x)=\frac{1}{1+\exp(x)}$ and $L_1$-regularization $g(y)=\|y\|_1$.

\section{Stochastic ADMM with batch size adaptation}\label{sec3}
This section proposes an adaptive batch size SADMM to solve problem \eqref{p}. Moreover, we study its convergence properties.

		\subsection{Proposed algorithms}
		The convergence of SADMM-type methods depends on bounding the descent of a merit function based on the ALF $\mathcal{L}_\beta(w_k)$. 
		For the SADMM algorithm \eqref{SADMM}, define a merit function $\varphi_k=\mathbb{E}[\mathcal{L}_\beta(w_k)+\gamma\|x_k-x_{k-1}\|^2]$ with constant $\gamma$ . The convergence analysis in \citep{HC18} has shown that the iterative decrease of $\varphi_k$ is
		\begin{equation}\label{MFbound1}
			\varphi_{k+1}-\varphi_k \leq \mathbb{E}(\frac{\gamma_1}{M_k}-\gamma_2\|x_{k+1}-x_k\|^2),
		\end{equation}
		with constants $\gamma_1$ and $\gamma_2$. 
		 Summing over $k=0$ to $K-1$ yields:
			\[
			\frac{\gamma_2}{K}\sum_{k=0}^{K-1}\mathbb{E}\|x_{k+1}-x_k\|^2\leq \frac{\varphi_0-\varphi_K}{K}+ \frac{1}{K}\sum_{k=0}^{K-1}\mathbb{E}\frac{\gamma_1}{M_k},
			\]
			To further achieve an $\epsilon$-stationary point, the above inequality suggests that the classic SADMM \eqref{SADMM} sets $K$ and $M_k$ to $\mathcal{O}(\epsilon^{-1})$.
			Note that the bound \eqref{MFbound1} naturally suggests that $M_k$ can be chosen such that the first term is at the same level as the second term, i.e., the batch size $M_k$ should adapt to $\|x_{k+1}-x_k\|^{-2}$. In this case, the descent of the merit function would be bounded solely by $\|x_{k+1}-x_k\|^{-2}$, and the convergence guarantee could be followed without setting a static batch size $M_k\equiv \mathcal{O}(\epsilon^{-1})$.
		However, since $\|x_{k+1}-x_k\|^{-2}$ is unknown at step $k$, we approximate it using the previous difference $\|x_k-x_{k-1}\|^{-2}$. This motivates the adaptive rule:
		\[M_k=\min\{c_\tau\sigma^2\|x_k-x_{k-1}\|^{-2},c_\epsilon\sigma^2\epsilon^{-1}\},\]
		where $c_\tau,c_\epsilon>0$. To incorporate this adaptive choice into the convergence analysis, we introduce a new merit function: $\phi_k=\mathbb{E}[\mathcal{L}_\beta(w_k)+\rho_1\|x_k-x_{k-1}\|^2+\rho_2\|x_{k-1}-x_{k-2}\|^2]$ with constants $\rho_1$ and $\rho_2$, which satisfies (as shown in Lemma \ref{SM1MFD}):
		\[
		\phi _{k+1}-\phi _k \leq -\delta_1 \mathbb{E}\|x_{k+1}-x_k\|^2+\delta_2 \epsilon,
		\]
		with constants $\delta_1$ and $\delta_2$. 
			Telescoping the above inequality, we have
			\[\frac{\delta_1}{K}\sum_{k=0}^{K-1}\mathbb{E}\|x_{k+1}-x_k\|^2\leq \frac{\phi_0-\phi_K}{K}+\delta_2 \epsilon,\]
			so setting $K=\mathcal{O}(\epsilon^{-1})$ suffices to reach an $\epsilon$-stationary point. The adaptive rule thus preserves convergence while allowing $M_k$ to be small when the iterates are changing rapidly, thereby reducing the overall computational cost. As the algorithm progresses, the difference norm is expected to decrease, also serving as an indicator of the optimization stage. 
		The complete adaptive batch size SADMM (AbsSADMM) is formalized in the Algorithm \ref{alg1}.
		
		\begin{algorithm}[h]
			\caption{AbsSADMM}\label{alg1}
			\begin{algorithmic}[1]
				\REQUIRE $w_0,K,\eta,\beta,c_\tau,c_\epsilon$;
				\STATE $x_{-1}=x_0$;
				\FOR{$k=0,1,\cdots,K-1$}
				
				\STATE Sample $\mathcal{I}_k$ from $[n]$ with $|\mathcal{I}_k|=\lceil M_k \rceil$, where\\
				$M_k=\min\{c_\tau\sigma^2\|x_k-x_{k-1}\|^{-2},c_\epsilon\sigma^2\epsilon^{-1}\}$;
				
				\STATE $y_{k+1}=\arg\min_y \{\mathcal{L}_\beta(x_k,y,\lambda_k)\}$;
				
				\STATE $x_{k+1}=\arg\min_x \{\tilde{\mathcal{L}}_\beta(x,y_{k+1},\lambda_k,x_k,\nabla f_{\mathcal{I}_k}(x_k))\}$;
				
				\STATE $\lambda_{k+1}=\lambda_k-\beta(Ax_{k+1}+By_{k+1}-c)$.
				\ENDFOR
			\end{algorithmic}
		\end{algorithm}

\subsection{Convergence analysis for AbsSADMM}\label{AbsSADMM}
This subsection analyses the convergence properties of Algorithm \ref{alg1}. Our method differs from the existing theoretical framework \citep{HC18} in that it dynamically scales $M_k$ with respect to $\|x_k-x_{k-1}\|^{-2}$, whereas the latter fixes $M_k\equiv\mathcal{O}(\epsilon^{-1})$. Here, we present a straightforward convergence analysis that effectively addresses the dependence of batch size adaptation and matches the optimal complexity, while enabling flexible parameter selection.
		
		Define a new merit function as follows,
		\[
		\phi_k=\mathbb{E}\left[\mathcal{L}_\beta(w_k)+\delta\|x_k-x_{k-1}\|^2
		+\gamma\|x_{k-1}-x_{k-2}\|^2\right], 
		\]
		where $\delta=\frac{1}{2c_\tau}+\frac{5}{\beta  \varsigma_A}(\frac{2}{c_\tau}+\frac{\zeta_{\max}^2}{\eta^2}+L^2)$ and $\gamma=\frac{5}{c_\tau\beta \varsigma_A}$. Let $\theta_k=\mathbb{E}[\|x_{k+1}-x_k\|^2+\|x_k-x_{k-1}\|^2+\|x_{k-1}-x_{k-2}\|^2]$. We prove the following convergence of Algorithm \ref{alg1}.
		
		\begin{theorem}\label{AdsCR1}
			Suppose the sequence $\{w_k\}$ is generated from Algorithm \ref{alg1} and Assumptions \ref{nograb}-\ref{grab} hold. Define $\psi_1=\beta^2 \|B\|\|A\|$, $\psi_2=3(\frac{1}{c_\tau}+L^2+\frac{\zeta_{\max}^2}{\eta^2})$ and $\psi_3=\frac{5}{\beta^2 \varsigma_A}(\frac{1}{c_\tau}+\frac{\zeta_{\max}^2}{\eta^2}+L^2)$. Let $c_\epsilon,\epsilon >0$ and choose constants $\eta,\beta,c_\tau>0$ such that $\rho=\frac{\zeta_{\min}}{\eta}+\frac{\beta\varsigma_A}{2}-\frac{L+1}{2}-\frac{10\zeta_{\max}^2}{\beta \varsigma_A\eta^2}-\frac{1}{2c_\tau}-\frac{10}{c_\tau \beta \varsigma_A}-\frac{5L^2}{\beta \varsigma_A}>0$.
			Then we have
			\[
			\min_{1\leq k \leq K}\mathbb{E}[{\rm dist}(0,\partial \mathcal{L}(w_k))^2]\leq  \frac{3\psi(\phi_0-\phi^*)}{\rho K}+\frac{\kappa}{c_\epsilon}\epsilon.
			\]
			where $\phi^*$ denotes a lower bound of $\phi_k$, $\psi=\max\{\psi_1,\psi_2,\psi_3\}$, and   $\kappa=\frac{1}{\rho}+\frac{20}{\rho\beta\varsigma_A}+\max\{\frac{10}{\beta^2 \varsigma_A},3\}$.
			Given $c_\epsilon=2\kappa$, it implies that $K$ satisfies
			\[K=6\psi(\phi_0-\phi^*)\rho^{-1}\epsilon^{-1}=\mathcal{O}(\epsilon^{-1}),\]
			then $w_{k^*}$ is an $\epsilon$-stationary point of \eqref{p}, where $k^*=\arg\min_k \theta_k$.
		\end{theorem}
		\begin{proof}
			Below we provide a sketch of the proof while the detailed one can be found in \ref{THE1}.
			
			Step 1. According to the optimality condition of each subproblem in Algorithm \ref{alg1}, we get the upper bound of $\mathbb{E}\|\lambda_{k+1}-\lambda_k\|^2$ (see also Lemma \ref{SM1bdual}),
			\[\begin{aligned}
				\mathbb{E}\|\lambda_{k+1}-\lambda_k\|^2 \leq & \frac{5}{\varsigma_A}\mathbb{E}\|\nabla f_{\mathcal{I}_k}(x_k)-\nabla f(x_k)\|^2+\frac{5}{\varsigma_A}\mathbb{E}\|\nabla f(x_{k-1})-\nabla f_{\mathcal{I}_{k-1}}(x_{k-1})\|^2\\
				&+\frac{5\zeta_{\max}^2}{\varsigma_A\eta^2}\mathbb{E}\|x_{k+1}-x_k\|^2+\frac{5}{\varsigma_A}(\frac{\zeta_{\max}^2}{\eta^2}+L^2)\mathbb{E}\|x_k-x_{k-1}\|^2.
			\end{aligned}\]
			
			Step 2. By the upper bound of $\mathbb{E}\|\lambda_{k+1}-\lambda_k\|^2$, we bound the decrease of $\phi_k$ over an epoch $k$ by follows (see also Lemma \ref{SM1MFD}),
			\[\phi_{k+1}\leq\phi_k-\rho\mathbb{E}\|x_{k+1}-x_k\|^2+(\frac{1}{2}+\frac{10}{\beta\varsigma_A})\frac{\epsilon}{ c_\epsilon}.\]
			
			Step 3. Telescoping the above inequality over $k$ and combining the lower boundedness of $\{\phi_k\}$, we obtain the claimed result.
		\end{proof}
		
		\begin{remark}
			Theorem \ref{AdsCR1} establishes convergence guarantees for Algorithm \ref{alg1} under the condition $\rho>0$, which allows for wide flexibility in choosing the positive parameters $\eta,\beta,c_\tau$. To ensure $\rho>0$, the inequality
			\begin{equation}\label{rhopar}
				-2\rho=\frac{20\zeta_{\max}^2}{\beta \varsigma_A\eta^2}-\frac{2\zeta_{\min}}{\eta}+\Theta<0,
			\end{equation}
			must hold, where $\Theta =1+L+\frac{1}{c_\tau}+\frac{20}{c_\tau \beta \varsigma_A}+\frac{10L^2}{\beta \varsigma_A}-\beta\varsigma_A$.
			This inequality forms a quadratic expression in $\frac{1}{\eta}$, so its discriminant must satisfy
			\begin{equation}\label{etapar}
				\Delta_\eta:=4\zeta_{\min}^2-\frac{80\zeta_{\max}^2}{\beta \varsigma_A}\Theta>0.
			\end{equation}
			Given $\varsigma_A>0$, \eqref{etapar} is ensured if $\Theta<0$, leading to a secondary quadratic inequality in $\beta$,
			\begin{equation}\label{betapar}
				c_\tau\varsigma_A^2\beta^2-(1+c_\tau+c_\tau L))\varsigma_A \beta-10(2+c_\tau L^2)>0.
			\end{equation} 
			The positive leading coefficient of the quadratic and the guaranteed positive discriminant (for $c_\tau>0$) imply that $\beta>\beta^+:=\frac{1+c_\tau+c_\tau L+\sqrt{(1+c_\tau+c_\tau L)^2+40c_\tau(2+c_\tau L^2)}}{2c_\tau\varsigma_A}$, where $\beta^+$ is the larger root. Then \eqref{betapar} holds and so does \eqref{etapar}.
			Returning to the original quadratic \eqref{rhopar} in $\frac{1}{\eta}$, the feasible region for $\eta$ is determined by its larger root $\frac{1}{\eta^+}:=\frac{(2\zeta_{\min}+\sqrt{\Delta_\eta})\varsigma_A \beta}{40\zeta_{\max}^2}$. Thus, the parameter space that ensures convergence is
			\[c_\tau>0,\beta>\beta^+,\eta>\eta^+.\]
		\end{remark}
		
		\begin{remark}
			Theorem \ref{AdsCR1} states that the algorithm \ref{alg1} achieves a $\mathcal{O}(\frac{1}{K})$ convergence rate by batch size adaptation. The complexity is derived as
			\[
			\sum_{k=0}^{K-1}M_k=\sum_{k=0}^{K-1}\min\{c_\tau\sigma^2\|x_k-x_{k-1}\|^{-2},c_\epsilon\sigma^2\epsilon^{-1}\}
			\leq\underset{\rm complexity~of~SADMM}{\underbrace{Kc_\epsilon\sigma^2\epsilon^{-1}=\mathcal{O}(\epsilon^{-2})}}.
			\]
			This shows that the worst-case complexity of AbsSADMM matches the best known result of SADMM \citep{HC18,ZBWW24}. In particular, the adaptive term $c_\tau\sigma^2\|x_k-x_{k-1}\|^{-2}$ can be significantly lower than the static term $c_\epsilon\sigma^2\epsilon^{-1}$, particularly during initial iterations. Consequently, AbsSADMM achieves practical efficiency gains over static-batch SADMM by using smaller batches early in the optimization.
		\end{remark}
        
\section{Extentions to variance reduction algorithms}\label{sec4}
We further extend the batch-size adaptation scheme to the popular VR algorithms SVRG-ADMM \citep{HCL16,ZK16} and SPIDER-ADMM \citep{HCH19}. Although they theoretically outperform SADMM by reducing the variance in gradient estimation, in practice they often suffer from slow convergence due to the full gradient evaluation in each outer loop. We refer to the adaptive methods as AbsSVRG-ADMM and AbsSPIDER-ADMM in Algorithm \ref{alg2} and Algorithm \ref{alg3}, respectively.
		
		\begin{algorithm}[h]
			\caption{AbsSVRG-ADMM}\label{alg2}
			\begin{algorithmic}[1]
				\REQUIRE $w_0^1,K,T,S=\lceil K/T\rceil,\eta,\beta,c_\tau,c_\epsilon,\tau_1,b$;
				\STATE $\tilde{x}^0=x_0^1$;
				
				\FOR{$s=1,2,\cdots,S$}
				
				\STATE Sample $\mathcal{N}_s$ from $[n]$ with $|\mathcal{N}_s|=\lceil N_s\rceil$, where batch size
				$N_s=\min\{c_\tau\sigma^2\tau_s^{-1},c_\epsilon\sigma^2\epsilon^{-1},n\}$;
				
				\STATE $g^s=\nabla f_{\mathcal{N}_s} (\tilde{x}^{s-1})$;
				
				\STATE Set $\tau_{s+1}=0$;
				
				\FOR{$t=0,1,\cdots,T-1$}
				
				\STATE Uniformly randomly pick a mini-batch $\mathcal{I}$ (with replacement) from $[n]$ with $|\mathcal{I}|=b$;
				
				\STATE $v_t^s=\nabla f_{\mathcal{I}}(x_t^s)-\nabla f_{\mathcal{I}}(\tilde{x}^{s-1})+g^s$;
				
				\STATE $y_{t+1}^s=\arg\min_y \{\mathcal{L}_\beta(x_t^s,y,\lambda_t^s)\}$;
				
				\STATE $x_{t+1}^s=\arg\min_x \{\tilde{\mathcal{L}}_\beta(x,y_{t+1}^s,\lambda_t^s,x_t^s,v_t^s)\}$;
				
				\STATE $\lambda_{t+1}^s=\lambda_t^s-\beta(Ax_{t+1}^s+By_{t+1}^s-c)$;
				
				\STATE $\tau_{s+1}\gets \tau_{s+1}+\|x_{t+1}^s-x_t^s\|^2/T$;
				\ENDFOR
				\STATE $\tilde{x}^s=x_0^{s+1}=x^s_T$, $y^{s+1}_0=y^s_T$, $\lambda^{s+1}_0=\lambda^s_T$.
				\ENDFOR
			\end{algorithmic}
		\end{algorithm}
		
		\begin{algorithm}[h]
			\caption{AbsSPIDER-ADMM}\label{alg3}
			\begin{algorithmic}[1]
				\REQUIRE $w_0,K,q,\eta,\beta,c_\tau,c_\epsilon,\tau_0,b$;
				
				\FOR{$k=0,1,\cdots,K-1$}
				
				\IF{mod $(k,q)=0$}
				
				\STATE Sample $\mathcal{N}_k$ from $[n]$ with $|\mathcal{N}_k|=\lceil N_k\rceil$, where batch size
				$N_k=\min\{c_\tau\sigma^2\tau_k^{-1},c_\epsilon\sigma^2\epsilon^{-1},n\}$, and compute $v_k=\nabla f_{\mathcal{N}_k}(x_k)$;
				
				\STATE Set $\tau_{k+1}=0$;
				
				\ELSE
				\STATE Uniformly randomly pick a mini-batch $\mathcal{I}$ (with replacement) from $[n]$ with $|\mathcal{I}|=b$;
				
				\STATE $v_k=\nabla f_{\mathcal{I}}(x_k)-\nabla f_{\mathcal{I}}(x_{k-1})+v_{k-1}$;
				
				\ENDIF
				
				\STATE $y_{k+1}=\arg\min_y \{\mathcal{L}_\beta(x_k,y,\lambda_k)\}$;
				
				\STATE $x_{k+1}=\arg\min_x \{\tilde{\mathcal{L}}_\beta(x,y_{k+1},\lambda_k,x_k,v_k)\}$;
				
				\STATE $\lambda_{k+1}=\lambda_k-\beta(Ax_{k+1}+By_{k+1}-c)$;
				
				\STATE $\tau_{k+1} \gets \tau_{k+1}+\|x_{k+1}-x_k\|^2/q$.
				
				\ENDFOR
			\end{algorithmic}
		\end{algorithm}

\subsection{Convergence analysis for AbsSVRG-ADMM}
		Let $w_t^s=(x_t^s,y_t^s,\lambda_t^s)$. Define a merit function $\phi_k$ as follows:
		\[
		\phi_t^s=\mathbb{E}[\mathcal{L}_\beta(w_t^s)+\frac{5}{\beta\varsigma_A}(\frac{\zeta_{\max}^2}{\eta^2}+L^2)\|x_t^s-x_{t-1}^s\|^2
		+\frac{5L^2}{b\beta\varsigma_A}\|x_{t-1}^s-\tilde{x}^{s-1}\|^2+\delta_t\|x_t^s-\tilde{x}^{s-1}\|^2],
		\]
		where the positive sequence $\{\delta_t\}$ satisfies $\delta_t=\frac{L^2}{2b}+\frac{10L^2}{b\beta\varsigma_A}+(1+\rho)\delta_{t+1}$
		with $\rho>0$. Let $\theta_t^s=\mathbb{E}[\|x_{t+1}^s-x_t\|^2+\|x_t^s-x_{t-1}^s\|^2+\frac{1}{b}(\|x_t^s-\tilde{x}^s\|^2+\|x_{t-1}^s-\tilde{x}^s\|^2)]$ and $K = ST$ represents the total number of iterations. We give the convergence of Algorithm \ref{alg2}.
		\begin{theorem}\label{AdsCR2}
			Suppose the sequence $\{w_t^s\}$ is generated from Algorithm \ref{alg2} and Assumptions \ref{nograb}-\ref{grab} hold. Let $\epsilon,c_\epsilon>0$. Choose the parameters $\tau_1,\eta,\beta,c_\tau,\rho>0$ used in Algorithm \ref{alg2} such that $\tau_1\leq \epsilon S$ and $\Omega_t:=\frac{\zeta_{\min}}{\eta}+\frac{\beta\varsigma_A}{2}-\frac{5}{\beta \varsigma_A}(L^2+\frac{2}{c_\tau}+\frac{2\zeta_{\max}^2}{\eta^2})-\frac{1+c_\tau(L+1)}{2c_\tau}-(1+\frac{1}{\rho})\delta_{t+1}>0$. Define $\psi_1=\beta^2 \|B\|\|A\|$, $\psi_2=3(L^2+\frac{\zeta_{\max}^2}{\eta^2})$ and $\psi_3=\frac{5}{\beta^2  \varsigma_A}(\frac{\zeta_{\max}^2}{\eta^2}+L^2)$. Then we have
			\[
			\min_{1\leq k \leq K}\mathbb{E}[{\rm dist}(0,\partial \mathcal{L}(w_t^s))^2]
			\leq\frac{2\omega_1(\phi_0^1-\phi^*)}{\gamma K}+\omega_2(\frac{1}{c_\tau}+\frac{1}{c_\epsilon})\epsilon.
			\]
			where $\gamma=\min\{\Omega_t,\frac{L^2}{2}\}$, $\omega_1=\psi+\frac{1}{c_\tau}\max\{3,\frac{10}{\varsigma_A\beta^2}\}$, $\psi=\max\{\psi_1,\psi_2,\psi_3\}$, $\omega_2=\max\{3,\frac{10}{\varsigma_A\beta^2}\}+1+\frac{20}{\beta\varsigma_A}$, and $\phi^*$ denotes a lower bound of $\phi_t^s$.
		\end{theorem}
		
		\begin{proof}
			Below we also provide a sketch and refer to \ref{THE2} for details.
			
			Step 1. We first establish the upper bound for the estimation variance $\mathbb{E}_{0,s}\|v_t^s-\nabla f(x_t^s)\|^2$. (see also Lemma \ref{SM2SG})
			\[
			\mathbb{E}_{0,s}\|v_t^s-\nabla f(x_t^s)\|^2\leq \frac{L^2}{b} \mathbb{E}_{0,s}\|x_t^s-\tilde{x}^{s-1}\|^2+\frac{I_{(N_s<n)}}{N_s}\sigma^2,
			\]
			where $\mathbb{E}_{t,s}(\cdot)$ denotes $\mathbb{E}(\cdot|x_0^1,x_0^2,\dots,x_2^1,\dots,x_t^s)$, and $I_{(A)} = 1$ if the event $A$ occurs and 0 otherwise.
			
			Step 2. According to the optimality condition of each subproblem in Algorithm \ref{alg2}, we have an upper bound of $\mathbb{E}_{0,s}\|\lambda_{t+1}^s-\lambda_t^s\|^2$ (see also Lemma \ref{SM2bdual}),
			\[\begin{aligned}
				\mathbb{E}_{0,s}\|\lambda_{t+1}^s-\lambda_t^s\|^2\leq &\frac{5}{\varsigma_A}\mathbb{E}_{0,s}\|v_t^s-\nabla f(x_t^s)\|^2+\frac{5}{\varsigma_A}\mathbb{E}_{0,s}\|\nabla f(x_{t-1}^s)-v_{t-1}^s\|^2\\
				&+\frac{5\zeta_{\max}^2}{\varsigma_A\eta^2}\mathbb{E}_{0,s}\|x_{t+1}^s-x_t^s\|^2+\frac{5}{\varsigma_A}(\frac{\zeta_{\max}^2}{\eta^2}+L^2)\mathbb{E}_{0,s}\|x_t^s-x_{t-1}^s\|^2.
			\end{aligned}\]
			%	\[\begin{aligned}
				%		\mathbb{E}\|\lambda_{t+1}^s-\lambda_t^s\|^2 \leq& \frac{5L^2}{b \varsigma_A}(\mathbb{E}\|x_t^s-\tilde{x}^s\|^2+\mathbb{E}\|x_{t-1}^s-\tilde{x}^s\|^2)+\frac{10\sigma^2I_{(N_s<n)}}{N_s \varsigma_A}\nonumber\\
				%		&+\frac{5\zeta_{\max}^2}{\varsigma_A\eta^2}\mathbb{E}\|x_{t+1}^s-x_t^s\|^2+\frac{5}{\varsigma_A}(\frac{\zeta_{\max}^2}{\eta^2}+L^2)\mathbb{E}\|x_t^s-x_{t-1}^s\|^2.
				%	\end{aligned}\]

			Step 3. By the upper bound of $\mathbb{E}\|\lambda_{t+1}^s-\lambda_t^s\|^2$, we get (see also Lemma \ref{SM2MFD}),
			$$\frac{1}{ST}\sum_{s=1}^{S}\sum_{t=0}^{T-1}(\Omega_t\mathbb{E}\|x_{t+1}^s-x_t^s\|^2+\frac{L^2}{2b}\mathbb{E}\|x_t^s-\tilde{x}^{s-1}\|^2)\leq \frac{\phi_0^1-\phi^*}{ST}+(\frac{1}{2}+\frac{10}{\beta\varsigma_A})(\frac{1}{c_\tau}+\frac{1}{c_\epsilon})\epsilon.$$
			%where $\phi^*$ denotes a lower bound of $\phi_k$ and $\Omega_t=\frac{\zeta_{\min}}{\eta}+\frac{\beta\varsigma_A}{2}-\frac{1}{2c_\tau}-\frac{L+1}{2}-\frac{5}{\beta \varsigma_A}(L^2+\frac{2}{c_\tau}+\frac{2\zeta_{\max}^2}{\eta^2})-(1+\frac{1}{\rho})\delta_{t+1}$.
			
			Step 4. Combining the above inequality and Definition \ref{stationary}, we obtain the claimed result.
		\end{proof}
		
		The following corollary provides the $\mathcal{O}(\frac{1}{K})$ convergence rate and the oracle complexity of Algorithm \ref{alg2} under certain choices of parameters. %The proof of the corollary is included in \ref{coro1}.
		\begin{corollary}\label{AbssvrgOC}
			Under the setting of Theorem \ref{AdsCR2}, choose $T=[n^{\frac{1}{3}}]$, $b=[n^{\frac{2}{3}}]$, $\rho=\frac{1}{T}$, $\eta=\frac{2\zeta_{\min}}{5L^2+L+2}$, $c_\tau=c_\epsilon=9+\frac{12}{L^2}(3+\beta^2 \|B\|\|A\|)$ and
			\[
			\beta\geq\max\left\{\frac{20}{\varsigma_A},\sqrt{\frac{10}{3\varsigma_A}},\sqrt{\frac{3}{\|B\|\|A\|}(L^2+\frac{\zeta_{\max}^2}{\eta^2})},\frac{1}{\varsigma_A}\sqrt{10(2+9L^2+\frac{2\zeta_{\max}^2}{\eta^2})}\right\}.
			\]
			Then, to find an $\epsilon$-approximate solution, the total number of iterations required by Algorithm \ref{alg2} is given by \[K=\frac{12(3+\|B\|\|A\|)(\phi_0^1-\phi^*)}{L^2 \epsilon}=\mathcal{O}(\epsilon^{-1}),\]
			and the oracle complexity of Algorithm \ref{alg2} is given by
			\[
			\sum_{s=1}^{S}\min\{c_\tau\sigma^2\tau_s^{-1},c_\epsilon\sigma^2\epsilon^{-1},n\}+Kb
			\leq\underset{\rm complexity~of~SVRG-ADMM}
			{\underbrace{S\min\{c_\epsilon\sigma^2\epsilon^{-1},n\}+Kb=\mathcal{O}(n+n^{\frac{2}{3}}\epsilon^{-1})}}.
			\]
		\end{corollary}
		
		Under the specified parameter choices, the worst-case complexity of AbsSVRG-ADMM matches the existing optimal complexity $\mathcal{O}(n+n^{\frac{2}{3}}\epsilon^{-1})$ established for SVRG-ADMM \citep{HCL16,ZWBS24}. In practice, the batch size adaptation often significantly reduces the static complexity of SVRG-ADMM. This occurs because the adaptive term $c_\tau\sigma^2\tau_s^{-1}$ is often smaller than the static threshold $\min\{c_\epsilon\sigma^2\epsilon^{-1},n\}$, especially in the early iterations, as shown empirically in Section \ref{sec5}.

\subsection{Convergence analysis for AbsSPIDER-ADMM}
In this subsection, we analyze the convergence properties of Algorithm \ref{alg3}. Let $k\gg q$ and $n_k=\lceil  k/q \rceil$ such that $(n_k-1)q\leq k \leq n_k q-1$.

Define a merit function $\phi_k$ as follows:
	\[
	\phi_k=\mathbb{E}\left[\mathcal{L}_\beta(w_k)+\frac{5}{\beta\varsigma_A}(\frac{\zeta_{\max}^2}{\eta^2}+L^2)\|x_k-x_{k-1}\|^2+\frac{5L^2}{b\beta\varsigma_A}\sum_{i=(n_k-1)q}^{k-1} \mathbb{E}\|x_{i+1}-x_i\|^2\right].
	\]
Let $\theta_k=\mathbb{E}[\|x_{k+1}-x_k\|^2+\|x_k-x_{k-1}\|^2+\frac{1}{q}\sum_{i=(n_k-1)q}^{k}\|x_{i+1}-x_i\|^2]$. Based on the above lemma, we give the convergence of Algorithm \ref{alg3}.
\begin{theorem}\label{AdsCR3}
	Suppose the sequence $\{w_k\}$ is generated from Algorithm \ref{alg3} and Assumptions \ref{nograb}-\ref{grab} hold. Let $\epsilon >0$ and $c_\tau,c_\epsilon>0$. Choose the parameters $b=q$ and $\eta,\beta,c_\tau,\tau_0>0$ such that $\chi=\frac{\zeta_{\min}}{\eta}+\frac{\beta\varsigma_A}{2}-\frac{L+1}{2}-\frac{10\zeta_{\max}^2}{\beta \varsigma_A\eta^2}-\frac{5L^2}{\beta \varsigma_A}-\frac{5L^2}{b\beta\varsigma_A}-\frac{qL^2}{2b}-\frac{10qL^2}{b\beta\varsigma_A}-\frac{1}{2c_\tau}-\frac{10}{c_\tau\beta\varsigma_A}>0$ and $\tau_0=c_d\epsilon$ with $1\leq c_d\leq q$. Define $\psi_1=\beta^2 \|B\|\|A\|$, $\psi_2=3(L^2+\frac{\zeta_{\max}^2}{\eta^2})$, $\psi_3=\frac{5}{\beta^2 \varsigma_A}(2L^2+\frac{\zeta_{\max}^2}{\eta^2})$. Then we have
	\[
	\min_{1\leq k \leq K}\mathbb{E}[{\rm dist}(0,\partial \mathcal{L}(w_k))^2]\leq\frac{3\omega_1(\phi_0-\phi^*)}{\chi K}+\omega_2(\frac{c_d}{c_\tau}+\frac{1}{c_\epsilon})\epsilon,
	\]
	where $\omega_1=\psi+\frac{1}{c_\tau}\max\{3,\frac{10}{\varsigma_A\beta^2}\}$, $\omega_2=\frac{3\psi}{\chi}(\frac{1}{2}+\frac{10}{\beta\varsigma_A})+\max\{\frac{10}{\beta^2 \varsigma_A},3\}$, $\psi=\max\{\psi_1,\psi_2,\psi_3\}$, and $\phi^*$ denotes a lower bound of $\phi_k$.
\end{theorem}
\begin{proof}
	The proof is similar to that of Theorem \ref{AdsCR2}. We sketch it here, and refer to \ref{THE3} for details.

    Step 1. We establish the upper bound for the estimation variance $\mathbb{E}\|v_t^s-\nabla f(x_t^s)\|^2$ (see also Lemma \ref{BEV3}).
    \[\mathbb{E}\|v_k-\nabla f(x_k)\|^2\leq \frac{L^2}{b}\sum_{i=(n_k-1)q}^{k-1} \mathbb{E}\|x_{i+1}-x_i\|^2+\frac{I_{(N_k<n)}}{N_k}\sigma^2.\]
	
	Step 2. By the optimality condition of each subproblem in Algorithm \ref{alg3}, we have the upper bound of $\|\lambda_{k+1}-\lambda_k\|^2$ (see also Lemma \ref{SM3bdual}),
	\[
	\begin{aligned}
			\|\lambda_{k+1}-\lambda_k\|^2&\leq \frac{5}{\varsigma_A}\|v_k-\nabla f(x_k)\|^2+\frac{5}{\varsigma_A}\|\nabla f(x_{k-1})-v_{k-1}\|^2\\
			&\quad+\frac{5\zeta_{\max}^2}{\varsigma_A\eta^2}\|x_{k+1}-x_k\|^2+\frac{5}{\varsigma_A}(\frac{\zeta_{\max}^2}{\eta^2}+L^2)\|x_k-x_{k-1}\|^2,
		\end{aligned}
	\]
	
	Step 3. By the upper bound of $\mathbb{E}\|\lambda_{k+1}-\lambda_k\|^2$, we get (see Lemma \ref{SM3MFD}),
	\[\frac{\chi}{K}\sum_{i=(n_k-1)q}^{k}\mathbb{E}\|x_{i+1}-x_i\|^2\leq\frac{\phi_0-\phi^*}{ K}+(\frac{1}{2}+\frac{10}{\beta\varsigma_A})(\frac{c_d}{c_\tau}+\frac{1}{c_\epsilon})\epsilon.\]
	
	Step 4. Combining the above inequality and the definition of $\epsilon$-stationary, we obtain the claimed result.
\end{proof}

The following corollary provides the $\mathcal{O}(\frac{1}{K})$ convergence rate and the oracle complexity of Algorithm \ref{alg3} under certain choices of parameters. The proof of the corollary is included in \ref{coro2}.
\begin{corollary}\label{AbsspiderOC}
	Under the setting of Theorem \ref{AdsCR3}, choose $b=q=\sqrt{n}$, $1\leq c_d\leq q$, $0<\eta<\frac{2\zeta_{\min}}{L^2+L+4}$,  $c_\tau=c_\epsilon=\frac{9(1+c_d)}{2}(4+\beta^2 \|B\|\|A\|)$ and
	\[
	\beta\geq\max\left\{\frac{20}{\varsigma_A},\sqrt{\frac{10}{3\varsigma_A}},\sqrt{\frac{3}{\|B\|\|A\|}(L^2+\frac{\zeta_{\max}^2}{\eta^2})},\frac{1}{\varsigma_A}\sqrt{20(1+2L^2+\frac{\zeta_{\max}^2}{\eta^2})}\right\}.
	\]
	Then, to find an $\epsilon$-approximate solution, the total number of iterations required by Algorithm \ref{alg3} is given by \[K=\frac{9(3+\beta^2\|B\|\|A\|)(\phi_0-\phi^*)}{\chi \epsilon}=\mathcal{O}(\epsilon^{-1}),\]
	and	the oracle complexity of Algorithm \ref{alg3} is given by
	\[\begin{aligned}
		&\underset{\rm complexity~of~AbsSPIDER-ADMM}{\underbrace{\sum_{k=0}^{K-1}\min\{\frac{c_\tau\sigma^2}{\frac{1}{q}\sum_{i=(n_k-1)q}^{k}\|x_{i+1}-x_i\|^2},c_\epsilon\sigma^2\epsilon^{-1},n\}+Kb}}\\
		\leq&\underset{\rm complexity~of~SPIDER-ADMM}
		{\underbrace{K\min\{c_\epsilon\sigma^2\epsilon^{-1},n\}+Kb=\mathcal{O}(n+n^{\frac{1}{2}}\epsilon^{-1})}}.
	\end{aligned}
	\]
\end{corollary}
This result shows that the worst-case oracle complexity of AbsSPIDER-ADMM is $\mathcal{O}(n+n^{\frac{1}{2}}\epsilon^{-1})$, which matches the best known order for SPIDER-ADMM \citep{HCH19}. Similarly to the argument at the end of Corollary \ref{AbssvrgOC}, our practical complexity can be much better than that of SPIDER-ADMM due to the batch size adaptation.

\section{Experiments}\label{sec5}
This section empirically evaluates SADMM with batch size adaptation on several tasks. All experiments were performed in MATLAB R2016b on a 64-bit laptop equipped with Intel i9-13900HX CPU and 32.0 GB RAM.
		
		We compare proposed AbsSADMM, AbsSVRG-ADMM and AbsSPIDER-ADMM with the corresponding vanilla methods: SADMM \citep{HC18}, SVRG-ADMM \citep{HCL16,ZK16}, and SPIDER-ADMM \citep{HCH19}.
		%The batch sizes for SADMM and SVRG/SPIDER-ADMM are set to $c_\epsilon\sigma^2\epsilon^{-1}$ and $\min\{c_\epsilon\sigma^2\epsilon^{-1},n\}$ respectively.
		We consider four publicly available datasets\footnote{\href{https://www.csie.ntu.edu.tw/~cjlin/libsvmtools/datasets/}
			{https://www.csie.ntu.edu.tw/$\sim$cjlin/libsvmtools/datasets/}.} as described in Table \ref{sample-table}. For each dataset, we take half of the samples as training data and the rest as test data. The same initialization $x_0=0$ is used for all methods in two experiments.
		
		\begin{table}[h]
			\caption{Summary of datasets and regularization parameters used for Problem \eqref{FLR}.}\label{sample-table}
			\centering
			\begin{tabular}{cccc}
				\toprule
				Datasets & Samples & Features & $\ell$\\
				\midrule
				phishing & 11055 & 68 & $5.5 \times 10^{-3}$\\
				a9a & 32,561 & 123 & $3 \times 10^{-6}$ \\
				w8a & 49749 & 300 & $10^{-4}$\\
				ijcnn1 & 49990 & 22 &  $5 \times 10^{-5}$ \\
				\bottomrule
			\end{tabular}
		\end{table}

\subsection{Fused logistic regression}
		Given a set of training samples $\{(a_i,b_i)\}_{i=1}^n$ with $a_i\in\mathbb{R}^m$ and $b_i\in\{-1,1\}$, the fused logistic regression \citep{LQZFZ18} aims to solve the following problem:
		\begin{equation}\label{OFLR}
			\min_{x} \frac{1}{n}\sum_{i=1}^{n}f_i(x)+\ell\|Lx\|_1,%+\tau_2\|x\|_1,
		\end{equation}
		where $f_i(x)=\log(1+\exp(-b_i a_i^\top x))$ is the logistic loss and the regularization parameters $\ell>0$. $L$ is a matrix with all ones on the diagonal, negative ones on the super-diagonal and zeros elsewhere. Problem \eqref{OFLR} can be expressed as \eqref{p} by adding an auxiliary variable $y = Lx$,
		\begin{equation}\label{FLR}
			\min_{x,y} \frac{1}{n}\sum_{i=1}^{n}f_i(x)+\ell\|y\|_1,~Lx-y=0.
		\end{equation}
		
		\begin{table}[htbp]
			\caption{The parameter settings used for \eqref{FLR}. All methods take $\tau_0=100$ and $T=5$. }\label{parameter-table}
			\centering
			\begin{tabular}{ccccc}
				\toprule
				Datasets & Methods & $\beta$ & $\eta$ & $c_\tau$ \\
				\midrule
				phishing & AbsSADMM  & 100 & 0.8  & 1 \\ %3000
				&AbsSVRG-ADMM & 100 & 1  & 1 \\
				&AbsSPIDER-ADMM & 100 & 0.3  & 1 \\
				ijcnn1& AbsSADMM  & 10 & 0.25  & 3 \\ %5000
				&AbsSVRG-ADMM & 15 & 0.2 & 5 \\
				&AbsSPIDER-ADMM & 5 & 0.4  & 6\\
				w8a& AbsSADMM  & 2 & 0.75  & 10\\ %1000
				&AbsSVRG-ADMM & 2 & 0.25  & 1 \\
				&AbsSPIDER-ADMM & 2 & 0.5  & 1\\
				a9a& AbsSADMM  & 10 & 0.1  & 1 \\ %3000
				&AbsSVRG-ADMM & 12 & 0.5  & 2\\
				&AbsSPIDER-ADMM & 8 & 0.9  & 150\\
				\bottomrule
			\end{tabular}
		\end{table}
		
		The parameter settings are shown in Table \ref{parameter-table}. All parameters in the vanilla method are the same as in the corresponding adaptation. For SADMM, we fix $\epsilon=10^{-3}$ and $c_\epsilon=3,5,1,3$ to solve different datasets in turn. For SVRG/SPIDER-ADMM, the outer batch size is set to the full batch and while the inner batch  size $b$ is set to 500 for the phishing dataset and 100 for the other datasets. To reduce statistical variability, the experimental results are repeated 5 rounds. Figure \ref{FLRAbsS} shows that the adaptive batch size algorithms (AbsSADMM and AbsSVRG/SPIDER-ADMM) consistently outperform their static counterparts. In particular, the adaptive methods exhibit significantly better convergence speed and stability than the static methods in the early stages of iteration, which is consistent with our theoretical expectations. These improvements can be attributed to batch size adaptation, which enables the algorithms to adjust dynamically to varying optimization paths, resulting in more efficient updates and faster convergence.

		\begin{figure}[htbp]
			\centering
			\subfigure[phishing \label{FLR_phishing}]{
				\includegraphics[width=0.82\columnwidth]{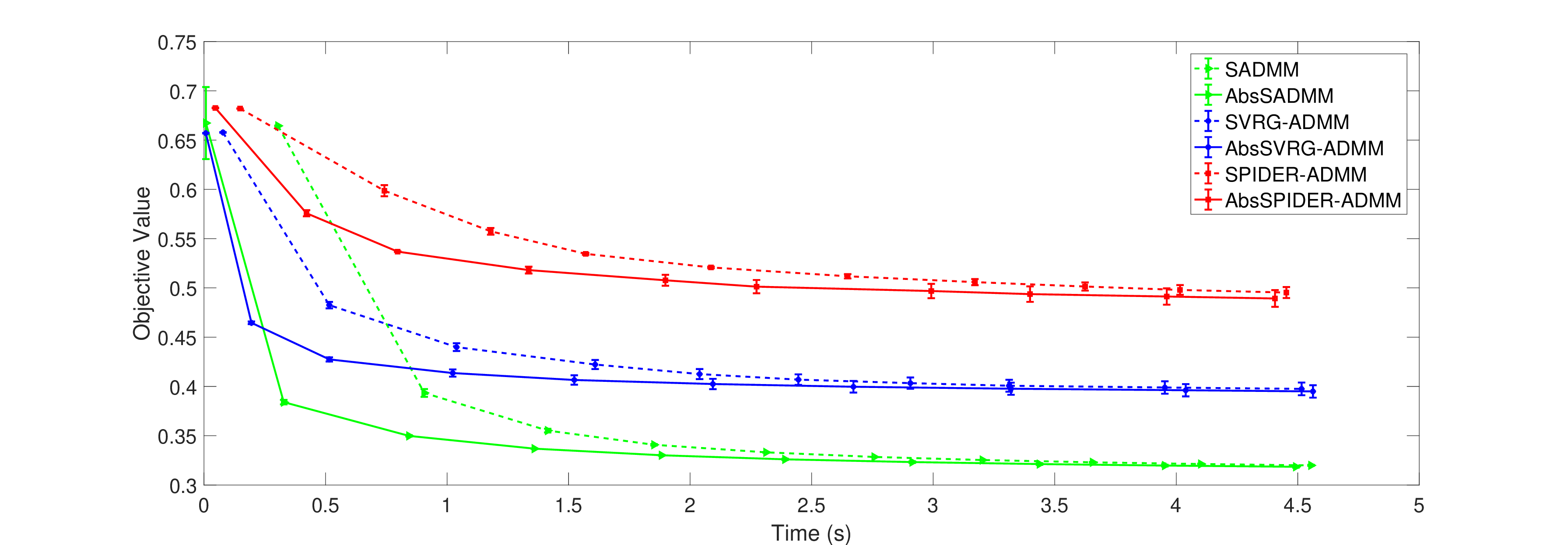}
			}
			\subfigure[ijcnn1 \label{FLR_ijcnn1}]{
				\includegraphics[width=0.82\columnwidth]{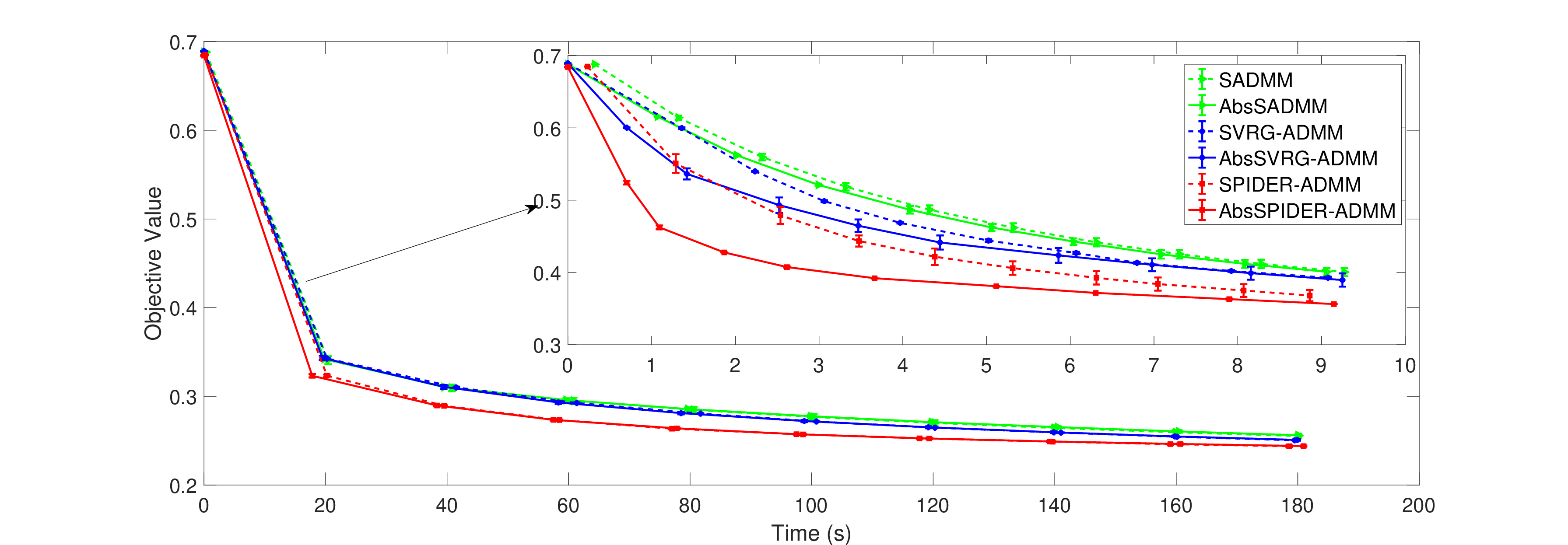}
			}
			\subfigure[w8a \label{FLR_w8a}]{
				\includegraphics[width=0.823\columnwidth]{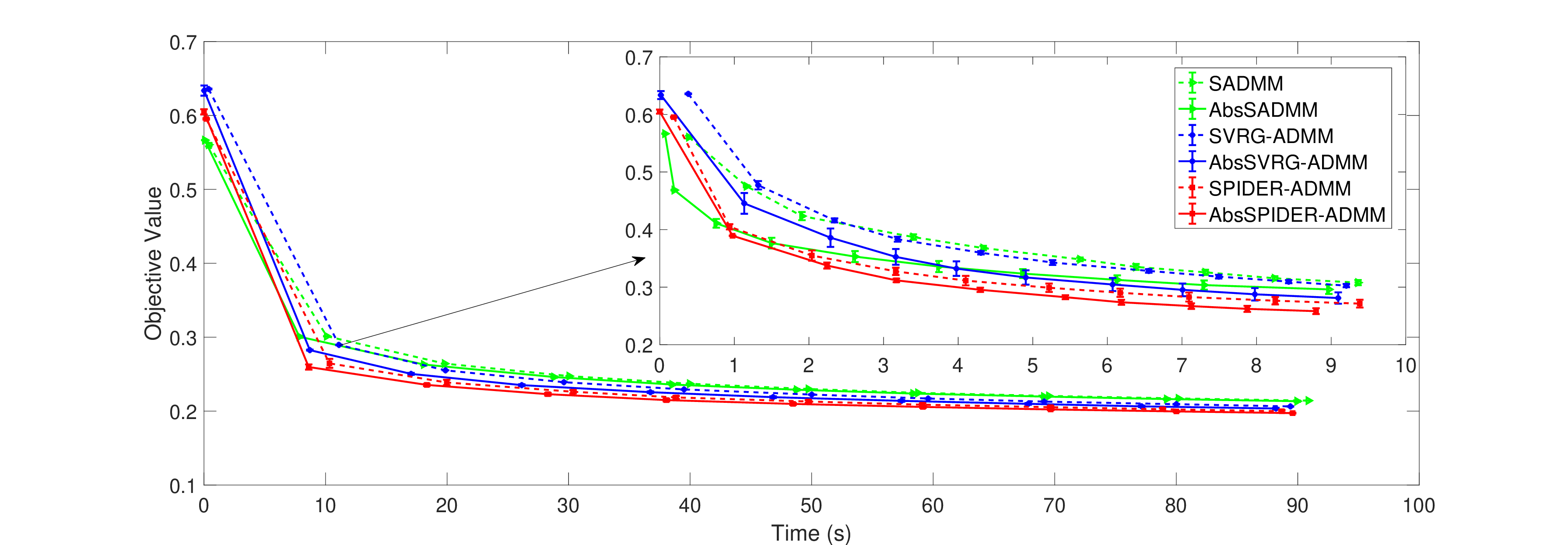}
			}
			\subfigure[a9a \label{FLR_a9a}]{
				\includegraphics[width=0.82\columnwidth]{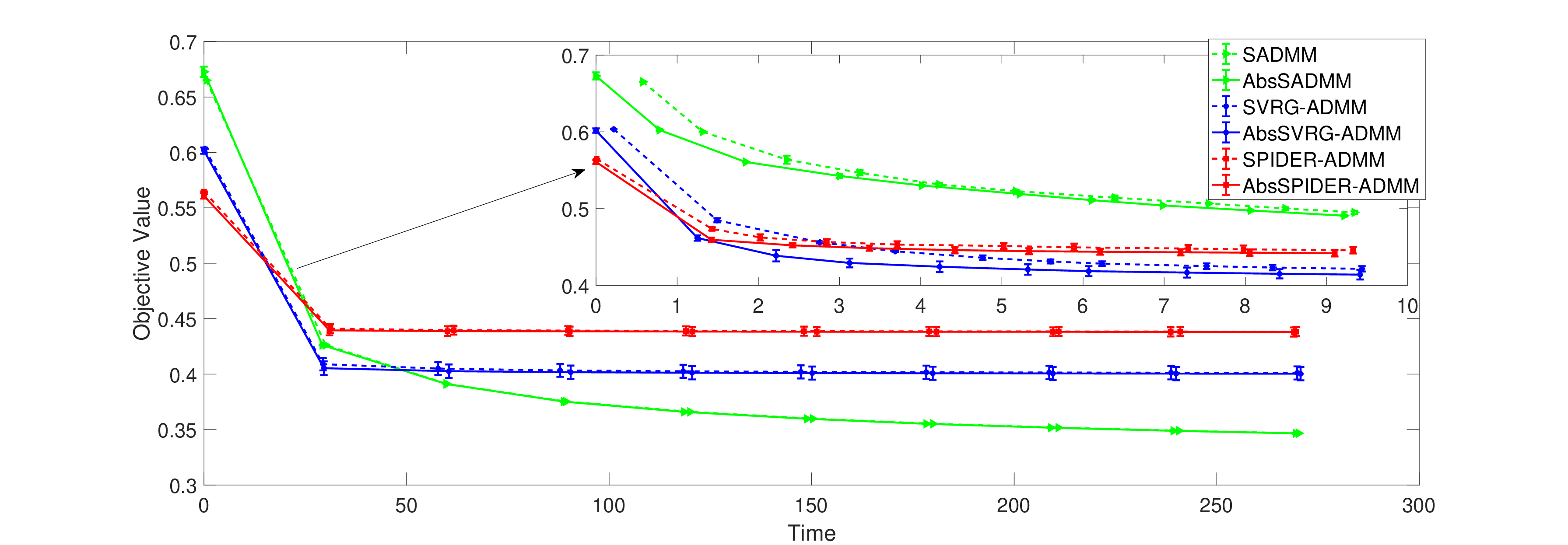}
			}
			\caption{	Comparison results of adaptive and vanilla SADMM for Problem \eqref{FLR}.}\label{FLRAbsS}
		\end{figure}
    
		\subsection{Graph-guided regularized lasso}
		Given a set of feature-label pair $\{(a_i,b_i)\}_{i=1}^n$ with $a_i\in\mathbb{R}^m$ and $b_i\in\{-1,1\}$, the graph-guided regularized lasso \citep{HCL16} aims to solve the following problem:
		\begin{equation}\label{OGGRL}
			\min_{x} \frac{1}{n}\sum_{i=1}^{n}f_i(x)+\ell_1\|Ax\|_1+\frac{\ell_2}{2}\|x\|^2,
		\end{equation}
		where $f_i(x)=\frac{1}{1+\exp(b_i a_i^\top x)}$ is the sigmoid loss and the regularization parameters $\ell_1,\ell_2>0$. $A=[G;I]$ is a matrix where $G$ is obtained by sparse inverse covariance matrix estimation \citep{FHT08,HSDR14}. Introducing the auxiliary variable $y = Ax$, \eqref{OGGRL} can be rewritten as \eqref{p},
		\begin{equation}\label{GGRL}
			\min_{x,y} \frac{1}{n}\sum_{i=1}^{n}(f_i(x)+\frac{\ell_2}{2}\|x\|^2)+\ell_1\|y\|_1,~ Ax-y=0.
		\end{equation}
		
		For (Abs)SADMM, we set $\beta=1.3$, $\eta=0.5$ and $c_\tau=1$. For (Abs)SVRG-ADMM, we set $\beta=\eta=1$, $T=8$, $b=500$, $c_\tau=1$ and $\tau_1=500$. For (Abs)SPIDER-ADMM, we set $\beta=1$, $\eta=2$, $q=5$, $b=500$, $c_\tau=1$ and $\tau_0=100$. We fix the parameter $\epsilon=10^{-4}$ and choose different $c_\epsilon$ in different cases, as indicated by the red line in Figures \ref{GGRLAbsS}-\ref{GGFLAbsSPIDER}.
		The experimental results show that our adaptive methods have a relatively faster convergence rate and lower overall complexity than non-adaptive methods.
		The batch size adaptation strategy enables these algorithms to effectively balance the trade-off between convergence rate and computational cost.

		\begin{figure}[htbp]
			\centering
			\subfigure[w8a \label{w8aAbsSGGRL}]{
				\includegraphics[width=0.475\columnwidth]{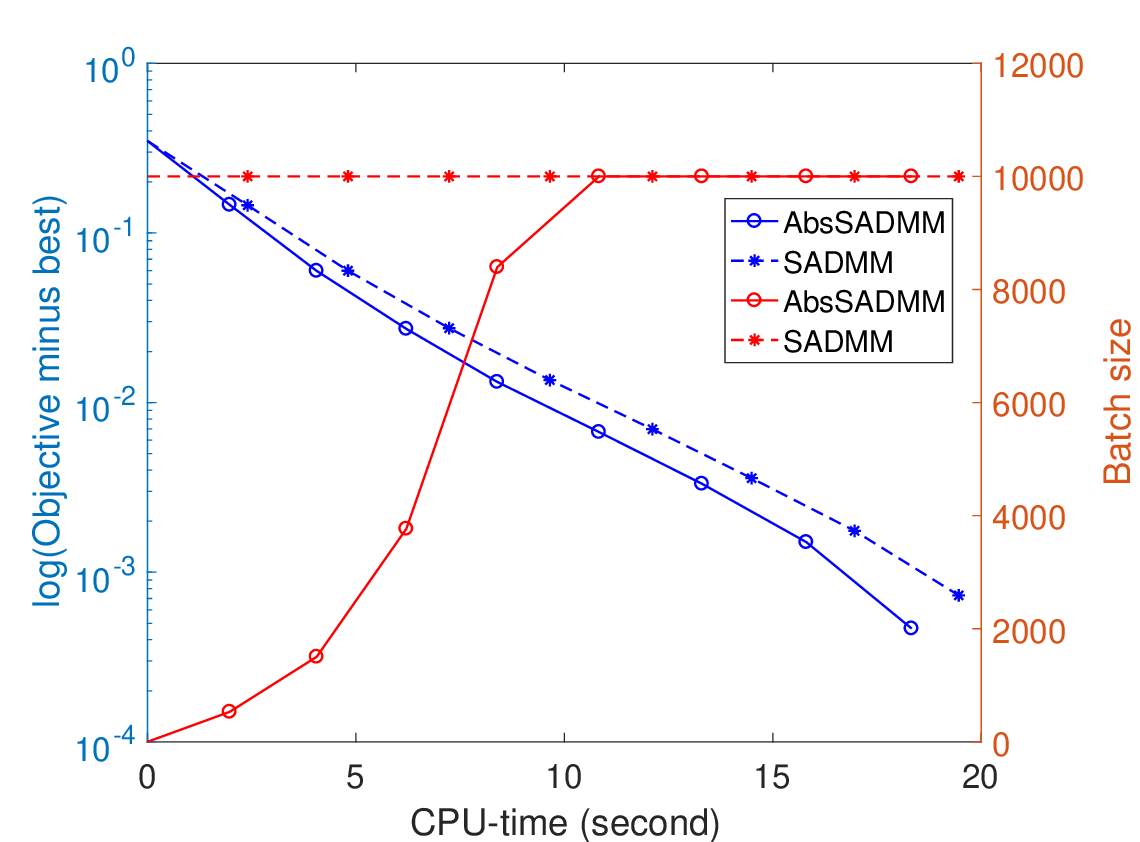}
			}
			\subfigure[ijcnn1 \label{ijcnn1AbsSGGRL}]{
				\includegraphics[width=0.475\columnwidth]{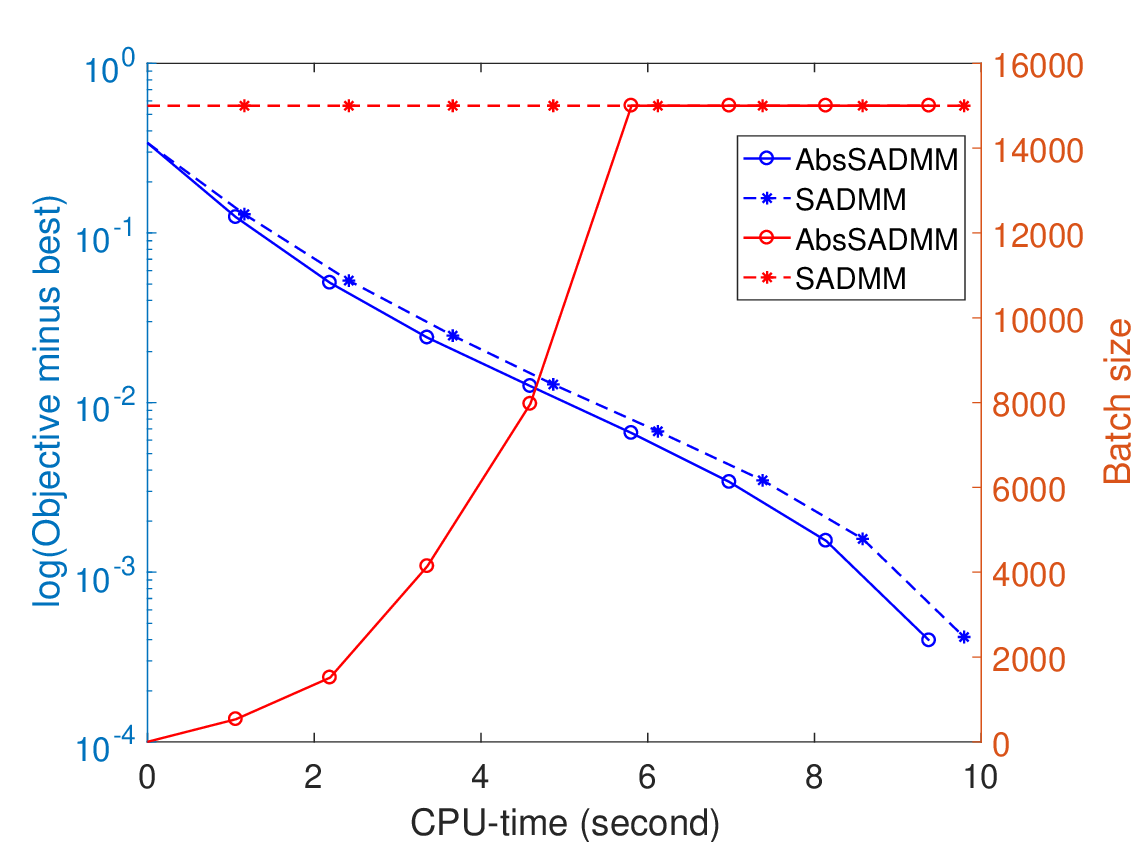}
			}
			\caption{Comparison results of (Abs)SADMM for \eqref{GGRL}. $\ell_1=\ell_2=10^{-2}$ for w8a and $\ell_1=10^{-5},\ell_2=10^{-2}$ for ijcnn1.}\label{GGRLAbsS}
		\end{figure}
		
		\begin{figure}[htbp]
			\centering
			\subfigure[a9a \label{GGFLASVRGa9a}]{
				\includegraphics[width=0.475\columnwidth]{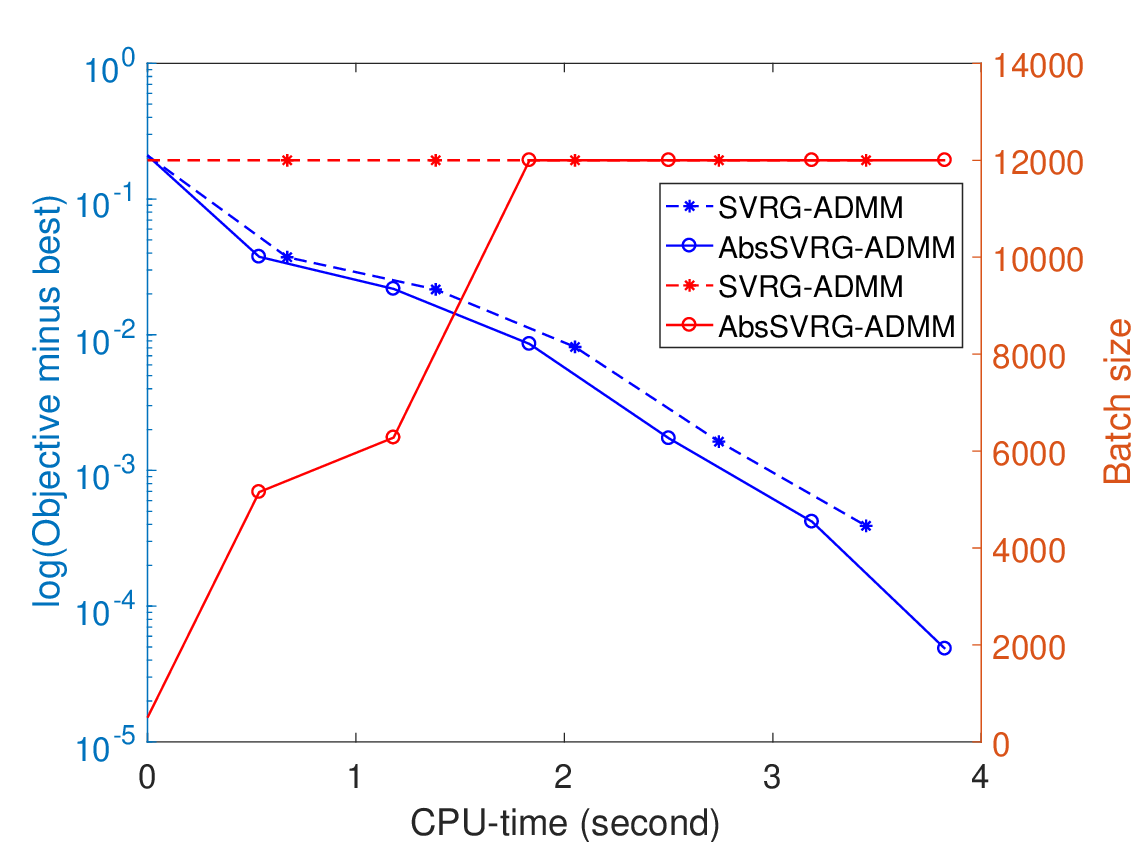}
			}
			\subfigure[ijcnn1 \label{GGFLASVRGijcnn1}]{
				\includegraphics[width=0.475\columnwidth]{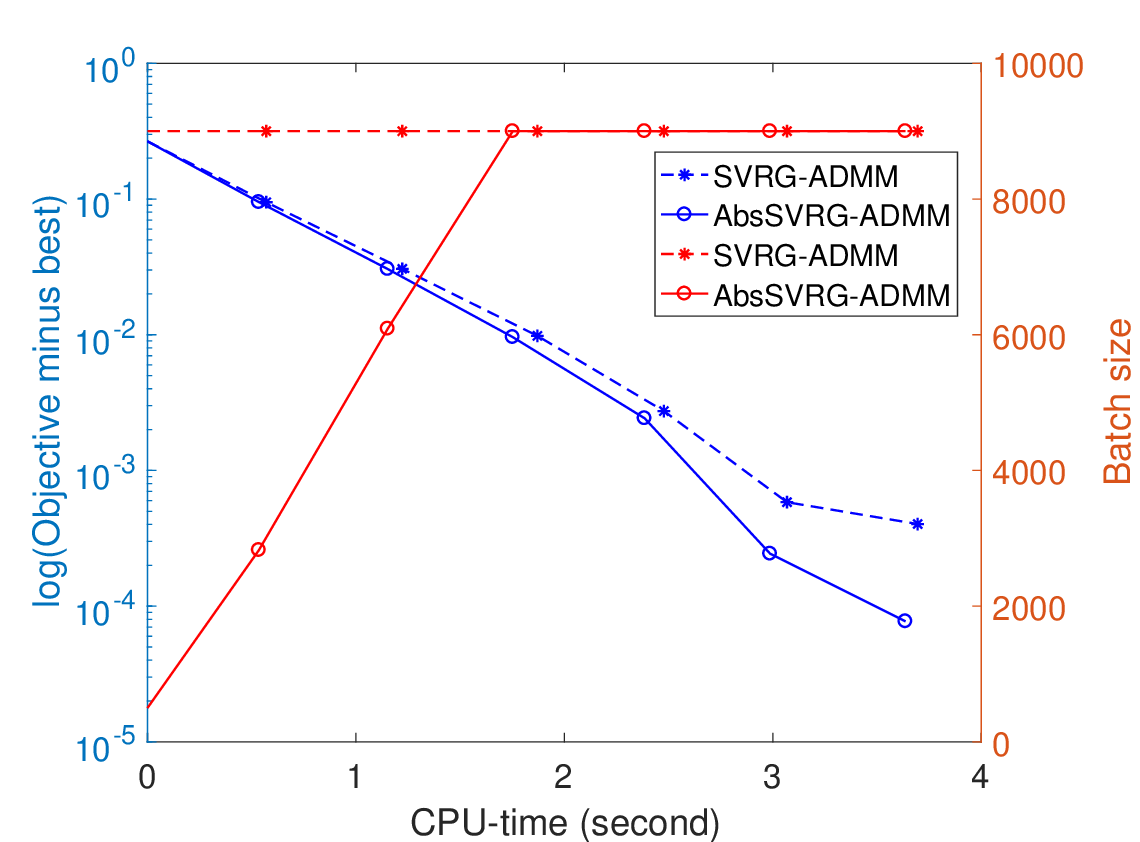}
			}
			\caption{Comparison results of (Abs)SVRG-ADMM for \eqref{GGRL}. $\ell_1=10^{-2},\ell_2=3 \times 10^{-5}$ for a9a and $\ell_1=1.5 \times 10^{-2},\ell_2=6\times 10^{-3}$ for ijcnn1.}\label{GGRLAbsSVRG}
		\end{figure}
		
		\begin{figure}[htbp]
			\centering
			\subfigure[phishing \label{GGRLASPIDERphishing}]{
				\includegraphics[width=0.475\columnwidth]{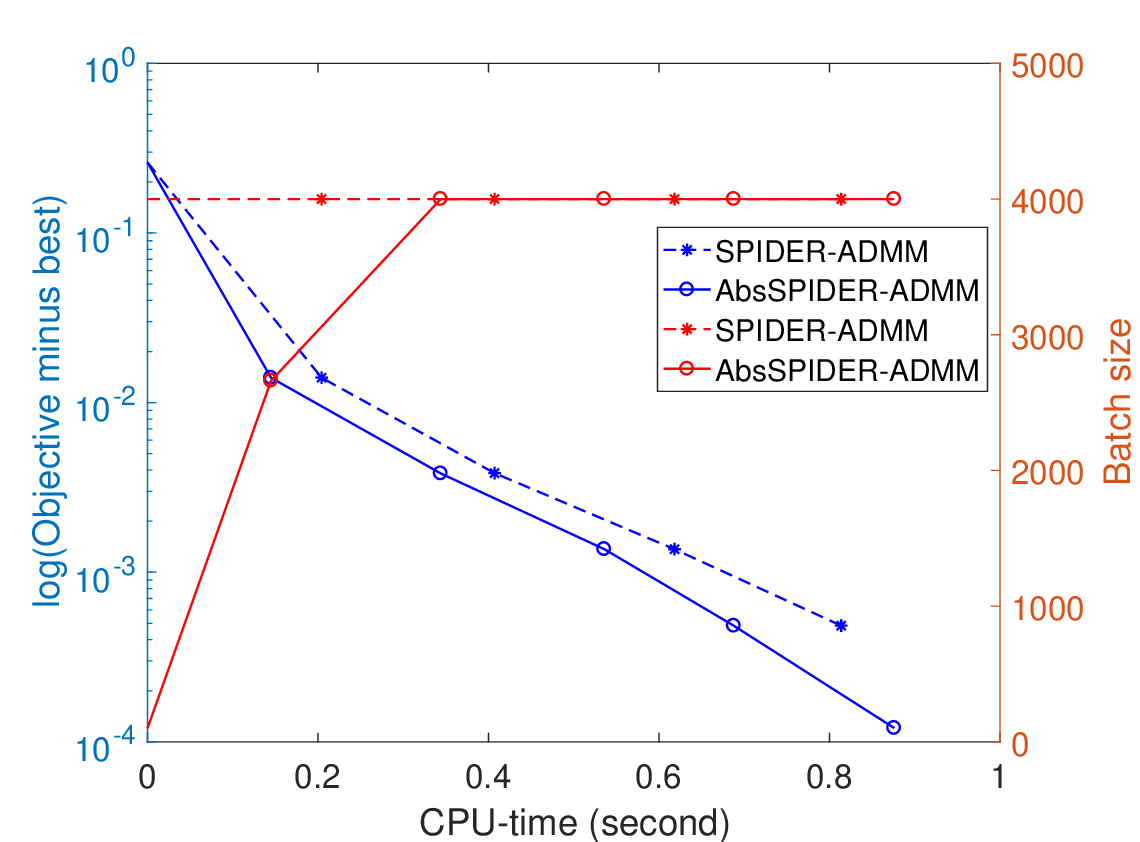}
			}
			\subfigure[w8a \label{GGRLASPIDERw8a}]{
				\includegraphics[width=0.475\columnwidth]{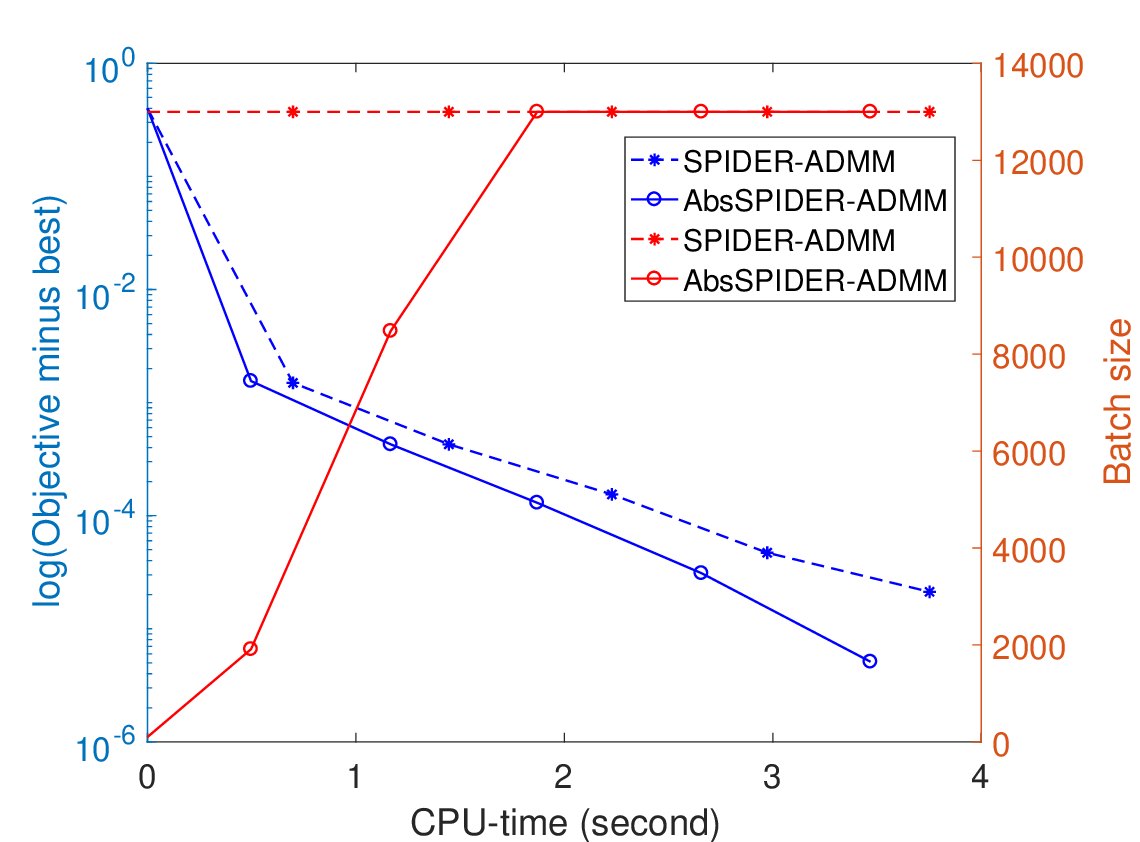}
			}
			\caption{Comparison results of (Abs)SPIDER-ADMM for \eqref{GGRL}. $\ell_1=1.1 \times 10^{-4},\ell_2=10^{-3}$ for phishing and $\ell_1=10^{-5},\ell_2=10^{-2}$ for w8a.}\label{GGFLAbsSPIDER}
		\end{figure}

	\section{Conclusions}\label{sec6}
	This work introduces a batch size adaptation scheme for stochastic ADMM and its variance-reduced extensions to solve nonconvex, nonsmooth optimization problems. The proposed framework adjusts mini-batch sizes based on differences in optimization progress over iterations. It begins with small batch sizes during the initial iterations to accelerate convergence. Then, it increases the batch size progressively as the iterations stabilize to reduce gradient variance and improve accuracy. We developed a simple convergence analysis that explicitly accounts for iteration dependencies in batch updates. Theoretical guarantees and empirical evaluations demonstrate that our adaptive methods outperform static batch algorithms in terms of computational efficiency.
	%For future work, we plan to explore accelerated nonconvex SADMM with batch size adaptation, incorporating techniques such as momentum or Anderson acceleration.

    \section*{Acknowledgments}
		This work was supported by the National Natural Science Foundation of China (12401419,12471401).

%% The Appendices part is started with the command \appendix;
%% appendix sections are then done as normal sections
\appendix
\section{Convergence Analysis for AbsSADMM}\label{app1}
\begin{lemma}[Upper-bound $\mathbb{E}\|\lambda_{k+1}-\lambda_k\|^2$]\label{SM1bdual}
			Under Assumptions \ref{nograb}-\ref{fucor} and given the sequence $\{w_k\}$ from Algorithm \ref{alg1}, it holds that
			\[\begin{aligned}
				\mathbb{E}\|\lambda_{k+1}-\lambda_k\|^2 \leq & \frac{5}{\varsigma_A}\mathbb{E}\|\nabla f_{\mathcal{I}_k}(x_k)-\nabla f(x_k)\|^2+\frac{5}{\varsigma_A}\mathbb{E}\|\nabla f(x_{k-1})-\nabla f_{\mathcal{I}_{k-1}}(x_{k-1})\|^2\\
				&+\frac{5\zeta_{\max}^2}{\varsigma_A\eta^2}\mathbb{E}\|x_{k+1}-x_k\|^2+\frac{5}{\varsigma_A}(\frac{\zeta_{\max}^2}{\eta^2}+L^2)\mathbb{E}\|x_k-x_{k-1}\|^2.
			\end{aligned}\]
		\end{lemma}
		\begin{proof}
			Using the optimal condition of the step 5 in Algorithm \ref{alg1}, we have
			\[\nabla f_{\mathcal{I}_k}(x_k)-A^\top\lambda_k+\beta A^\top(Ax_{k+1}+By_{k+1}-c)+\frac{G}{\eta}(x_{k+1}-x_k)=0.\]
			Using the step 6 of Algorithm \ref{alg1}, we have
			\begin{equation}\label{SM1Alamk1}
				A^\top\lambda_{k+1}=\nabla f_{\mathcal{I}_k}(x_k)+\frac{G}{\eta}(x_{k+1}-x_k).
			\end{equation}
			and
				\[\lambda_{k+1}=(A^\top)^+(\nabla f_{\mathcal{I}_k}(x_k)+\frac{G}{\eta}(x_{k+1}-x_k)).\]
				By Assumption \ref{fucor}, it has $(A^\top)^+=A(A^\top A)^{-1}$ and $\zeta_{\max}(((A^\top)^+)^\top (A^\top)^+)=\zeta_{\max}((A^\top A)^{-1})=\frac{1}{\varsigma_A}$. Then, we have
			\begin{align}
				&\|\lambda_{k+1}-\lambda_k\|^2\nonumber\\
				=&\|(A^\top)^+(\nabla f_{\mathcal{I}_k}(x_k)+\frac{G}{\eta}(x_{k+1}-x_k)-\nabla f_{\mathcal{I}_{k-1}}(x_{k-1}) -\frac{G}{\eta}(x_k-x_{k-1}))\|^2\nonumber\\
				\leq & \frac{1}{\varsigma_A}\|\nabla f_{\mathcal{I}_k}(x_k)+\frac{G}{\eta}(x_{k+1}-x_k)-\nabla f_{\mathcal{I}_{k-1}}(x_{k-1}) -\frac{G}{\eta}(x_k-x_{k-1})\|^2\nonumber\\
				=& \frac{1}{\varsigma_A}\|\nabla f_{\mathcal{I}_k}(x_k)-\nabla f(x_k)+\nabla f(x_k)-\nabla f(x_{k-1})\nonumber\\
				&+\nabla f(x_{k-1})-\nabla f_{\mathcal{I}_{k-1}}(x_{k-1}) +\frac{G}{\eta}(x_{k+1}-x_k)-\frac{G}{\eta}(x_k-x_{k-1})\|^2\nonumber\\
				\leq &\frac{5}{\varsigma_A}(\|\nabla f_{\mathcal{I}_k}(x_k)-\nabla f(x_k)\|^2+\|\nabla f(x_k)-\nabla f(x_{k-1})\|^2\nonumber\\
				&+\|\nabla f(x_{k-1})-\nabla f_{\mathcal{I}_{k-1}}(x_{k-1})\|^2+\|\frac{G}{\eta}(x_{k+1}-x_k)\|^2+\|\frac{G}{\eta}(x_k-x_{k-1})\|^2) \nonumber\\
				\leq &\frac{5}{\varsigma_A}\|\nabla f_{\mathcal{I}_k}(x_k)-\nabla f(x_k)\|^2+\frac{5}{\varsigma_A}\|\nabla f(x_{k-1})-\nabla f_{\mathcal{I}_{k-1}}(x_{k-1})\|^2\nonumber\\
				&+\frac{5\zeta_{\max}^2}{\varsigma_A\eta^2}\|x_{k+1}-x_k\|^2+\frac{5}{\varsigma_A}(\frac{\zeta_{\max}^2}{\eta^2}+L^2)\|x_k-x_{k-1}\|^2,\label{SM1lamk1k}
			\end{align}
			where the last inequality holds by Assumption \ref{Lsm}.
			Taking expectation $\mathbb{E}(\cdot)$ over the above inequality, we get the result and complete the proof.
		\end{proof}

\begin{lemma}\label{SM1MFD}
			Under Assumptions \ref{nograb}-\ref{fucor}, suppose the sequence $\{w_k\}$ is generated by Algorithm \ref{alg1}, and define a merit function $\phi_k$ as follows:
			\[\phi_k=\mathbb{E}\left[\mathcal{L}_\beta(w_k)+\delta\|x_k-x_{k-1}\|^2+\gamma\|x_{k-1}-x_{k-2}\|^2\right],\]
			where $\delta=\frac{1}{2c_\tau}+\frac{5}{\beta  \varsigma_A}(\frac{2}{c_\tau}+\frac{\zeta_{\max}^2}{\eta^2}+L^2)$ and $\gamma=\frac{5}{c_\tau\beta \varsigma_A}$, then
			\begin{equation}\label{SM1mfd}
				\phi_{k+1}\leq\phi_k-\rho\mathbb{E}\|x_{k+1}-x_k\|^2+(\frac{1}{2}+\frac{10}{\beta\varsigma_A})\frac{\epsilon}{ c_\epsilon},
			\end{equation}
			where $\rho=\frac{\zeta_{\min}}{\eta}+\frac{\beta\varsigma_A}{2}-\frac{L+1}{2}-\frac{10\zeta_{\max}^2}{\beta \varsigma_A\eta^2}-\frac{1}{2c_\tau}-\frac{10}{c_\tau \beta \varsigma_A}-\frac{5L^2}{\beta \varsigma_A}$.
		\end{lemma}
		\begin{proof}
			From the definition of the ALF $\mathcal{L}_\beta$, it follows that
			\begin{align}
				&\mathcal{L}_\beta(x_{k+1},y_{k+1},\lambda_k)-\mathcal{L}_\beta(x_k,y_{k+1},\lambda_k)\nonumber\\
				= & f(x_{k+1})-f(x_k)-\langle\lambda_k, A(x_{k+1}-x_k)\rangle+\frac{\beta}{2}\|Ax_{k+1}+By_{k+1}-c\|^2-\frac{\beta}{2}\|Ax_k+By_{k+1}-c\|^2 \nonumber\\
				=& f(x_{k+1})-f(x_k)-\langle\lambda_k-\beta(Ax_{k+1}+By_{k+1}-c), A(x_{k+1}-x_k)\rangle-\frac{\beta}{2}\|A(x_{k+1}-x_k)\|^2\nonumber\\
				=& f(x_{k+1})-f(x_k)-\langle A^\top\lambda_{k+1}, x_{k+1}-x_k\rangle-\frac{\beta}{2}\|A(x_{k+1}-x_k)\|^2\nonumber\\
				\overset{(i)}{=}&f(x_{k+1})-f(x_k)-\langle \nabla f_{\mathcal{I}_k}(x_k)+\frac{G}{\eta}(x_{k+1}-x_k), x_{k+1}-x_k\rangle-\frac{\beta}{2}\|A(x_{k+1}-x_k)\|^2\nonumber\\
				\overset{(ii)}{\leq} &\langle \nabla f(x_k)-\nabla f_{\mathcal{I}_k}(x_k), x_{k+1}-x_k\rangle+\frac{L}{2}\|x_{k+1}-x_k\|^2-\frac{1}{\eta}\|x_{k+1}-x_k\|_G^2-\frac{\beta}{2}\|A(x_{k+1}-x_k)\|^2\nonumber\\
				\leq &\frac{1}{2}\|\nabla f(x_k)-\nabla f_{\mathcal{I}_k}(x_k)\|^2-(\frac{\zeta_{\min}}{\eta}+\frac{\beta\varsigma_A}{2}-\frac{L+1}{2})\|x_{k+1}-x_k\|^2,\label{SM1Lxk1k}
			\end{align}
			where the equality (i) holds by \eqref{SM1Alamk1} and the inequality (ii) holds by Assumption \ref{Lsm}.
			Since $y_{k+1}$ is a minimizer of step 4 in Algorithm \ref{alg1}, we have
			\begin{equation}\label{SM1Lyk1k}
				\mathcal{L}_\beta(x_k,y_{k+1},\lambda_k)
				\leq\mathcal{L}_\beta(x_k,y_k,\lambda_k).
			\end{equation}
			Using the step 6 in Algorithm \ref{alg1}, we have
			\begin{equation}\label{SM1Llamk1k}
				\mathcal{L}_\beta(x_{k+1},y_{k+1},\lambda_{k+1})-\mathcal{L}_\beta(x_{k+1},y_{k+1},\lambda_k)=\frac{1}{\beta}\|\lambda_{k+1}-\lambda_k\|^2.
			\end{equation}
			Combining \eqref{SM1Lxk1k}-\eqref{SM1Llamk1k}, we have
			\[\begin{aligned}
				\mathcal{L}_\beta(w_{k+1})\leq &\mathcal{L}_\beta(w_k)+\frac{1}{2}\|\nabla f(x_k)-\nabla f_{\mathcal{I}_k}(x_k)\|^2\\
				&-(\frac{\zeta_{\min}}{\eta}
				+\frac{\beta\varsigma_A}{2}-\frac{L+1}{2})\|x_{k+1}-x_k\|^2+\frac{1}{\beta}\|\lambda_{k+1}-\lambda_k\|^2.
			\end{aligned}\]
			Taking expectation $\mathbb{E}(\cdot)$ over the above inequality yields 
			\[\begin{aligned}
				\mathbb{E}\mathcal{L}_\beta(w_{k+1})\leq & \mathbb{E}\mathcal{L}_\beta(w_k)+\frac{1}{2}\mathbb{E}\|\nabla f(x_k)-\nabla f_{\mathcal{I}_k}(x_k)\|^2\\ &-(\frac{\zeta_{\min}}{\eta}+\frac{\beta\varsigma_A}{2}-\frac{L+1}{2})\mathbb{E}\|x_{k+1}-x_k\|^2+\frac{1}{\beta}\mathbb{E}\|\lambda_{k+1}-\lambda_k\|^2.
			\end{aligned}\]
			Combining the above inequality with Lemma \ref{SM1bdual}, we have
			\[\begin{aligned}
				&\mathbb{E}\mathcal{L}_\beta(w_{k+1})-\mathbb{E}\mathcal{L}_\beta(w_k)\\
				\leq & \frac{1}{2}\mathbb{E}\|\nabla f(x_k)-\nabla f_{\mathcal{I}_k}(x_k)\|^2-(\frac{\zeta_{\min}}{\eta}+\frac{\beta\varsigma_A}{2}-\frac{L+1}{2})\mathbb{E}\|x_{k+1}-x_k\|^2\\
				&+\frac{5}{\beta\varsigma_A}\mathbb{E}\|\nabla f_{\mathcal{I}_k}(x_k)-\nabla f(x_k)\|^2+\frac{5}{\beta\varsigma_A}\mathbb{E}\|\nabla f(x_{k-1})-\nabla f_{\mathcal{I}_{k-1}}(x_{k-1})\|^2\\
				&+\frac{5\zeta_{\max}^2}{\beta\varsigma_A\eta^2}\mathbb{E}\|x_{k+1}-x_k\|^2+\frac{5}{\beta\varsigma_A}(\frac{\zeta_{\max}^2}{\eta^2}+L^2)\mathbb{E}\|x_k-x_{k-1}\|^2.
			\end{aligned}\]
			Then, we upper-bound $\mathbb{E}\|\nabla f(x_k)-\nabla f_{\mathcal{I}_k}(x_k)\|^2$ in the above inequality through the following steps,
				\[\begin{aligned}
					&\mathbb{E}\|\nabla f(x_k)-\nabla f_{\mathcal{I}_k}(x_k)\|^2=\mathbb{E}\left\|\nabla f(x_k)-\frac{1}{\lceil M_k \rceil}\sum_{i\in \mathcal{I}_k} \nabla f_i(x_k)\right\|^2\\
					&\leq\mathbb{E}\frac{1}{M_k^2}\left\|\sum_{i\in \mathcal{I}_k}(\nabla f(x_k)-\nabla f_i(x_k))\right\|^2\\
					&= \mathbb{E}\frac{1}{M_k^2}\sum_{i\in \mathcal{I}_k}\sum_{j\in \mathcal{I}_k}\langle\nabla f(x_k)-\nabla f_i(x_k),\nabla f(x_k)-\nabla f_j(x_k)\rangle\\
					&= \mathbb{E}_{x_0,\ldots,x_k}\left(\mathbb{E}_k\frac{1}{M_k^{2}}\sum_{i\in \mathcal{I}_k}\sum_{j\in \mathcal{I}_k}\langle\nabla f(x_k)-\nabla f_i(x_k),\nabla f(x_k)-\nabla f_j(x_k)\rangle\right)\\
					&= \mathbb{E}_{x_0,\ldots,x_k}\frac{1}{M_k^{2}}\sum_{i\in \mathcal{I}_k}\sum_{j\in \mathcal{I}_k}\mathbb{E}_k\langle\nabla f(x_k)-\nabla f_{i}(x_k),\nabla f(x_k)-\nabla f_{j}(x_k)\rangle\\
					&\overset{(i)}{=}\mathbb{E}_{x_0,\ldots,x_k}\frac{1}{M_k^{2}}\sum_{i\in \mathcal{I}_k}\mathbb{E}_k\left\|\nabla f(x_k)-\nabla f_{i}(x_k)\right\|^{2} \overset{(ii)}{\leq} \mathbb{E}\frac{\sigma^{2}}{M_k},
				\end{aligned}\]
				where the inequality holds by $\lceil M_k \rceil \geq M_k$; $(i)$ follows from $\mathbb{E}_k\nabla f_{i}(x_k)=\nabla f(x_k)$, and $\mathbb{E}_k\langle\nabla f(x_k)-\nabla f_i(x_k),\nabla f(x_k)-\nabla f_{j}(x_k)\rangle=0$ for $i\neq j$; and $(ii)$ follows from Assumption \ref{nograb}.  
				Next, we derive an upper bound for $\mathbb{E}\frac{\sigma^{2}}{M_k}$. Following from the definition of $M_k$, we have
				\begin{equation}\label{Mk}
					\mathbb{E}\frac{\sigma^2}{M_k}\leq\mathbb{E}\left(\frac{\|x_k-x_{k-1}\|^2}{c_\tau}+\frac{\epsilon}{c_\epsilon} \right)\leq\frac{\mathbb{E}\|x_k-x_{k-1}\|^2}{c_\tau}+\frac{\epsilon}{c_\epsilon},
			\end{equation}
			Thus, we have
			\[\begin{aligned}
				&\mathbb{E}\mathcal{L}_\beta(w_{k+1})-\mathbb{E}\mathcal{L}_\beta(w_k)\\
				\leq &(\frac{1}{2}+\frac{10}{\beta\varsigma_A})\frac{\epsilon}{c_\epsilon}+(\frac{1}{2c_\tau}+\frac{5}{c_\tau\beta  \varsigma_A}+\frac{5}{\beta\varsigma_A}(\frac{\zeta_{\max}^2}{\eta^2}+L^2))\mathbb{E}\|x_k-x_{k-1}\|^2\\
				&+\frac{5}{c_\tau\beta \varsigma_A}\|x_{k-1}-x_{k-2}\|^2-(\frac{\zeta_{\min}}{\eta}+\frac{\beta\varsigma_A}{2}-\frac{L+1}{2}-\frac{5\zeta_{\max}^2}{\beta \varsigma_A\eta^2})\mathbb{E}\|x_{k+1}-x_k\|^2.
			\end{aligned}\]
			Define the following merit function,
			\[\phi_k=\mathbb{E}\left[\mathcal{L}_\beta(w_k)+\delta\|x_k-x_{k-1}\|^2+\gamma\|x_{k-1}-x_{k-2}\|^2\right],\]
			with $\delta=\frac{1}{2c_\tau}+\frac{5}{\beta  \varsigma_A}(\frac{2}{c_\tau}+\frac{\zeta_{\max}^2}{\eta^2}+L^2)$ and $\gamma=\frac{5}{c_\tau\beta \varsigma_A}$. Then we have \eqref{SM1mfd} and the proof is completed.
		\end{proof}

By combining with Lemmas \ref{SM1bdual} and \ref{SM1MFD}, we prove the convergence of AbsSADMM in Theorem \ref{AdsCR1}.
		\begin{proof}[\textbf{Proof of Theorem \ref{AdsCR1}}]\label{THE1}
			First, we prove that $\phi_k$ is bounded from below. 
			Since $A$ is a full column rank matrix, we have $(A^\top)^+ = A(A^\top A)^{-1}$ and $\zeta_{\max}(((A^\top)^+)^\top (A^\top)^+)=\zeta_{\max}((A^\top A)^{-1})=\frac{1}{\varsigma_A}$. Then, we have
			\[\begin{aligned}
				&\mathbb{E}[f(x_{k+1})+g(y_{k+1})-\lambda_{k+1}^\top (Ax_{k+1}+By_{k+1}-c)+\frac{\beta}{2}\|Ax_{k+1}+By_{k+1}-c\|^2]\\
				\overset{(i)}{=}&\mathbb{E}[f(x_{k+1})+g(y_{k+1})-\langle (A^\top)^+(\nabla f_{\mathcal{I}_k}(x_k)+\frac{G}{\eta}(x_{k+1}-x_k)), Ax_{k+1}+By_{k+1}-c\rangle\\
				&+\frac{\beta}{2}\|Ax_{k+1}+By_{k+1}-c\|^2]\\
				\overset{(ii)}{\geq}& \mathbb{E}[f(x_{k+1})+g(y_{k+1})-\frac{1}{2\beta}\|(A^\top)^+(\nabla f_{\mathcal{I}_k}(x_k)+\frac{G}{\eta}(x_{k+1}-x_k))\|^2\\
				&-\frac{\beta}{2}\|Ax_{k+1}+By_{k+1}-c\|^2+\frac{\beta}{2}\|Ax_{k+1}+By_{k+1}-c\|^2]\\
				\overset{(iii)}{\geq}&\mathbb{E}[f(x_{k+1})+g(y_{k+1})-\frac{1}{2\beta\varsigma_A}\|\nabla f_{\mathcal{I}_k}(x_k)+\frac{G}{\eta}(x_{k+1}-x_k)-\nabla f(x_k)+\nabla f(x_k)\|^2]\\
				\overset{(iv)}{\geq}&\mathbb{E}[f(x_{k+1})+g(y_{k+1})-\frac{3}{2\beta\varsigma_A}\|\nabla f_{\mathcal{I}_k}(x_k)-\nabla f(x_k)\|^2-\frac{3}{2\beta\varsigma_A}\mathbb{E}\|\nabla f(x_k)\|^2\\
				&-\frac{3\zeta_{\max}^2}{2\beta\eta^2\varsigma_A}\mathbb{E}\|x_{k+1}-x_k\|^2]\\
				\geq&\mathbb{E}[f(x_{k+1})+g(y_{k+1})-\frac{3(\sigma^2+\mu^2)}{2\beta\varsigma_A}-\frac{3\zeta_{\max}^2}{2\beta\eta^2\varsigma_A}\|x_{k+1}-x_k\|^2],
			\end{aligned}\]
			where $(i)$ holds by \eqref{SM1Alamk1}; $(ii)$ is obtained by applying $\langle a,b \rangle\leq \frac{1}{2\beta}\|a\|^2+\frac{\beta}{2}\|b\|^2$ to the term $\langle (A^\top)^+(\nabla f_{\mathcal{I}_k}(x_k)+\frac{G}{\eta}(x_{k+1}-x_k)), Ax_{k+1}+By_{k+1}-c\rangle$; $(iii)$ holds by $\zeta_{\max}(((A^\top)^+)^\top (A^\top)^+)=\frac{1}{\varsigma_A}$; and $(iv)$ follows from Assumptions \ref{nograb} and \ref{grab}. Using Assumption \ref{lb} and the definition of $\phi_k$, we have
			\begin{equation}\label{MFLB}
				\phi_{k+1}\geq f^*+g^*-\frac{3(\sigma^2+\mu^2)}{2\beta\varsigma_A}.
			\end{equation}
			It follows that the function $\phi_k$ is bounded from below. Let $\phi^*$ denotes a lower bound of sequence $\{\phi_k\}$. 
			Choose the parameters $\eta,\beta,c_\tau$ used in Algorithm \ref{alg1} such that
			\[\rho=\frac{\zeta_{\min}}{\eta}+\frac{\beta\varsigma_A}{2}-\frac{L+1}{2}-\frac{1}{2c_\tau}-\frac{10\zeta_{\max}^2}{\beta \varsigma_A\eta^2}-\frac{10}{c_\tau \beta \varsigma_A}-\frac{5L^2}{\beta \varsigma_A}>0.\]
			Telescoping inequality \eqref{SM1mfd} over $k$ from $0$ to $K$, we have
			\begin{equation}\label{SM1the}
				\frac{1}{K}\sum_{k=0}^{K-1}\mathbb{E}\|x_{k+1}-x_k\|^2\leq\frac{\phi_0-\phi^*}{\rho K}+(\frac{1}{2}+\frac{10}{\beta\varsigma_A})\frac{\epsilon}{\rho c_\epsilon}.
			\end{equation}
			By step 4 in Algorithm \ref{alg1}, there exists a sub-gradient $\mu \in \partial_g (y_{k+1})$ such that
				\begin{align}
					\mathbb{E}[{\rm dist}(B^\top \lambda_{k+1},\partial_g (y_{k+1}))^2]\leq & \mathbb{E}\|\mu-B^\top \lambda_{k+1}\|^2\nonumber\\
					=&\mathbb{E}\|B^\top\lambda_k-\beta B^\top (Ax_k+By_{k+1}-c)-B^\top \lambda_{k+1})\|^2\nonumber\\
					=&\mathbb{E}\|\beta B^\top A(x_{k+1}-x_k)\|^2\nonumber\\
					\leq&\beta^2 \|B\|\|A\|\mathbb{E}\|x_{k+1}-x_k\|^2\nonumber\\
					\leq & \beta^2 \|B\|\|A\|\theta_k.\label{SM1disty}
				\end{align}
				By \eqref{SM1Alamk1} and \eqref{Mk}, we have
				\begin{align}
					&\mathbb{E}\|A^\top \lambda_{k+1}-\nabla f(x_{k+1})\|^2\nonumber\\
					=&\mathbb{E}\|\nabla f_{\mathcal{I}_k}(x_k)+\frac{G}{\eta}(x_{k+1}-x_k)-\nabla f(x_k)+\nabla f(x_k)-\nabla f(x_{k+1})\|^2\nonumber\\
					\leq&\mathbb{E}\|\nabla f_{\mathcal{I}_k}(x_k)-\nabla f(x_k)\|+\mathbb{E}\|\nabla f(x_k)-\nabla f(x_{k+1})\|^2+\mathbb{E}\|\frac{G}{\eta}(x_{k+1}-x_k)\|^2\nonumber\\
					\leq& \mathbb{E}\frac{3\sigma^2}{M_k}+ 3(L^2+\frac{\zeta_{\max}^2}{\eta^2})\mathbb{E}\|x_{k+1}-x_k\|^2\nonumber\\
					\leq & \frac{3}{c_\tau}\mathbb{E}\|x_k-x_{k-1}\|^2+\frac{3}{c_\epsilon}\epsilon+ 3(L^2+\frac{\zeta_{\max}^2}{\eta^2})\mathbb{E}\|x_{k+1}-x_k\|^2\nonumber\\
					\leq&3(\frac{1}{c_\tau}+L^2+\frac{\zeta_{\max}^2}{\eta^2})\theta_k+\frac{3}{c_\epsilon}\epsilon.\label{SM1distx}
			\end{align}
			By the step 6 of Algorithm \ref{alg1} and Lemma \ref{SM1bdual}, we have
			\begin{align}
				&\mathbb{E}\|Ax_{k+1}+By_{k+1}-c\|^2=\frac{1}{\beta^2}\mathbb{E}\|\lambda_{k+1}-\lambda_k\|^2\nonumber\\
				\leq&\mathbb{E}\frac{5\sigma^2}{\beta^2 \varsigma_AM_k}+\mathbb{E}\frac{5\sigma^2}{\beta^2 \varsigma_AM_{k-1}}+\frac{5\zeta_{\max}^2}{\beta^2\varsigma_A\eta^2}\mathbb{E}\|x_{k+1}-x_k\|^2+\frac{5}{\beta^2 \varsigma_A}(\frac{\zeta_{\max}^2}{\eta^2}+L^2)\mathbb{E}\|x_k-x_{k-1}\|^2\nonumber\\
				\leq&\frac{10}{\beta^2 \varsigma_Ac_\epsilon}\epsilon+\frac{5\zeta_{\max}^2}{\beta^2\varsigma_A\eta^2}\mathbb{E}\|x_{k+1}-x_k\|^2+\frac{5}{\beta^2 \varsigma_A}(\frac{1}{c_\tau}+\frac{\zeta_{\max}^2}{\eta^2}+L^2)\mathbb{E}\|x_k-x_{k-1}\|^2\nonumber\\
				&+\frac{5}{\beta^2 \varsigma_A c_\tau}\mathbb{E}\|x_{k-1}-x_{k-2}\|^2\nonumber\\
				\leq & \frac{10}{\beta^2 \varsigma_A c_\epsilon}\epsilon+\frac{5}{\beta^2  \varsigma_A}(\frac{1}{c_\tau}+\frac{\zeta_{\max}^2}{\eta^2}+L^2)\theta_k.\label{SM1distlam}
			\end{align}
			Let
			\[\psi_1=\beta^2 \|B\|\|A\|,~\psi_2=3(\frac{1}{c_\tau}+L^2+\frac{\zeta_{\max}^2}{\eta^2}) ,~\psi_3=\frac{5}{\beta^2  \varsigma_A}(\frac{1}{c_\tau}+\frac{\zeta_{\max}^2}{\eta^2}+L^2).\]
			and define a useful variable $\theta_k=\mathbb{E}\|x_{k+1}-x_k\|^2+\mathbb{E}\|x_k-x_{k-1}\|^2+\mathbb{E}\|x_{k-1}-x_{k-2}\|^2$.
			Since
			\[\sum_{k=0}^{K-1}\|x_k-x_{k-1}\|^2+\|x_K-x_{K-1}\|^2=\sum_{k=0}^{K-1}\|x_{k+1}-x_k\|^2+\|x_0-x_{-1}\|^2.\]
			It follows from $x_0=x_{-1}$ that $\sum_{k=0}^{K-1}\|x_k-x_{k-1}\|^2\leq \sum_{k=0}^{K-1}\|x_{k+1}-x_k\|^2$. 
			By \eqref{SM1the}-\eqref{SM1distlam} and Definition \ref{stationary}, we have
			\[\begin{aligned}
				\min_{1\leq k \leq K}\mathbb{E}[{\rm dist}(0,\partial \mathcal{L}(w_k))^2]&\leq \frac{\psi}{K} \sum_{k=1}^{K-1}\theta_k+\frac{\epsilon}{c_\epsilon}\max\{\frac{10}{\beta^2 \varsigma_A},3\}\\
				&\leq\frac{3\psi(\phi_0-\phi^*)}{\rho K}+\frac{\epsilon}{c_\epsilon}(\frac{1}{\rho}+\frac{20}{\rho\beta\varsigma_A}+\max\{\frac{10}{\beta^2 \varsigma_A},3\}),
			\end{aligned}
			\]
			where $\psi=\max\{\psi_1,\psi_2,\psi_3\}$. Then, plugging $K=\frac{6\psi(\phi_0-\phi^*)}{\rho\epsilon}$ and $c_\epsilon=2\kappa$ with $\kappa=\frac{1}{\rho}+\frac{20}{\rho\beta\varsigma_A}+\max\{\frac{10}{\beta^2 \varsigma_A},3\}$ in the above inequality, we have
			\[\min_{1\leq k \leq K}\mathbb{E}[{\rm dist}(0,\partial \mathcal{L}(w_k))^2]\leq\frac{\epsilon}{2}+\frac{\epsilon}{2}\leq\epsilon.\]
			The proof is completed.
		\end{proof}

\section{Convergence Analysis for AbsSVRG-ADMM}\label{app2}
\begin{lemma}\label{SM2SG}
			Suppose that Assumptions \ref{nograb} and \ref{Lsm} hold. Algorithm \ref{alg2} generates the stochastic gradient $\{v_t^s\}$ satisfies 
			\[\mathbb{E}_{0,s}\|v_t^s-\nabla f(x_t^s)\|^2\leq \frac{L^2}{b} \mathbb{E}_{0,s}\|x_t^s-\tilde{x}^{s-1}\|^2+\frac{I_{(N_s<n)}}{N_s}\sigma^2,\]
			where $\mathbb{E}_{t,s}(\cdot)$ denotes $\mathbb{E}(\cdot|x_0^1,x_0^2,\dots,x_2^1,\dots,x_t^s)$, and $I_{(A)} = 1$ if the event $A$ occurs and 0 otherwise.
		\end{lemma}
		\begin{proof}
				Let $F_i=\nabla f_{i}(x_t^s)-\nabla f_{i}(\tilde{x}^{s-1})-\nabla f(x_t^s)+\nabla f(\tilde{x}^{s-1})$. Based on line 8 in Algorithm \ref{alg2}, taking the expectation $\mathbb{E}_{0,s}$, we have
				\[\begin{aligned}
					&\mathbb{E}_{0,s}\|v_t^s-\nabla f(x_t^s)\|^2\\
					&=\mathbb{E}_{0,s}\|\nabla f_{\mathcal{I}}(x_t^s)-\nabla f_{\mathcal{I}}(\tilde{x}^{s-1})-\nabla f(x_t^s)+\nabla f(\tilde{x}^{s-1})+g^s-\nabla f(\tilde{x}^{s-1})\|^2\\
					&=\mathbb{E}_{0,s}\|\nabla f_{\mathcal{I}}(x_t^s)-\nabla f_{\mathcal{I}}(\tilde{x}^{s-1})-\nabla f(x_t^s)+\nabla f(\tilde{x}^{s-1})\|^2+\mathbb{E}_{0,s}\|g^s-\nabla f(\tilde{x}^{s-1})\|^2\\
					&\quad+2\mathbb{E}_{0,s}\langle \nabla f_{\mathcal{I}}(x_t^s)-\nabla f_{\mathcal{I}}(\tilde{x}^{s-1})-\nabla f(x_t^s)+\nabla f(\tilde{x}^{s-1}),g^s-\nabla f(\tilde{x}^{s-1})\rangle\\
					&\overset{(i)}{=}\mathbb{E}_{0,s}\|\nabla f_{\mathcal{I}}(x_t^s)-\nabla f_{\mathcal{I}}(\tilde{x}^{s-1})-\nabla f(x_t^s)+\nabla f(\tilde{x}^{s-1})\|^2+\mathbb{E}_{0,s}\|g^s-\nabla f(\tilde{x}^{s-1})\|^2\\
					&=\frac{1}{b^2}\sum_{i\in\mathcal{I}}\mathbb{E}_{0,s}\|F_i\|^2+\frac{2}{b^2}\sum_{i<j,i,j\in\mathcal{I}}\mathbb{E}_{0,s}\langle F_i, F_j \rangle+\mathbb{E}_{0,s}\|g^s-\nabla f(\tilde{x}^{s-1})\|^2\\
					&\overset{(ii)}{=}\frac{1}{b}\mathbb{E}_{0,s}\|F_i\|^2+\mathbb{E}_{0,s}\|g^s-\nabla f(\tilde{x}^{s-1})\|^2\\
					&\overset{(iii)}{\leq}\frac{1}{b}\mathbb{E}_{0,s}\|\nabla f_{i}(x_t^s)-\nabla f_{i}(\tilde{x}^{s-1})\|^2+\mathbb{E}_{0,s}\|g^s-\nabla f(\tilde{x}^{s-1})\|^2\\
					&\overset{(iv)}{\leq}\frac{L^2}{b} \mathbb{E}_{0,s}\|x_t^s-\tilde{x}^{s-1}\|^2+\frac{I_{(N_s<n)}}{N_s}\sigma^2,
				\end{aligned}\]
				where $(i)$ follows from the fact that 
				\[\mathbb{E}_{x_1^s,\dots,x_t^s}\mathbb{E}_{t,s}\langle \nabla f_{\mathcal{I}}(x_t^s)-\nabla f_{\mathcal{I}}(\tilde{x}^{s-1})-\nabla f(x_t^s)+\nabla f(\tilde{x}^{s-1}),g^s-\nabla f(\tilde{x}^{s-1})\rangle=0;\]
				$(ii)$ follows from the fact that 
				\[\mathbb{E}_{0,s}\langle F_i, F_j \rangle=\mathbb{E}_{x_1^s,\dots,x_t^s}\mathbb{E}_{t,s}\langle F_i, F_j \rangle=\mathbb{E}_{x_1^s,\dots,x_t^s}\langle \mathbb{E}_{t,s}F_i, \mathbb{E}_{t,s}F_j \rangle=0;\]
				$(iii)$ follows from the fact that $\mathbb{E}\|x-\mathbb{E}x\|^2\leq \mathbb{E}\|x\|^2$; $(iv)$ follows by Assumptions \ref{nograb} and \ref{Lsm}, Lemma B.2 in \cite{LJCJ17} and the fact that $N_s\leq\lceil N_s\rceil$ is fixed given $x_0^1,\dots,x_0^s$.
		\end{proof}
\begin{lemma}\label{SM2bdual}
			Under Assumptions \ref{nograb}-\ref{fucor} and given the sequence $\{w_t^s\}$ from Algorithm \ref{alg2}, it holds that
			\[\begin{aligned}
				\mathbb{E}_{0,s}\|\lambda_{t+1}^s-\lambda_t^s\|^2\leq &\frac{5}{\varsigma_A}\mathbb{E}_{0,s}\|v_t^s-\nabla f(x_t^s)\|^2+\frac{5}{\varsigma_A}\mathbb{E}_{0,s}\|\nabla f(x_{t-1}^s)-v_{t-1}^s\|^2\\
				&+\frac{5\zeta_{\max}^2}{\varsigma_A\eta^2}\mathbb{E}_{0,s}\|x_{t+1}^s-x_t^s\|^2+\frac{5}{\varsigma_A}(\frac{\zeta_{\max}^2}{\eta^2}+L^2)\mathbb{E}_{0,s}\|x_t^s-x_{t-1}^s\|^2.
			\end{aligned}\]
		\end{lemma}
		\begin{proof}
			Using the optimal condition of the step 10 and the step 11 in Algorithm \ref{alg2}, similar to \eqref{SM1Alamk1}, we have
			\begin{equation}\label{SM2Alamk1}
				A^\top\lambda_{t+1}^s=v_t^s+\frac{G}{\eta}(x_{t+1}^s-x_t^s).
			\end{equation}
			and
				\[\lambda_{t+1}^s=(A^\top)^+(v_t^s+\frac{G}{\eta}(x_{t+1}^s-x_t^s)).\]
			By Assumption \ref{fucor}, similar to \eqref{SM1lamk1k}, we have
			\[\begin{aligned}
				\|\lambda_{t+1}^s-\lambda_t^s\|^2\leq &\frac{5}{\varsigma_A}\|v_t^s-\nabla f(x_t^s)\|^2+\frac{5}{\varsigma_A}\|\nabla f(x_{t-1}^s)-v_{t-1}^s\|^2\\
				&+\frac{5\zeta_{\max}^2}{\varsigma_A\eta^2}\|x_{t+1}^s-x_t^s\|^2+\frac{5}{\varsigma_A}(\frac{\zeta_{\max}^2}{\eta^2}+L^2)\|x_t^s-x_{t-1}^s\|^2
			\end{aligned}\]
			Taking expectation $\mathbb{E}_{0,s}(\cdot)$ over the above inequality, we get the result.
		\end{proof}

	\begin{lemma}\label{SM2MFD}
			Under Assumptions \ref{nograb}-\ref{fucor}, suppose the sequence $\{w_t^s\}$ is generated from Algorithm \ref{alg2}, and define a merit function $\phi_k$ as follows:
			\[
			\phi_t^s=\mathbb{E}[\mathcal{L}_\beta(w_t^s)+\frac{5}{\beta\varsigma_A}(\frac{\zeta_{\max}^2}{\eta^2}+L^2)\|x_t^s-x_{t-1}^s\|^2+\frac{5L^2}{b\beta\varsigma_A}\|x_{t-1}^s-\tilde{x}^{s-1}\|^2+\delta_t\|x_t^s-\tilde{x}^{s-1}\|^2],
			\]
			where the positive sequence $\{\delta_t\}$ satisfies $\delta_t=\frac{L^2}{2b}+\frac{10L^2}{b\beta\varsigma_A}+(1+\rho)\delta_{t+1}$
			with $\rho>0$. Let $\epsilon >0$ and $c_\tau,c_\epsilon>0$. Choose the parameters $\tau_1,c_\tau,\beta,\eta,b>0$ used in Algorithm \ref{alg2} such that $\tau_1\leq \epsilon S$ and $\Omega_t>0$. Then, we have
			\[
			\frac{1}{ST}\sum_{s=1}^{S}\sum_{t=0}^{T-1}(\Omega_t\mathbb{E}\|x_{t+1}^s-x_t^s\|^2+\frac{L^2}{2b}\|x_t^s-\tilde{x}^{s-1}\|^2)\leq \frac{\phi_0^1-\phi^*}{ST}+(\frac{1}{2}+\frac{10}{\beta\varsigma_A})(\frac{1}{c_\tau}+\frac{1}{c_\epsilon})\epsilon,
			\]
			where $\phi^*$ denotes a lower bound of $\phi_t^s$ and $\Omega_t=\frac{\zeta_{\min}}{\eta}+\frac{\beta\varsigma_A}{2}-\frac{1}{2c_\tau}-\frac{L+1}{2}-\frac{5}{\beta \varsigma_A}(L^2+\frac{2}{c_\tau}+\frac{2\zeta_{\max}^2}{\eta^2})-(1+\frac{1}{\rho})\delta_{t+1}$.
		\end{lemma}
		
		\begin{proof}
			From the definition of the ALF $\mathcal{L}_\beta$, it follows that
			\begin{align}
				&\mathcal{L}_\beta(x_{t+1}^s,y_{t+1}^s,\lambda_t^s)-\mathcal{L}_\beta(x_t^s,y_{t+1}^s,\lambda_t^s)\nonumber\\
				= & f(x_{t+1}^s)-f(x_t^s)-\langle\lambda_t^s, A(x_{t+1}^s-x_t^s)\rangle+\frac{\beta}{2}\|Ax_{t+1}^s+By_{t+1}^s-c\|^2-\frac{\beta}{2}\|Ax_t^s+By_{t+1}^s-c\|^2 \nonumber\\
				=& f(x_{t+1}^s)-f(x_t^s)-\langle\lambda_t^s-\beta(Ax_{t+1}^s+By_{t+1}^s-c), A(x_{t+1}^s-x_t^s)\rangle-\frac{\beta}{2}\|A(x_{t+1}^s-x_t^s)\|^2\nonumber\\
				=& f(x_{t+1}^s)-f(x_t^s)-\langle A^\top\lambda_{t+1}^s, x_{t+1}^s-x_t^s\rangle-\frac{\beta}{2}\|A(x_{t+1}^s-x_t^s)\|^2\nonumber\\
				\overset{(i)}{=}&f(x_{t+1}^s)-f(t_k)^s-\langle v_t^s+\frac{G}{\eta}(x_{t+1}^s-x_t^s), x_{t+1}^s-x_t^s\rangle-\frac{\beta}{2}\|A(x_{t+1}^s-x_t^s)\|^2\nonumber\\
				\overset{(ii)}{\leq} &\langle \nabla f(x_t^s)-v_t^s, x_{t+1}^s-x_t^s\rangle+\frac{L}{2}\|x_{t+1}^s-x_t^s\|^2-\frac{1}{\eta}\|x_{t+1}^s-x_t^s\|_G^2-\frac{\beta}{2}\|A(x_{t+1}^s-x_t^s)\|^2\nonumber\\
				\leq &\frac{1}{2}\|\nabla f(x_t^s)-v_t^s\|^2-(\frac{\zeta_{\min}}{\eta}+\frac{\beta\varsigma_A}{2}-\frac{L+1}{2})\|x_{t+1}^s-x_t^s\|^2,\label{SM2Lxk1k}
			\end{align}
			where $(i)$ holds by \eqref{SM2Alamk1} and $(ii)$ holds by Assumption \ref{Lsm}.
			Since $y_{k+1}$ is a minimizer of step 9 in Algorithm \ref{alg2}, we have
			\begin{equation}\label{SM2Lyk1k}
				\mathcal{L}_\beta(x_t^s,y_{t+1}^s,\lambda_t^s)
				\leq\mathcal{L}_\beta(x_t^s,y_t^s,\lambda_t^s).
			\end{equation}
			Using the step 11 in Algorithm \ref{alg2}, we have
			\begin{equation}\label{SM2Llamk1k}
				\mathcal{L}_\beta(x_{t+1}^s,y_{t+1}^s,\lambda_{t+1}^s)-\mathcal{L}_\beta(x_{t+1}^s,y_{t+1}^s,\lambda_t^s)=\frac{1}{\beta}\|\lambda_{t+1}^s-\lambda_t^s\|^2.
			\end{equation}
			Combining \eqref{SM2Lxk1k}-\eqref{SM2Llamk1k}, we have
			\[\begin{aligned}
				&\mathcal{L}_\beta(w_{t+1}^s)-\mathcal{L}_\beta(w_t^s)\\
				\leq & \frac{1}{2}\|\nabla f(x_t^s)-v_t^s\|^2-(\frac{\zeta_{\min}}{\eta}+\frac{\beta\varsigma_A}{2}-\frac{L+1}{2})\|x_{t+1}^s-x_t^s\|^2+\frac{1}{\beta}\|\lambda_{t+1}^s-\lambda_t^s\|^2.
			\end{aligned}\]
			Then, taking expectation $\mathbb{E}_{0,s}(\cdot)$ over the above inequality yields
			\[\begin{aligned}
				\mathbb{E}_{0,s}\mathcal{L}_\beta(w_{t+1}^s)\leq& \mathbb{E}_{0,s}\mathcal{L}_\beta(w_t^s)+\frac{1}{2}\mathbb{E}_{0,s}\|\nabla f(x_t^s)-v_t^s\|^2\\
				&-(\frac{\zeta_{\min}}{\eta}+\frac{\beta\varsigma_A}{2}-\frac{L+1}{2})\mathbb{E}_{0,s}\|x_{t+1}^s-x_t^s\|^2+\frac{1}{\beta}\mathbb{E}_{0,s}\|\lambda_{t+1}^s-\lambda_t^s\|^2.
			\end{aligned}\]
			Combining Lemmas \ref{SM2SG} and \ref{SM2bdual}, we have
			\begin{align}
				&\mathbb{E}_{0,s}\mathcal{L}_\beta(w_{t+1}^s)-\mathbb{E}_{0,s}\mathcal{L}_\beta(w_t^s)\nonumber\\
				\leq & (\frac{L^2}{2b}+\frac{5L^2}{b\beta\varsigma_A})\mathbb{E}_{0,s}\|x_t^s-\tilde{x}^{s-1}\|^2+\frac{\sigma^2 I_{(N_s<n)}}{N_s}(\frac{1}{2}+\frac{10}{\beta\varsigma_A})\nonumber\\
				&+\frac{5L^2}{b\beta\varsigma_A}\mathbb{E}_{0,s}\|x_{t-1}^s-\tilde{x}^{s-1}\|^2+\frac{5}{\beta\varsigma_A}(\frac{\zeta_{\max}^2}{\eta^2}+L^2)\mathbb{E}_{0,s}\|x_t^s-x_{t-1}^s\|^2\nonumber\\
				&-(\frac{\zeta_{\min}}{\eta}+\frac{\beta\varsigma_A}{2}-\frac{L+1}{2}-\frac{5\zeta_{\max}^2}{\beta \varsigma_A\eta^2})\mathbb{E}_{0,s}\|x_{t+1}^s-x_t^s\|^2.\label{SM2wst1st}
			\end{align}
			Then, we bound $\mathbb{E}_{0,s}\|x_{t+1}^s-\tilde{x}^{s-1}\|^2$ as
			\begin{align}
				&\mathbb{E}_{0,s}\|x_{t+1}^s-\tilde{x}^{s-1}\|^2=\mathbb{E}_{0,s}[\|x_{t+1}^s-x_t^s\|^2+2\langle x_{t+1}^s-x_t^s,x_t^s-\tilde{x}^{s-1}\rangle+\|x_t^s-\tilde{x}^{s-1}\|^2]\nonumber\\
				&\leq\mathbb{E}_{0,s}[\|x_{t+1}^s-x_t^s\|^2+2( \frac{1}{2\rho}\|x_{t+1}^s-x_t^s\|^2+\frac{\rho}{2}\|x_t^s-\tilde{x}^{s-1}\|^2)+\|x_t^s-\tilde{x}^{s-1}\|^2]\nonumber\\
				&=(1+\frac{1}{\rho})\mathbb{E}_{0,s}\|x_{t+1}^s-x_t^s\|^2+(1+\rho)\mathbb{E}_{0,s}\|x_t^s-\tilde{x}^{s-1}\|^2,\label{SM2xst1tx}
			\end{align}
			where the inequality is due to the Cauchy-Schwarz inequality with $\rho>0$. Recall that $N_s=\min\{c_\tau\sigma^2\tau_s^{-1},c_\epsilon\sigma^2\epsilon^{-1},n\}$ and $c_\tau,c_\epsilon>0$, it has
			\begin{equation}\label{SM2NS}
				\frac{\sigma^2 I_{(N_s<n)}}{N_s}\leq \frac{\sigma^2}{\min\{c_\tau\tau_s^{-1},c_\epsilon\epsilon^{-1}\}}=\max\left\{\frac{\tau_s}{c_\tau},\frac{\epsilon}{c_\epsilon}\right\}\leq\frac{\tau_s}{c_\tau}+\frac{\epsilon}{c_\epsilon}.
			\end{equation}
			
			Define the following merit function,
			\[
			\phi_t^s=\mathbb{E}[\mathcal{L}_\beta(w_t^s)+\frac{5}{\beta\varsigma_A}(\frac{\zeta_{\max}^2}{\eta^2}+L^2)\|x_t^s-x_{t-1}^s\|^2+\frac{5L^2}{b\beta\varsigma_A}\|x_{t-1}^s-\tilde{x}^{s-1}\|^2+\delta_t\|x_t^s-\tilde{x}^{s-1}\|^2],
			\]
			where $\delta_t=\frac{L^2}{b}+\frac{10L^2}{b\beta\varsigma_A}+(1+\rho)\delta_{t+1}$. Combining \eqref{SM2wst1st}-\eqref{SM2NS} and taking the expectation over $x_0^1,\dots,x_0^s$, we have
			\begin{equation}\label{SM2MF1}
				\phi_{t+1}^s-\phi_t^s
				\leq-\chi_t\mathbb{E}\|x_{t+1}^s-x_t^s\|^2-\frac{L^2}{2b}\mathbb{E}\|x_t^s-\tilde{x}^{s-1}\|^2+(\frac{1}{2}+\frac{10}{\beta\varsigma_A})(\frac{\tau_s}{c_\tau}+\frac{\epsilon}{c_\epsilon}).
			\end{equation}
			where $\chi_t=\frac{\zeta_{\min}}{\eta}+\frac{\beta\varsigma_A}{2}-\frac{L+1}{2}-\frac{10\zeta_{\max}^2}{\beta \varsigma_A\eta^2}-\frac{5L^2}{\beta \varsigma_A}-(1+\frac{1}{\rho})\delta_{t+1}$.
			
			Next, we prove the relationship between $\phi_1^{s+1}$ and $\phi_T^s$. Since $\tilde{x}^s=x_0^{s+1}=x^s_T$, we have
			\begin{equation}\label{SM2vs10}
				v_0^{s+1}=\nabla f_{\mathcal{I}}(x_0^{s+1})-\nabla f_{\mathcal{I}}(\tilde{x}^s)+\nabla f_{\mathcal{N}_{s+1}} (\tilde{x}^s)=\nabla f_{\mathcal{N}_{s+1}} (x_0^{s+1})=\nabla f_{\mathcal{N}_{s+1}} (x^s_T).
			\end{equation}
			By \eqref{SM2Alamk1} and $\lambda^{s+1}_0=\lambda^s_T$, similar to Lemma \ref{SM2bdual}, we obtain
			\[\begin{aligned}
				\|\lambda_1^{s+1}-\lambda_0^{s+1}\|^2
				&\leq \frac{1}{\varsigma_A}\|v_{T-1}^s+\frac{G}{\eta}(x_T^s-x_{T-1}^s)-v_0^{s+1}-\frac{G}{\eta}(x_1^{s+1}-x_0^{s+1})\|^2\\
				&= \frac{1}{\varsigma_A}\|v_{T-1}^s-\nabla f(x_{T-1}^s)+\nabla f(x_{T-1}^s)-\nabla f(x_T^s)\\
				& \quad+\nabla f(x_T^s)-\nabla f_{\mathcal{N}_{s+1}} (x^s_T)+\frac{G}{\eta}(x_T^s-x_{T-1}^s)-\frac{G}{\eta}(x_1^{s+1}-x_0^{s+1})\|^2\\
				&\leq \frac{5}{\varsigma_A}\|v_{T-1}^s-\nabla f(x_{T-1}^s)\|^2+\frac{5}{\varsigma_A}\|\nabla f(x_T^s)-\nabla f_{\mathcal{N}_{s+1}} (x^s_T)\|^2\\
				&\quad+\frac{5\zeta_{\max}^2}{\varsigma_A\eta^2}\|x_1^{s+1}-x_0^{s+1}\|^2+\frac{5}{\varsigma_A}(\frac{\zeta_{\max}^2}{\eta^2}+L^2)\|x_T^s-x_{T-1}^s\|^2,
			\end{aligned}\]
			Taking $\mathbb{E}_{0,s}(\cdot)$ over the above inequality yields,
			\begin{align}
				\mathbb{E}_{0,s}\|\lambda_1^{s+1}-\lambda_0^{s+1}\|^2 \leq& \frac{5L^2}{b \varsigma_A}\mathbb{E}_{0,s}\|x_{T-1}^s-\tilde{x}^{s-1}\|^2+\frac{5\sigma^2}{ \varsigma_A}(\frac{I_{(N_s<n)}}{N_s}+\frac{1}{N_{s+1}})\nonumber\\
				&+\frac{5\zeta_{\max}^2}{\varsigma_A\eta^2}\mathbb{E}_{0,s}\|x_1^{s+1}-x_0^{s+1}\|^2+\frac{5}{\varsigma_A}(\frac{\zeta_{\max}^2}{\eta^2}+L^2)\mathbb{E}_{0,s}\|x_T^s-x_{T-1}^s\|^2,\label{SM2lams110}
			\end{align}
			where the inequality follows by Lemma \ref{SM2SG}, Assumptions \ref{nograb} and \ref{Lsm}.
			Since $\tilde{x}^s=x_0^{s+1}=x^s_T$, $y^{s+1}_0=y^s_T$ and $\lambda^{s+1}_0=\lambda^s_T$, using \eqref{SM2Lyk1k}, we have
			\begin{equation}\label{SM2Lyk1ks}
				\mathbb{E}_{0,s}[\mathcal{L}_\beta(x_0^{s+1},y_1^{s+1},\lambda_0^{s+1})]
				\leq\mathbb{E}_{0,s}[\mathcal{L}_\beta(x_0^{s+1},y_0^{s+1},\lambda_0^{s+1})]=\mathbb{E}_{0,s}[\mathcal{L}_\beta(x_T^s,y_T^s,\lambda_T^s)].
			\end{equation}
			By \eqref{SM2Lxk1k}, \eqref{SM2vs10}, \eqref{SM2NS} and Assumption \ref{nograb}, we have
			\begin{align}
				&\mathbb{E}_{0,s}[\mathcal{L}_\beta(x_1^{s+1},y_1^{s+1},\lambda_0^{s+1})-\mathcal{L}_\beta(x_0^{s+1},y_1^{s+1},\lambda_0^{s+1})]\nonumber\\
				\overset{\eqref{SM2Lxk1k}}{\leq}&\frac{1}{2}\mathbb{E}_{0,s}\|\nabla f(x_0^{s+1})-v_0^{s+1}\|^2-(\frac{\zeta_{\min}}{\eta}+\frac{\beta\varsigma_A}{2}-\frac{L+1}{2})\mathbb{E}_{0,s}\|x_1^{s+1}-x_0^{s+1}\|^2\nonumber\\
				\overset{\eqref{SM2vs10}}{\leq}&\frac{\sigma^2}{2N_{s+1}}-(\frac{\zeta_{\min}}{\eta}+\frac{\beta\varsigma_A}{2}-\frac{L+1}{2})\mathbb{E}_{0,s}\|x_1^{s+1}-x_0^{s+1}\|^2\nonumber\\
				\overset{\eqref{SM2NS}}{\leq}&\frac{1}{2}(\frac{\tau_s}{c_\tau}+\frac{\epsilon}{c_\epsilon})-(\frac{\zeta_{\min}}{\eta}+\frac{\beta\varsigma_A}{2}-\frac{L+1}{2})\mathbb{E}_{0,s}\|x_1^{s+1}-x_0^{s+1}\|^2.\label{SM2Lxk1ks}
			\end{align}
			By \eqref{SM2Llamk1k},\eqref{SM2lams110} and \eqref{SM2NS}, we have
			\begin{align}
				&\mathbb{E}_{0,s}[\mathcal{L}_\beta(x_1^{s+1},y_1^{s+1},\lambda_1^{s+1})-\mathcal{L}_\beta(x_1^{s+1},y_1^{s+1},\lambda_0^{s+1})]\overset{\eqref{SM2Llamk1k}}{=}\frac{1}{\beta}\mathbb{E}_{0,s}\|\lambda_1^{s+1}-\lambda_0^{s+1}\|^2\nonumber\\
				&\overset{\eqref{SM2lams110}}{\leq} \frac{5\zeta_{\max}^2}{\beta\varsigma_A\eta^2}\mathbb{E}_{0,s}\|x_1^{s+1}-x_0^{s+1}\|^2+\frac{5}{\varsigma_A}(\frac{\zeta_{\max}^2}{\eta^2}+L^2)\mathbb{E}_{0,s}\|x_T^s-x_{T-1}^s\|^2\nonumber\\
				&\quad+\frac{5L^2}{b\beta \varsigma_A}\mathbb{E}_{0,s}\|x_{T-1}^s-\tilde{x}^{s-1}\|^2+\frac{5\sigma^2}{ \beta\varsigma_A}(\frac{I_{(N_s<n)}}{N_s}+\frac{1}{N_{s+1}})\nonumber\\
				&\overset{\eqref{SM2NS}}{\leq} \frac{5\zeta_{\max}^2}{\beta\varsigma_A\eta^2}\mathbb{E}_{0,s}\|x_1^{s+1}-x_0^{s+1}\|^2+\frac{5}{\varsigma_A}(\frac{\zeta_{\max}^2}{\eta^2}+L^2)\mathbb{E}_{0,s}\|x_T^s-x_{T-1}^s\|^2\nonumber\\
				&\quad+\frac{5L^2}{b\beta \varsigma_A}\mathbb{E}_{0,s}\|x_{T-1}^s-\tilde{x}^{s-1}\|^2+\frac{5}{ \beta\varsigma_A}(\frac{\tau_s+\tau_{s+1}}{c_\tau}+\frac{2\epsilon}{c_\epsilon}).
				\label{SM2Llamk1ks}
			\end{align}
			Combining \eqref{SM2Lyk1ks}-\eqref{SM2Llamk1ks}, we have
			\[\begin{aligned}
				&\mathbb{E}_{0,s}[\mathcal{L}_\beta(w_1^{s+1})-\mathcal{L}_\beta(w_T^s)]\\
				\leq&\frac{5L^2}{b\beta \varsigma_A}\mathbb{E}_{0,s}\|x_{T-1}^s-\tilde{x}^{s-1}\|^2
				+\frac{5}{\beta\varsigma_A}(\frac{\zeta_{\max}^2}{\eta^2}+L^2)\mathbb{E}_{0,s}\|x_T^s-x_{T-1}^s\|^2+\frac{1}{2}(\frac{\tau_s}{c_\tau}+\frac{\epsilon}{c_\epsilon})\\
				&+\frac{5}{ \beta\varsigma_A}(\frac{\tau_s+\tau_{s+1}}{c_\tau}+\frac{2\epsilon}{c_\epsilon})-(\frac{\zeta_{\min}}{\eta}+\frac{\beta\varsigma_A}{2}-\frac{L+1}{2}-\frac{5\zeta_{\max}^2}{\beta\varsigma_A\eta^2})\mathbb{E}_{0,s}\|x_1^{s+1}-x_0^{s+1}\|^2.%\label{SM2ws11sT}
			\end{aligned}\]
			Taking expectation of the above inequality over $x_0^1,\dots,x_0^s$, we obtain
			\begin{align}
				\phi_1^{s+1}\leq&\phi_T^s-\chi_T\mathbb{E}\|x_1^{s+1}-x_0^{s+1}\|^2-\delta_T\mathbb{E}\|x_T^s-\tilde{x}^{s-1}\|^2\nonumber\\
				&+\frac{5}{ \beta\varsigma_A}(\frac{\tau_s+\tau_{s+1}}{c_\tau}+\frac{2\epsilon}{c_\epsilon})+\frac{1}{2}(\frac{\tau_s}{c_\tau}+\frac{\epsilon}{c_\epsilon})\nonumber\\
				\leq&\phi_T^s-\chi_T\mathbb{E}\|x_1^{s+1}-x_0^{s+1}\|^2-\frac{L^2}{2b}\|x_T^s-\tilde{x}^{s-1}\|^2\nonumber\\
				&+\frac{5}{ \beta\varsigma_A}(\frac{\tau_s+\tau_{s+1}}{c_\tau}+\frac{2\epsilon}{c_\epsilon})+\frac{1}{2}(\frac{\tau_s}{c_\tau}+\frac{\epsilon}{c_\epsilon}).\label{SM2MF2}
			\end{align}
			where $\delta_T=\frac{L^2}{b}+\frac{10L^2}{b\beta\varsigma_A}\geq\frac{L^2}{2b}$, $\chi_T=\frac{\zeta_{\min}}{\eta}+\frac{\beta\varsigma_A}{2}-\frac{L+1}{2}-\frac{10\zeta_{\max}^2}{\beta\varsigma_A\eta^2}-\frac{5L^2}{\beta\varsigma_A}-\delta_1$.
			Telescoping \eqref{SM2MF1} and \eqref{SM2MF2} over $t$ from 0 to $T-1$ and over $s$ from 1 to $S$, we have
			\[\begin{aligned}
				&\sum_{s=1}^{S}\sum_{t=0}^{T-1}(\phi_{t+1}^s-\phi_t^s)\\
				\leq&-\sum_{s=1}^{S}\sum_{t=0}^{T-1}\chi_t\mathbb{E}\|x_{t+1}^s-x_t^s\|^2-\frac{L^2}{2b}\sum_{s=1}^{S}\sum_{t=0}^{T-1}\mathbb{E}\|x_t^s-\tilde{x}^{s-1}\|^2+(\frac{1}{2}+\frac{10}{\beta\varsigma_A})\sum_{s=1}^{S}\sum_{t=0}^{T-1}(\frac{\tau_s}{c_\tau}+\frac{\epsilon}{c_\epsilon})\\
			\end{aligned}\]
			\[\begin{aligned}
				\leq&-\sum_{s=1}^{S}\sum_{t=0}^{T-1}\chi_t\mathbb{E}\|x_{t+1}^s-x_t^s\|^2-\frac{L^2}{2b}\sum_{s=1}^{S}\sum_{t=0}^{T-1}\mathbb{E}\|x_t^s-\tilde{x}^{s-1}\|^2\\
				&+(\frac{1}{2c_\tau}+\frac{10}{c_\tau\beta\varsigma_A})\sum_{s=2}^{S}\sum_{t=0}^{T-1}\mathbb{E}\|x_{t+1}^{s-1}-x_t^{s-1}\|^2+(\frac{1}{2}+\frac{10}{\beta\varsigma_A})\frac{T\tau_1}{c_\tau}+(\frac{1}{2}+\frac{10}{\beta\varsigma_A})\frac{ST\epsilon}{c_\epsilon}\\
				\leq&-\sum_{s=1}^{S}\sum_{t=0}^{T-1}\chi_t\mathbb{E}\|x_{t+1}^s-x_t^s\|^2-\frac{L^2}{2b}\sum_{s=1}^{S}\sum_{t=0}^{T-1}\mathbb{E}\|x_t^s-\tilde{x}^{s-1}\|^2\\
				&+(\frac{1}{2c_\tau}+\frac{10}{c_\tau\beta\varsigma_A})\sum_{s=1}^{S}\sum_{t=0}^{T-1}\mathbb{E}\|x_{t+1}^s-x_t^s\|^2+(\frac{1}{2}+\frac{10}{\beta\varsigma_A})(\frac{1}{c_\tau}+\frac{1}{c_\epsilon})ST\epsilon,
			\end{aligned}\]
			and
			\[\begin{aligned}
				&\sum_{s=1}^{S}\sum_{t=0}^{T-1}(\phi_1^{s+1}-\phi_T^s)\\
				\leq&-\sum_{s=1}^{S}\sum_{t=0}^{T-1}\chi_T\mathbb{E}\|x_1^{s+1}-x_0^{s+1}\|^2-\frac{L^2}{2b}\sum_{s=1}^{S}\sum_{t=0}^{T-1}\mathbb{E}\|x_T^s-\tilde{x}^{s-1}\|^2\\
				&+\sum_{s=1}^{S}\sum_{t=0}^{T-1}\frac{5}{ \beta\varsigma_A}(\frac{\tau_s+\tau_{s+1}}{c_\tau}+\frac{2\epsilon}{c_\epsilon})+\sum_{s=1}^{S}\sum_{t=0}^{T-1}\frac{1}{2}(\frac{\tau_s}{c_\tau}+\frac{\epsilon}{c_\epsilon})\\
				\leq&-\sum_{s=1}^{S}\sum_{t=0}^{T-1}\chi_T\mathbb{E}\|x_1^{s+1}-x_0^{s+1}\|^2-\frac{L^2}{2b}\sum_{s=1}^{S}\sum_{t=0}^{T-1}\mathbb{E}\|x_T^s-\tilde{x}^{s-1}\|^2\\
				&+\frac{5}{c_\tau\beta\varsigma_A}\sum_{s=2}^{S}\sum_{t=0}^{T-1}\mathbb{E}\|x_{t+1}^{s-1}-x_t^{s-1}\|^2+(\frac{1}{2}+\frac{10}{\beta\varsigma_A})\frac{T\tau_1}{c_\tau}\\
				&+(\frac{1}{2c_\tau}+\frac{5}{c_\tau\beta\varsigma_A})\sum_{s=2}^{S}\sum_{t=0}^{T-1}\mathbb{E}\|x_{t+1}^s-x_t^s\|^2
				+(\frac{1}{2}+\frac{10}{\beta\varsigma_A})\frac{ST\epsilon}{c_\epsilon}\\
				\leq&-\sum_{s=1}^{S}\sum_{t=0}^{T-1}\chi_T\mathbb{E}\|x_1^{s+1}-x_0^{s+1}\|^2-\frac{L^2}{2b}\sum_{s=1}^{S}\sum_{t=0}^{T-1}\mathbb{E}\|x_T^s-\tilde{x}^{s-1}\|^2\\
				& +(\frac{1}{2c_\tau}+\frac{10}{c_\tau\beta\varsigma_A})\sum_{s=1}^{S}\sum_{t=0}^{T-1}\mathbb{E}\|x_{t+1}^s-x_t^s\|^2
				+(\frac{1}{2}+\frac{10}{\beta\varsigma_A})(\frac{1}{c_\tau}+\frac{1}{c_\epsilon})ST\epsilon.
			\end{aligned}\]
			where the second inequality follows from the definition $\tau_s=\frac{1}{T}\sum_{t=0}^{T-1}\|x_{t+1}^{s-1}-x_t^{s-1}\|^2$ for $s=2,\dots,S$, and the third inequality follows from $\tau_1\leq \epsilon S$. Thus, we obtain
			\begin{align}
				\phi_T^S-\phi_0^1\leq&-\sum_{s=1}^{S}\sum_{t=0}^{T-1}\Omega_t\mathbb{E}\|x_{t+1}^s-x_t^s\|^2-\frac{L^2}{2b}\sum_{s=1}^{S}\sum_{t=0}^{T-1}\mathbb{E}\|x_t^s-\tilde{x}^{s-1}\|^2\nonumber\\
				&+(\frac{1}{2}+\frac{10}{\beta\varsigma_A})(\frac{1}{c_\tau}+\frac{1}{c_\epsilon})ST\epsilon,\label{SM2MF3}
			\end{align}
			where $\Omega_t=\frac{\zeta_{\min}}{\eta}+\frac{\beta\varsigma_A}{2}-\frac{1}{2c_\tau}-\frac{L+1}{2}-\frac{5}{\beta \varsigma_A}(L^2+\frac{2}{c_\tau}+\frac{2\zeta_{\max}^2}{\eta^2})-(1+\frac{1}{\rho})\delta_{t+1}$.
			
			Choose the parameters $\eta,\beta,c_\tau,\rho>0$ used in Algorithm \ref{alg2} such that $\Omega_t>0$. By the definition of function $\phi_t^s$, we have
			\[\begin{aligned}
				\phi_t^s\geq\mathbb{E}[\mathcal{L}_\beta(w_t^s)]
				=&\mathbb{E}[f(x_t^s)+g(y_t^s)-\langle\lambda_t^s,Ax_t^s+By_t^s-c\rangle+\frac{\beta}{2}\|Ax_t^s+By_t^s-c\|^2]\\
				=&\mathbb{E}[f(x_t^s)+g(y_t^s)-\frac{1}{\beta}\langle\lambda_t^s,\lambda_{t-1}^s-\lambda_t^s\rangle+\frac{1}{2\beta}\|\lambda_{t-1}^s-\lambda_t^s\|^2]\\
				=&\mathbb{E}[f(x_t^s)+g(y_t^s)-\frac{1}{2\beta}\|\lambda_{t-1}^s\|^2+\frac{1}{2\beta}\|\lambda_t^s\|^2+\frac{1}{2\beta}\|\lambda_{t-1}^s-\lambda_t^s\|^2]\\
				\geq&\mathbb{E}[f(x_t^s)+g(y_t^s)-\frac{1}{2\beta}\|\lambda_{t-1}^s\|^2+\frac{1}{2\beta}\|\lambda_t^s\|^2].
			\end{aligned}\]
			Summing the above inequality over $t$ from 0 to $T-1$ and over $s$ from 1 to $S$, we have
			\[\sum_{s=1}^{S}\sum_{t=0}^{T-1}\phi_t^s\geq f^*+g^*+\frac{1}{2\beta}\|\lambda_0^1\|^2.\]
			Thus, the function $\phi_t^s$ is bounded from below. Let $\phi^*$ denotes a lower bound of $\phi_t^s$.
			It follows from \eqref{SM2MF3} that
			\[\frac{1}{ST}\sum_{s=1}^{S}\sum_{t=0}^{T-1}(\Omega_t\mathbb{E}\|x_{t+1}^s-x_t^s\|^2+\frac{L^2}{2b}\mathbb{E}\|x_t^s-\tilde{x}^{s-1}\|^2)\leq \frac{\phi_0^1-\phi^*}{ST}+(\frac{1}{2}+\frac{10}{\beta\varsigma_A})(\frac{1}{c_\tau}+\frac{1}{c_\epsilon})\epsilon.\]
			The proof is completed.
		\end{proof}
	
	Next, we prove the convergence of Algorithm \ref{alg2} in Theorem \ref{AdsCR2}.
		\begin{proof}[\textbf{Proof of Theorem \ref{AdsCR2}}]\label{THE2}
			First, we define an useful variable $\theta_t^s=\mathbb{E}\|x_{t+1}^s-x_t\|^2+\mathbb{E}\|x_t^s-x_{t-1}^s\|^2+\frac{1}{b}(\|x_t^s-\tilde{x}^s\|^2+\|x_{t-1}^s-\tilde{x}^s\|^2)$. By the step 9 in Algorithm \ref{alg2}, similar to \eqref{SM1disty}, we have
			\begin{equation}\label{SM2disty}
				\mathbb{E}[{\rm dist}(B^\top \lambda_{t+1}^s,\partial_g (y_{t+1}^s))^2]\leq \beta^2 \|B\|\|A\|\theta_t^s.
			\end{equation}
			By the step 10 of Algorithm \ref{alg2}, similar to \eqref{SM1distx},  we have
			\begin{align}
				&\mathbb{E}\|A^\top \lambda_{t+1}-\nabla f(x_{t+1})\|^2=\mathbb{E}\|v_t^s-\nabla f(x_t^s)+\frac{G}{\eta}(x_{t+1}^s-x_t^s)+\nabla f(x_t^s)-\nabla f(x_{t+1}^s)\|^2\nonumber\\
				&\overset{(i)}{\leq} \frac{3L^2}{b} \mathbb{E}\|x_t^s-\tilde{x}^{s-1}\|^2+\frac{3I_{(N_s<n)}}{N_s}\sigma^2+ 3(L^2+\frac{\zeta_{\max}^2}{\eta^2})\mathbb{E}\|x_{t+1}^s-x_t^s\|^2\nonumber\\
				&\overset{(ii)}{\leq} \frac{3L^2}{b} \mathbb{E}\|x_t^s-\tilde{x}^{s-1}\|^2+3(\frac{\tau_s}{c_\tau}+\frac{\epsilon}{c_\epsilon})+ 3(L^2+\frac{\zeta_{\max}^2}{\eta^2})\mathbb{E}\|x_{t+1}^s-x_t^s\|^2\nonumber\\
				&\leq 3(L^2+\frac{\zeta_{\max}^2}{\eta^2})\theta_t^s+3(\frac{\tau_s}{c_\tau}+\frac{\epsilon}{c_\epsilon}),\label{SM2distx}
			\end{align}
			where $(i)$ follows by Lemma \eqref{SM2SG} and $(ii)$ follows by \eqref{SM2NS}.
			By the step 11 of Algorithm \ref{alg2}, we have
			\begin{align}
				&\mathbb{E}\|Ax_{t+1}^s+By_{t+1}^s-c\|^2
				=\frac{1}{\beta^2}\mathbb{E}\|\lambda_{t+1}^s-\lambda_t^s\|^2\nonumber\\
				&\leq\frac{5L^2}{b \varsigma_A\beta^2}(\mathbb{E}\|x_t^s-\tilde{x}^s\|^2+\mathbb{E}\|x_{t-1}^s-\tilde{x}^s\|^2)+\frac{10\sigma^2I_{(N_s<n)}}{N_s \varsigma_A\beta^2}\nonumber\\
				&\quad+\frac{5\zeta_{\max}^2}{\varsigma_A\beta^2\eta^2}\mathbb{E}\|x_{t+1}^s-x_t^s\|^2+\frac{5}{\varsigma_A\beta^2}(\frac{\zeta_{\max}^2}{\eta^2}+L^2)\mathbb{E}\|x_t^s-x_{t-1}^s\|^2\nonumber\\
				&\leq  \frac{5}{ \varsigma_A\beta^2 }(\frac{\zeta_{\max}^2}{\eta^2}+L^2)\theta_t^k+\frac{10}{\varsigma_A\beta^2 }(\frac{\tau_s}{c_\tau}+\frac{\epsilon}{c_\epsilon}),\label{SM2distlam}
			\end{align}
			where the first inequality holds by Lemma \ref{SM2bdual}. Let
			\[\psi_1=\beta^2 \|B\|\|A\|,~\psi_2=3(L^2+\frac{\zeta_{\max}^2}{\eta^2}) ,~\psi_3=\frac{5}{\beta^2  \varsigma_A}(L^2+\frac{\zeta_{\max}^2}{\eta^2}).\]
			By \eqref{SM2disty}-\eqref{SM2distlam}, Definition \ref{stationary} and Lemma \ref{SM2MFD}, we have
			\[\begin{aligned}
				&\min_{1\leq k \leq K}\mathbb{E}[{\rm dist}(0,\partial \mathcal{L}(w_t^s))^2]
				\leq\frac{\psi}{K}\sum_{s=1}^{S}\sum_{t=0}^{T-1}\theta_t^s+\frac{\max\{3,\frac{10}{\varsigma_A\beta^2}\}}{c_\tau K}\sum_{s=1}^{S}T\tau_s+\frac{\epsilon}{c_\epsilon}\max\{3,\frac{10}{\varsigma_A\beta^2}\}\\
				&\leq\frac{\psi}{K}\sum_{s=1}^{S}\sum_{t=0}^{T-1}\theta_t^s+\frac{\max\{3,\frac{10}{\varsigma_A\beta^2}\}}{c_\tau K}\sum_{s=1}^{S}\sum_{t=0}^{T-1}\|x_{t+1}^s-x_t^s\|^2+\max\{3,\frac{10}{\varsigma_A\beta^2}\}(\frac{1}{c_\tau}+\frac{1}{c_\epsilon})\epsilon\\
				&\leq\frac{1}{K}(\psi+\frac{1}{c_\tau }\max\{3,\frac{10}{\varsigma_A\beta^2}\})\sum_{s=1}^{S}\sum_{t=0}^{T-1}\theta_t^s+\max\{3,\frac{10}{\varsigma_A\beta^2}\}(\frac{1}{c_\tau}+\frac{1}{c_\epsilon})\epsilon\\
				&\leq\frac{2\omega_1(\phi_0^1-\phi^*)}{\gamma K}+\omega_2(\frac{1}{c_\tau}+\frac{1}{c_\epsilon})\epsilon,
			\end{aligned}\]
			where $\omega_1=\psi+\frac{1}{c_\tau}\max\{3,\frac{10}{\varsigma_A\beta^2}\}$, $\omega_2=\max\{3,\frac{10}{\varsigma_A\beta^2}\}+\frac{\omega_1}{\gamma }(1+\frac{20}{\beta\varsigma_A})$, $\psi=\max\{\psi_1,\psi_2,\psi_3\}$, $\gamma=\min\{\Omega_t,\frac{L^2}{2}\}$, the second inequality holds by $\tau_1\leq \epsilon S$ and $\tau_s=\frac{1}{T}\sum_{t=0}^{T-1}\|x_{t+1}^{s-1}-x_t^{s-1}\|^2$ for $s=2,\dots,S$.
		\end{proof}
	
	We analyses the complexity of Algorithm \ref{alg2} in Corollary \ref{AbssvrgOC}.
		\begin{proof}[\textbf{Proof of Corollary \ref{AbssvrgOC}}]\label{coro1}
			Let $\rho=\frac{1}{T}$, recursing on $t$ for $0\leq t\leq T-1$, we have
			\[\begin{aligned}
				\delta_{t+1}&=(\frac{L^2}{2b}+\frac{10L^2}{b\beta\varsigma_A})\frac{(1+\rho)^{T-t}-1}{\rho}=\frac{T}{b}(\frac{L^2}{2}+\frac{10L^2}{\beta\varsigma_A})((1+\frac{1}{T})^{T-t}-1)\\
				&\leq\frac{T}{b}(\frac{L^2}{2}+\frac{10L^2}{\beta\varsigma_A})(e-1)\leq \frac{2T}{b}(\frac{L^2}{2}+\frac{10L^2}{\beta\varsigma_A}),
			\end{aligned}\]
			where the inequality holds by $(1+\frac{1}{T})^T$ is an increasing function and $\lim_{T\to\infty}(1+\frac{1}{T})^T=e$. It follows that,
			\[\begin{aligned}
				&\Omega_t-\frac{L^2}{2}=\frac{\zeta_{\min}}{\eta}+\frac{\beta\varsigma_A}{2}-\frac{1}{2c_\tau}-\frac{L^2+L+1}{2}-\frac{5}{\beta \varsigma_A}(L^2+\frac{2}{c_\tau}+\frac{2\zeta_{\max}^2}{\eta^2})-(1+\frac{1}{\rho})\delta_{t+1}\\
				\geq&\frac{\zeta_{\min}}{\eta}+\frac{\beta\varsigma_A}{2}-\frac{1}{2c_\tau}-\frac{L^2+L+1}{2}-\frac{5}{\beta \varsigma_A}(L^2+\frac{2}{c_\tau}+\frac{2\zeta_{\max}^2}{\eta^2})-(1+T)\frac{2T}{b}(\frac{L^2}{2}+\frac{10L^2}{\beta\varsigma_A})\\
				\geq&\frac{\zeta_{\min}}{\eta}+\frac{\beta\varsigma_A}{2}-\frac{1}{2c_\tau}-\frac{L^2+L+1}{2}-\frac{5}{\beta \varsigma_A}(L^2+\frac{2}{c_\tau}+\frac{2\zeta_{\max}^2}{\eta^2})-\frac{4T^2}{b}(\frac{L^2}{2}+\frac{10L^2}{\beta\varsigma_A})\\
				\overset{(i)}{=}&\frac{\zeta_{\min}}{\eta}-\frac{5L^2+L+1}{2}-\frac{1}{2}+\frac{\beta\varsigma_A}{2}-\frac{5}{\beta \varsigma_A}(9L^2+2+\frac{2\zeta_{\max}^2}{\eta^2}),
			\end{aligned}\]
			where $(i)$ holds by $T=[n^{\frac{1}{3}}]$, $b=[n^{\frac{2}{3}}]$, and $c_\tau=c_\epsilon=9+\frac{12}{L^2}(3+\beta^2 \|B\|\|A\|)>1$. 
			Let $\beta\geq\max\{\frac{20}{\varsigma_A},\sqrt{\frac{10}{3\varsigma_A}},\sqrt{\frac{3}{\|B\|\|A\|}(L^2+\frac{\zeta_{\max}^2}{\eta^2})},\frac{1}{\varsigma_A}\sqrt{10(2+9L^2+\frac{2\zeta_{\max}^2}{\eta^2})}\}$ and $\eta=\frac{2\zeta_{\min}}{5L^2+L+2}$. 
			Then, we have $\Omega_t>\frac{L^2}{2}$ and $\psi_3<\psi_2<\psi_1$.
			It easy verifies that $\psi=\psi_1=\beta^2 \|B\|\|A\|=\mathcal{O}(1)$ and $\gamma= \frac{L^2}{2}=\mathcal{O}(1)$, which are independent on $n$ and $K$.
			
			Thus, we have
			\[\begin{aligned}
				\min_{1\leq k \leq K}\mathbb{E}[{\rm dist}(0,\partial \mathcal{L}(w_k))^2]\leq&\frac{2\omega_1(\phi_0^1-\phi^*)}{\gamma K}+\omega_2(\frac{1}{c_\tau}+\frac{1}{c_\epsilon})\epsilon\\
				\leq&\frac{4\omega_1(\phi_0^1-\phi^*)}{L^2 K}+(3+\frac{4\omega_1}{L^2})\frac{2}{c_\epsilon}\epsilon\\
				\leq&\frac{4(3+\|B\|\|A\|)(\phi_0^1-\phi^*)}{L^2 K}+\frac{2}{3}\epsilon.
			\end{aligned}\]
			Then, to find an $\epsilon$-approximate solution, the total number of iterations required by Algorithm \ref{alg2} is given by \[K=\frac{12(3+\|B\|\|A\|)(\phi_0^1-\phi^*)}{L^2 \epsilon}=\mathcal{O}(\epsilon^{-1}),\]
			and the total number of complexity calls required by Algorithm \ref{alg2} is given by
			\[\begin{aligned}
				&\sum_{s=1}^{S}\min\{\frac{c_\tau\sigma^2}{\frac{1}{T}\sum_{t=0}^{T-1}\|x_{t+1}^{s-1}-x_t^{s-1}\|^2},c_\epsilon\sigma^2\epsilon^{-1},n\}+Kb\\
				\leq&
				S\min\{c_\epsilon\sigma^2\epsilon^{-1},n\}+Kb=\mathcal{O}(n+n^{\frac{2}{3}}\epsilon^{-1}).
			\end{aligned}
			\]
			The proof is completed.
		\end{proof}
	
	\section{Convergence Analysis for AbsSPIDER-SADMM}\label{app3}
	Throughout this section, let $n_k=\lceil  k/q \rceil$ such that $(n_k-1)q\leq k \leq n_k q-1$.
    \begin{lemma}\citep{JWWZZL20}\label{BEV3}
    Suppose that Assumptions \ref{nograb} and \ref{Lsm} hold. Algorithm \ref{alg3} generates the stochastic gradient $\{v_k\}$ satisfies 
		\[
			\mathbb{E}\|v_k-\nabla f(x_k)\|^2\leq \frac{L^2}{b}\sum_{i=(n_k-1)q}^{k-1} \mathbb{E}\|x_{i+1}-x_i\|^2+\frac{I_{(N_k<n)}}{N_k}\sigma^2.
		\]
    \end{lemma}
	\begin{lemma}\label{SM3bdual}
		Under Assumptions \ref{nograb}-\ref{fucor} and given the sequence $\{w_k\}$ from Algorithm \ref{alg3}, it holds that
		\[
        \begin{aligned}
			\|\lambda_{k+1}-\lambda_k\|^2&\leq \frac{5}{\varsigma_A}\|v_k-\nabla f(x_k)\|^2+\frac{5}{\varsigma_A}\|\nabla f(x_{k-1})-v_{k-1}\|^2\\
			&\quad+\frac{5\zeta_{\max}^2}{\varsigma_A\eta^2}\|x_{k+1}-x_k\|^2+\frac{5}{\varsigma_A}(\frac{\zeta_{\max}^2}{\eta^2}+L^2)\|x_k-x_{k-1}\|^2,
		\end{aligned}
		% \begin{aligned}
		% 	\mathbb{E}\|\lambda_{k+1}-\lambda_k\|^2 \leq& \frac{10L^2}{b \varsigma_A}\sum_{i=(n_k-1)q}^{k-1} \mathbb{E}\|x_{i+1}-x_i\|^2+\frac{10\sigma^2I_{(N_k<n)}}{N_k \varsigma_A}\nonumber\\
		% 	&+\frac{5\zeta_{\max}^2}{\varsigma_A\eta^2}\mathbb{E}\|x_{k+1}-x_k\|^2
		% 	+\frac{5}{\varsigma_A}(\frac{\zeta_{\max}^2}{\eta^2}+L^2)\mathbb{E}\|x_k-x_{k-1}\|^2,
		% \end{aligned}
		\]
		where $I_{(A)} = 1$ if the event $A$ occurs and 0 otherwise.
	\end{lemma}
	\begin{proof}
		Using the optimal condition of the step 10 in Algorithm \ref{alg3}, we have
		\[v_k-A^\top\lambda_k+\beta A^\top(Ax_{k+1}+By_{k+1}-c)+\frac{G}{\eta}(x_{k+1}-x_k)=0.\]
		Using the step 11 of Algorithm \ref{alg3}, we have
		\begin{equation}\label{SM3Alamk1}
			A^\top\lambda_{k+1}=v_k+\frac{G}{\eta}(x_{k+1}-x_k).
		\end{equation}
		By Assumption \ref{fucor}, we have
		\[\begin{aligned}
			\|\lambda_{k+1}-\lambda_k\|^2
			&\leq  \frac{1}{\varsigma_A}\|v_k+\frac{G}{\eta}(x_{k+1}-x_k)-v_{k-1}-\frac{G}{\eta}(x_k-x_{k-1})\|^2\\
			&\leq \frac{5}{\varsigma_A}(\|v_k-\nabla f(x_k)\|^2+\|\nabla f(x_k)-\nabla f(x_{k-1})\|^2\\
			&\quad+\|\nabla f(x_{k-1})-v_{k-1}\|^2+\|\frac{G}{\eta}(x_{k+1}-x_k)\|^2+\|\frac{G}{\eta}(x_k-x_{k-1})\|^2) \nonumber\\
			&\leq \frac{5}{\varsigma_A}\|v_k-\nabla f(x_k)\|^2+\frac{5}{\varsigma_A}\|\nabla f(x_{k-1})-v_{k-1}\|^2\\
			&\quad+\frac{5\zeta_{\max}^2}{\varsigma_A\eta^2}\|x_{k+1}-x_k\|^2+\frac{5}{\varsigma_A}(\frac{\zeta_{\max}^2}{\eta^2}+L^2)\|x_k-x_{k-1}\|^2,
		\end{aligned}\]
		where the last inequality holds by Assumption \ref{Lsm}. The proof is completed.
		% Taking expectation $\mathbb{E}(\cdot)$ over the above inequality yields,
		% \[
		% \begin{aligned}
		% 	&\mathbb{E}\|\lambda_{k+1}-\lambda_k\|^2\\
		% 	\leq &\frac{5L^2}{b \varsigma_A}\sum_{i=(n_k-1)q}^{k-1} \mathbb{E}\|x_{i+1}-x_i\|^2+\frac{5L^2}{b \varsigma_A}\sum_{i=(n_k-1)q}^{k-2} \mathbb{E}\|x_{i+1}-x_i\|^2+\frac{10\sigma^2I_{(N_k<n)}}{N_k \varsigma_A}\\
		% 	&+\frac{5\zeta_{\max}^2}{\varsigma_A\eta^2}\|x_{k+1}-x_k\|^2+\frac{5}{\varsigma_A}(\frac{\zeta_{\max}^2}{\eta^2}+L^2)\|x_k-x_{k-1}\|^2\\
		% 	\leq& \frac{10L^2}{b \varsigma_A}\sum_{i=(n_k-1)q}^{k-1} \mathbb{E}\|x_{i+1}-x_i\|^2+\frac{10\sigma^2I_{(N_k<n)}}{N_k \varsigma_A}+\frac{5\zeta_{\max}^2}{\varsigma_A\eta^2}\mathbb{E}\|x_{k+1}-x_k\|^2\\
		% 	&
		% 	+\frac{5}{\varsigma_A}(\frac{\zeta_{\max}^2}{\eta^2}+L^2)\mathbb{E}\|x_k-x_{k-1}\|^2,
		% \end{aligned}
		% \]
		% where the first inequality follows from Assumptions \ref{nograb} and \ref{Lsm}, and \eqref{BEV3}. 
	\end{proof}
	
	\begin{lemma}\label{SM3MFD}
	Under Assumptions \ref{nograb}-\ref{fucor}, suppose the sequence $\{w_k\}$ is generated from Algorithm \ref{alg3}, and define a merit function $\phi_k$ as follows:
	\[\phi_k=\mathbb{E}\left[\mathcal{L}_\beta(w_k)+\frac{5}{\beta\varsigma_A}(\frac{\zeta_{\max}^2}{\eta^2}+L^2)\|x_k-x_{k-1}\|^2+\frac{5L^2}{b\beta\varsigma_A}\sum_{i=(n_k-1)q}^{k-1} \mathbb{E}\|x_{i+1}-x_i\|^2\right].\]
	Let $\epsilon >0$ and $c_\tau,c_\epsilon>0$. Choose the parameters $\tau_0,b,q,c_\tau,\eta,\beta>0$ such that $\chi>0$ and $\tau_0=c_d\epsilon$ with $1\leq c_d\leq q$. Then we have
	\[\frac{\chi}{K}\sum_{i=(n_k-1)q}^{k}\|x_{i+1}-x_i\|^2\leq\frac{\phi_0-\phi^*}{ K}+(\frac{1}{2}+\frac{10}{\beta\varsigma_A})(\frac{c_d}{c_\tau}+\frac{1}{c_\epsilon})\epsilon,\]
	where $\chi=\frac{\zeta_{\min}}{\eta}+\frac{\beta\varsigma_A}{2}-\frac{L+1}{2}-\frac{10\zeta_{\max}^2}{\beta \varsigma_A\eta^2}-\frac{5L^2}{\beta \varsigma_A}-\frac{5L^2}{b\beta\varsigma_A}-\frac{qL^2}{2b}-\frac{10qL^2}{b\beta\varsigma_A}-\frac{1}{2c_\tau}-\frac{10}{c_\tau\beta\varsigma_A}$ and $\phi^*$ denotes a lower bound of $\phi_k$.
	\end{lemma}
	\begin{proof}
		From the definition of the augmented Lagrangian function $\mathcal{L}_\beta$, it follows that
		\begin{align}
			&\mathcal{L}_\beta(x_{k+1},y_{k+1},\lambda_k)-\mathcal{L}_\beta(x_k,y_{k+1},\lambda_k)\nonumber\\
			= & f(x_{k+1})-f(x_k)-\langle\lambda_k, A(x_{k+1}-x_k)\rangle+\frac{\beta}{2}\|Ax_{k+1}+By_{k+1}-c\|^2\nonumber\\
			&-\frac{\beta}{2}\|Ax_k+By_{k+1}-c\|^2 \nonumber\\
			=& f(x_{k+1})-f(x_k)-\langle\lambda_k-\beta(Ax_{k+1}+By_{k+1}-c), A(x_{k+1}-x_k)\rangle-\frac{\beta}{2}\|A(x_{k+1}-x_k)\|^2\nonumber\\
            =& f(x_{k+1})-f(x_k)-\langle A^\top\lambda_{k+1}, x_{k+1}-x_k\rangle-\frac{\beta}{2}\|A(x_{k+1}-x_k)\|^2\nonumber
            \end{align}
            \begin{align}
			\overset{(i)}{=}&f(x_{k+1})-f(x_k)-\langle v_k+\frac{G}{\eta}(x_{k+1}-x_k), x_{k+1}-x_k\rangle-\frac{\beta}{2}\|A(x_{k+1}-x_k)\|^2\nonumber\\
			\overset{(ii)}{\leq} &\langle \nabla f(x_k)-v_k, x_{k+1}-x_k\rangle+\frac{L}{2}\|x_{k+1}-x_k\|^2-\frac{1}{\eta}\|x_{k+1}-x_k\|_G^2-\frac{\beta}{2}\|A(x_{k+1}-x_k)\|^2\nonumber\\
			\leq &\frac{1}{2}\|\nabla f(x_k)-v_k\|^2-(\frac{\zeta_{\min}}{\eta}+\frac{\beta\varsigma_A}{2}-\frac{L+1}{2})\|x_{k+1}-x_k\|^2,\label{SM3Lxk1k}
		\end{align}
		where $(i)$ holds by \eqref{SM3Alamk1} and $(ii)$ holds by Assumption \ref{Lsm}.
		Since $y_{k+1}$ is a minimizer of step 9 in Algorithm \ref{alg3}, we have
		\begin{equation}\label{SM3Lyk1k}
			\mathcal{L}_\beta(x_k,y_{k+1},\lambda_k)
			\leq\mathcal{L}_\beta(x_k,y_k,\lambda_k).
		\end{equation}
		Using the step 11 in Algorithm \ref{alg3}, we have
		\begin{equation}\label{SM3Llamk1k}
			\mathcal{L}_\beta(x_{k+1},y_{k+1},\lambda_{k+1})-\mathcal{L}_\beta(x_{k+1},y_{k+1},\lambda_k)=\frac{1}{\beta}\|\lambda_{k+1}-\lambda_k\|^2.
		\end{equation}
		Combining \eqref{SM3Lxk1k}-\eqref{SM3Llamk1k} and Lemma \ref{SM3bdual}, we have
		\[\begin{aligned}
			&\mathcal{L}_\beta(w_{k+1})-\mathcal{L}_\beta(w_k)\\
			\leq &\frac{1}{2}\|\nabla f(x_k)-v_k\|^2-(\frac{\zeta_{\min}}{\eta}
			+\frac{\beta\varsigma_A}{2}-\frac{L+1}{2})\|x_{k+1}-x_k\|^2+\frac{1}{\beta}\|\lambda_{k+1}-\lambda_k\|^2.
			% \leq & \frac{1}{2}\|\nabla f(x_k)-v_k\|^2-(\frac{\zeta_{\min}}{\eta}+\frac{\beta\varsigma_A}{2}-\frac{L+1}{2})\|x_{k+1}-x_k\|^2+\frac{5\zeta_{\max}^2}{\beta\varsigma_A\eta^2}\mathbb{E}\|x_{k+1}-x_k\|^2\\
			% &+\frac{10L^2}{b\beta \varsigma_A}\sum_{i=(n_k-1)q}^{k-1} \mathbb{E}\|x_{i+1}-x_i\|^2+\frac{10\sigma^2I_{(N_k<n)}}{\beta N_k \varsigma_A}+\frac{5}{\beta\varsigma_A}(\frac{\zeta_{\max}^2}{\eta^2}+L^2)\mathbb{E}\|x_k-x_{k-1}\|^2.
		\end{aligned}\]
		Taking expectation $\mathbb{E}(\cdot)$ over the above inequality and following from Lemmas \ref{BEV3} and \ref{SM3bdual}, we have
		\[\begin{aligned}
			&\mathbb{E}(\mathcal{L}_\beta(w_{k+1})-\mathcal{L}_\beta(w_k))\\
			\leq & (\frac{L^2}{2b}+\frac{10L^2}{b\beta\varsigma_A})\sum_{i=(n_k-1)q}^{k-1} \mathbb{E}\|x_{i+1}-x_i\|^2+\frac{5}{\beta  \varsigma_A}(\frac{\zeta_{\max}^2}{\eta^2}+L^2)\mathbb{E}\|x_k-x_{k-1}\|^2\\
			&+\frac{\sigma^2 I_{(N_k<n)}}{N_k}(\frac{1}{2}+\frac{10}{\beta\varsigma_A})-(\frac{\zeta_{\min}}{\eta}+\frac{\beta\varsigma_A}{2}-\frac{L+1}{2}-\frac{5\zeta_{\max}^2}{\beta \varsigma_A\eta^2})\mathbb{E}\|x_{k+1}-x_k\|^2.
		\end{aligned}\]
		Define the following merit function,
		\[\phi_k=\mathbb{E}\left[\mathcal{L}_\beta(w_k)+\frac{5}{\beta\varsigma_A}(\frac{\zeta_{\max}^2}{\eta^2}+L^2)\|x_k-x_{k-1}\|^2+\frac{5L^2}{b\beta\varsigma_A}\sum_{i=(n_k-1)q}^{k-1} \mathbb{E}\|x_{i+1}-x_i\|^2\right].\]
		Then we have
		\begin{align}
			\phi_{k+1}-\phi_k
			\leq&-(\frac{\zeta_{\min}}{\eta}+\frac{\beta\varsigma_A}{2}-\frac{L+1}{2}-\frac{10\zeta_{\max}^2}{\beta \varsigma_A\eta^2}-\frac{5L^2}{\beta \varsigma_A}-\frac{5L^2}{b\beta\varsigma_A})\mathbb{E}\|x_{k+1}-x_k\|^2\nonumber\\
			&+(\frac{L^2}{2b}+\frac{10L^2}{b\beta\varsigma_A})\sum_{i=(n_k-1)q}^{k-1} \mathbb{E}\|x_{i+1}-x_i\|^2+\frac{\sigma^2 I_{(N_k<n)}}{N_k}(\frac{1}{2}+\frac{10}{\beta\varsigma_A}),\label{SM3MF1}
		\end{align}
		where the inequality holds by
		\begin{equation}\label{xi1ikk1}
			\sum_{i=(n_k-1)q}^{k} \mathbb{E}\|x_{i+1}-x_i\|^2=\sum_{i=(n_k-1)q}^{k-1} \mathbb{E}\|x_{i+1}-x_i\|^2+\mathbb{E}\|x_{k+1}-x_k\|^2.
		\end{equation}
		Since $(n_k-1)q\leq k\leq n_kq-1$, and let $(n_k-1)q\leq l\leq n_kq-1$, telescoping \eqref{SM3MF1} over $k$ from $(n_k-1)q$ to $k$, we have
		\begin{align}
			&\quad\phi_{k+1}-\phi_{(n_k-1)q}\nonumber\\
			&\overset{(i)}{\leq}-(\frac{\zeta_{\min}}{\eta}+\frac{\beta\varsigma_A}{2}-\frac{L+1}{2}-\frac{10\zeta_{\max}^2}{\beta \varsigma_A\eta^2}-\frac{5L^2}{\beta \varsigma_A}-\frac{5L^2}{b\beta\varsigma_A}) \sum_{l=(n_k-1)q}^{k}\mathbb{E}\|x_{l+1}-l_k\|^2\nonumber\\
			&\quad+(\frac{L^2}{2b}+\frac{10L^2}{b\beta\varsigma_A})\sum_{l=(n_k-1)q}^{k}\sum_{i=(n_k-1)q}^{k-1} \mathbb{E}\|x_{i+1}-x_i\|^2\nonumber\\
			&\quad+\sum_{l=(n_k-1)q}^{k}\frac{\tau_l}{c_\tau}(\frac{1}{2}+\frac{10}{\beta\varsigma_A})+\frac{\epsilon}{c_\epsilon}(\frac{1}{2}+\frac{10}{\beta\varsigma_A})\nonumber\\
			&\overset{(ii)}{\leq}-(\frac{\zeta_{\min}}{\eta}+\frac{\beta\varsigma_A}{2}-\frac{L+1}{2}-\frac{10\zeta_{\max}^2}{\beta \varsigma_A\eta^2}-\frac{5L^2}{\beta \varsigma_A}-\frac{5L^2}{b\beta\varsigma_A}) \sum_{l=(n_k-1)q}^{k}\mathbb{E}\|x_{l+1}-l_k\|^2\nonumber\\
            &\quad+(\frac{qL^2}{2b}+\frac{10qL^2}{b\beta\varsigma_A})\sum_{i=(n_k-1)q}^{k-1} \mathbb{E}\|x_{i+1}-x_i\|^2\nonumber\\
			&\quad+(\frac{1}{2c_\tau}+\frac{10}{c_\tau\beta\varsigma_A})\sum_{i=(n_k-1)q}^{k}\|x_{i+1}-x_i\|^2+(\frac{1}{2}+\frac{10}{\beta\varsigma_A})(\frac{c_d}{c_\tau}+\frac{1}{c_\epsilon})\epsilon\nonumber\\
			&\leq-\chi \sum_{i=(n_k-1)q}^{k}\|x_{i+1}-x_i\|^2+(\frac{1}{2}+\frac{10}{\beta\varsigma_A})(\frac{c_d}{c_\tau}+\frac{1}{c_\epsilon})\epsilon,\label{SM3MF}
		\end{align}
		where $\chi=\frac{\zeta_{\min}}{\eta}+\frac{\beta\varsigma_A}{2}-\frac{L+1}{2}-\frac{10\zeta_{\max}^2}{\beta \varsigma_A\eta^2}-\frac{5L^2}{\beta \varsigma_A}-\frac{5L^2}{b\beta\varsigma_A}-\frac{qL^2}{2b}-\frac{10qL^2}{b\beta\varsigma_A}-\frac{1}{2c_\tau}-\frac{10}{c_\tau\beta\varsigma_A}$, $(i)$ holds by \eqref{SM2NS} and $(ii)$ holds by $\tau_0=c_d\epsilon$ with the constant $1\leq c_d \leq q$ and $\tau_k=\frac{1}{q}\sum_{i=(n_k-1)q}^{k}\|x_{i+1}-x_i\|^2$ for $k=1,\dots,K-1$.
		
		Since $(A^\top)^+=A(A^\top A)^{-1}$ by Assumption \ref{fucor} and
		\[\begin{aligned}
			&\langle (A^\top)^+ (v_k+\frac{G}{\eta}(x_{k+1}-x_k)), Ax_{k+1}+By_{k+1}-c\rangle\\
			\leq& \frac{1}{2\beta}\|(A^\top)^+ (v_k+\frac{G}{\eta}(x_{k+1}-x_k))\|^2+\frac{\beta}{2}\|Ax_{k+1}+By_{k+1}-c\|^2,
		\end{aligned}
		\]
		we have
		\[\begin{aligned}
			&\mathbb{E}[f(x_{k+1})+g(y_{k+1})-\lambda_{k+1}^\top (Ax_{k+1}+By_{k+1}-c)+\frac{\beta}{2}\|Ax_{k+1}+By_{k+1}-c\|^2]\\
			\overset{(i)}{=}&\mathbb{E}[f(x_{k+1})+g(y_{k+1})-\langle (A^\top)^+(v_k+\frac{G}{\eta}(x_{k+1}-x_k)), Ax_{k+1}+By_{k+1}-c\rangle \\
			&+\frac{\beta}{2}\|Ax_{k+1}+By_{k+1}-c\|^2]\\
			\geq&\mathbb{E}[f(x_{k+1})+g(y_{k+1})-\frac{1}{2\beta\varsigma_A}\|v_k+\frac{G}{\eta}(x_{k+1}-x_k)-\nabla f(x_k)+\nabla f(x_k)\|^2]\\
			\geq&\mathbb{E}[f(x_{k+1})+g(y_{k+1})-\frac{3}{2\beta\varsigma_A}\|v_k-\nabla f(x_k)\|^2-\frac{3}{2\beta\varsigma_A}\|\nabla f(x_k)\|^2-\frac{3\zeta_{\max}^2}{2\beta\eta^2\varsigma_A}\|x_{k+1}-x_k\|^2]\\
			\overset{(ii)}{\geq}&\mathbb{E}[f(x_{k+1})+g(y_{k+1})-\frac{3}{2\beta\varsigma_A}(\frac{L^2}{b}\sum_{i=(n_k-1)q}^{k-1} \mathbb{E}\|x_{i+1}-x_i\|^2+\frac{I_{(N_k<n)}}{N_k}\sigma^2)\\
			&-\frac{3\mu^2}{2\beta\varsigma_A}-\frac{3\zeta_{\max}^2}{2\beta\eta^2\varsigma_A}\|x_{k+1}-x_k\|^2]\\
			\overset{(iii)}{\geq}&\mathbb{E}[f(x_{k+1})+g(y_{k+1})-\frac{3}{2\beta\varsigma_A}(\frac{L^2}{b}+\frac{1}{q c_\tau})\sum_{i=(n_k-1)q}^{k-1} \mathbb{E}\|x_{i+1}-x_i\|^2\\
			&-\frac{3}{2\beta\varsigma_A}(\frac{\zeta_{\max}^2}{\eta^2}+\frac{1}{qc_\tau})\|x_{k+1}-x_k\|^2-\frac{3}{2\beta\varsigma_A}\left(\mu^2+(\frac{c_d}{c_\tau}+\frac{1}{c_\epsilon})\epsilon\right)],
		\end{aligned}\]
		where $(i)$ holds by \eqref{SM3Alamk1}, $(ii)$ holds by \eqref{BEV3} and Assumption \ref{grab}, $(iii)$  holds by the definition of $\tau_k$ and \eqref{xi1ikk1}. Using the definition of $\phi_k$ and Assumption \ref{lb}, we have
		\[\phi_{k+1}\geq f^*+g^*-\frac{3}{2\beta\varsigma_A}(\mu^2+(\frac{c_d}{c_\tau}+\frac{1}{c_\epsilon})\epsilon).\]
		It follows that the function $\phi_k$ is bounded from below. Let $\phi^*$ denotes a lower bound of $\phi_k$.
		
		Given $b,q,c_\tau,\eta,\beta>0$ such that $\chi>0$. Telescoping \eqref{SM3MF} over $k$ from 0 to $K-1$, we have
		\[
		\frac{\chi}{K}\sum_{i=(n_k-1)q}^{k}\|x_{i+1}-x_i\|^2\leq\frac{\phi_0-\phi^*}{ K}+(\frac{1}{2}+\frac{10}{\beta\varsigma_A})(\frac{c_d}{c_\tau}+\frac{1}{c_\epsilon})\epsilon.
		\]
		The proof is completed.
	\end{proof}
	
	Next, we prove the convergence of Algorithm \ref{alg3} in Theorem \ref{AdsCR3}.
	\begin{proof}[\textbf{Proof of Theorem \ref{AdsCR3}}]\label{THE3}
		First, we define a useful variable $\theta_k=\mathbb{E}[\|x_{k+1}-x_k\|^2+\|x_k-x_{k-1}\|^2+\frac{1}{q}\sum_{i=(n_k-1)q}^{k}\|x_{i+1}-x_i\|^2]$. By the optimal condition of the step 9 in Algorithm \ref{alg3}, similar to \eqref{SM2disty}, we have,
		\begin{equation}\label{SM3disty}
			\mathbb{E}[{\rm dist}(B^\top \lambda_{k+1},\partial_g (y_{k+1}))^2]\leq  \beta^2 \|B\|\|A\|\theta_k.
		\end{equation}
		By the step 10 of Algorithm \ref{alg3}, we have
		\begin{align}
			&\mathbb{E}\|A^\top \lambda_{k+1}-\nabla f(x_{k+1})\|^2\nonumber\\
			=&\mathbb{E}\|v_k+\frac{G}{\eta}(x_{k+1}-x_k)-\nabla f(x_k)+\nabla f(x_k)-\nabla f(x_{k+1})\|^2\nonumber\\
			\overset{\eqref{BEV3}}{\leq}& \frac{3L^2}{b}\sum_{i=(n_k-1)q}^{k-1} \mathbb{E}\|x_{i+1}-x_i\|^2+\frac{3I_{(N_k<n)}}{N_k}\sigma^2+ 3(L^2+\frac{\zeta_{\max}^2}{\eta^2})\|x_{k+1}-x_k\|^2\nonumber\\
			\overset{\eqref{SM2NS}}{\leq}&\frac{3L^2}{b}\sum_{i=(n_k-1)q}^{k-1} \mathbb{E}\|x_{i+1}-x_i\|^2+3(L^2+\frac{\zeta_{\max}^2}{\eta^2})\|x_{k+1}-x_k\|^2+3(\frac{\tau_s}{c_\tau}+\frac{\epsilon}{c_\epsilon})\nonumber\\
			\leq&3(L^2+\frac{\zeta_{\max}^2}{\eta^2})\theta_k+3(\frac{\tau_k}{c_\tau}+\frac{\epsilon}{c_\epsilon}).\label{SM3distx}
		\end{align}
		By the step 11 of Algorithm \ref{alg3}, we have
		\begin{align}
			&\mathbb{E}\|Ax_{k+1}+By_{k+1}-c\|^2
			=\frac{1}{\beta^2}\mathbb{E}\|\lambda_{k+1}-\lambda_k\|^2\nonumber\\
			\leq& \frac{10L^2}{b \beta^2\varsigma_A}\sum_{i=(n_k-1)q}^{k-1} \mathbb{E}\|x_{i+1}-x_i\|^2+\frac{10\sigma^2I_{(N_k<n)}}{N_k \beta^2\varsigma_A}+\frac{5\zeta_{\max}^2}{\beta^2\varsigma_A\eta^2}\mathbb{E}\|x_{k+1}-x_k\|^2\nonumber\\
			&+\frac{5}{\beta^2\varsigma_A}(\frac{\zeta_{\max}^2}{\eta^2}+L^2)\mathbb{E}\|x_k-x_{k-1}\|^2\nonumber
            \end{align}
            \begin{align}
			\leq& \frac{10L^2}{b \beta^2\varsigma_A}\sum_{i=(n_k-1)q}^{k-1} \mathbb{E}\|x_{i+1}-x_i\|^2+\frac{10}{ \beta^2\varsigma_A}(\frac{\tau_s}{c_\tau}+\frac{\epsilon}{c_\epsilon})+\frac{5\zeta_{\max}^2}{\beta^2\varsigma_A\eta^2}\mathbb{E}\|x_{k+1}-x_k\|^2\nonumber\\
			&+\frac{5}{\beta^2\varsigma_A}(\frac{\zeta_{\max}^2}{\eta^2}+L^2)\mathbb{E}\|x_k-x_{k-1}\|^2\nonumber\\
			\leq&\frac{5}{\beta^2\varsigma_A}(\frac{\zeta_{\max}^2}{\eta^2}+2L^2)\theta_k+\frac{10}{ \beta^2\varsigma_A}(\frac{\tau_s}{c_\tau}+\frac{\epsilon}{c_\epsilon}),\label{SM3distlam}
		\end{align}
		where the first inequality holds by Lemma \ref{SM3bdual}. Let
		\[\psi_1=\beta^2 \|B\|\|A\|,~\psi_2=3(L^2+\frac{\zeta_{\max}^2}{\eta^2}) ,~\psi_3=\frac{5}{\beta^2\varsigma_A}(2L^2+\frac{\zeta_{\max}^2}{\eta^2}).\]
		Since
		\[\sum_{k=1}^{K-1}\sum_{i=(n_k-1)q}^{k}\|x_{i+1}-x_i\|^2\leq q\sum_{k=1}^{K-1}\|x_{k+1}-x_k\|^2,\]
		by \eqref{SM3disty}-\eqref{SM3distlam} and Lemma \ref{SM3MFD}, we have
		\[\begin{aligned}
			&\min_{1\leq k \leq K}\mathbb{E}[{\rm dist}(0,\partial \mathcal{L}(w_k))^2]\leq \frac{\psi}{K} \sum_{k=1}^{K-1}\theta_k+(\frac{\tau_s}{c_\tau}+\frac{\epsilon}{c_\epsilon})\max\{\frac{10}{\beta^2 \varsigma_A},3\}\\
			&\leq\frac{1}{K}(\psi+\frac{1}{c_\tau }\max\{3,\frac{10}{\varsigma_A\beta^2}\})\sum_{k=1}^{K-1}\theta_k+\max\{3,\frac{10}{\varsigma_A\beta^2}\}(\frac{c_d}{c_\tau}+\frac{1}{c_\epsilon})\epsilon\\
			&\leq\frac{3\omega_1(\phi_0-\phi^*)}{\chi K}+\omega_2(\frac{c_d}{c_\tau}+\frac{1}{c_\epsilon})\epsilon,
		\end{aligned}
		\]
		where $\omega_1=\psi+\frac{1}{c_\tau}\max\{3,\frac{10}{\varsigma_A\beta^2}\}$, $\omega_2=\frac{3\omega_1}{\chi}(\frac{1}{2}+\frac{10}{\beta\varsigma_A})+\max\{\frac{10}{\beta^2 \varsigma_A},3\}$, $\psi=\max\{\psi_1,\psi_2,\psi_3\}$, the second inequality holds by the definition of $\tau_k$.
	\end{proof}
	
	We analyses the complexity of Algorithm \ref{alg3} in Corollary \ref{AbsspiderOC}.
	\begin{proof}[\textbf{Proof of Corollary \ref{AbsspiderOC}}] \label{coro2}
		Given $b=q\geq 1$ and $c_\tau=c_\epsilon=\frac{9(1+c_d)}{2}(4+\beta^2 \|B\|\|A\|)>1$. It follows that,
		\[\chi-1\geq\frac{\zeta_{\min}}{\eta}-\frac{L^2+L+4}{2}+\frac{\beta\varsigma_A}{2}-\frac{10}{\beta \varsigma_A}(1+2L^2+\frac{\zeta_{\max}^2}{\eta^2}).\]
		Let $\beta\geq\max\{\frac{20}{\varsigma_A},\sqrt{\frac{10}{3\varsigma_A}},\sqrt{\frac{3}{\|B\|\|A\|}(L^2+\frac{\zeta_{\max}^2}{\eta^2})},\frac{1}{\varsigma_A}\sqrt{20(1+2L^2+\frac{\zeta_{\max}^2}{\eta^2})}\}$ and $0<\eta<\frac{2\zeta_{\min}}{L^2+L+4}$.
		Then, we have $\chi>1$ and $\psi_3<\psi_2<\psi_1$.
		It easy verifies that $\psi=\psi_1=\beta^2 \|B\|\|A\|=\mathcal{O}(1)$, which is independent on $n$ and $K$.
		Thus, we have
		\[\begin{aligned}
			\min_{1\leq k \leq K}\mathbb{E}[{\rm dist}(0,\partial \mathcal{L}(w_k))^2]\leq&\frac{3\omega_1(\phi_0-\phi^*)}{\chi K}+\omega_2(\frac{c_d}{c_\tau}+\frac{1}{c_\epsilon})\epsilon\\
			\leq&\frac{3\omega_1(\phi_0-\phi^*)}{\chi K}+3(\omega_1+1)\frac{1+c_d}{c_\epsilon}\epsilon\\
			\leq&\frac{3(\psi+3)(\phi_0-\phi^*)}{\chi K}+\frac{2}{3}\epsilon.
		\end{aligned}\]
		Then, to find an $\epsilon$-approximate solution, the total number of iterations required by Algorithm \ref{alg3} is given by \[K=\frac{9(3+\beta^2\|B\|\|A\|)(\phi_0-\phi^*)}{\chi \epsilon}=\mathcal{O}(\epsilon^{-1}),\]
		and the total number of oracle complexity calls required by Algorithm \ref{alg3} is given by
		\[\begin{aligned}
			&\underset{\rm complexity~of~AbsSPIDER-ADMM}{\underbrace{\sum_{k=0}^{K-1}\min\{\frac{c_\tau\sigma^2}{\frac{1}{q}\sum_{i=(n_k-1)q}^{k}\|x_{i+1}-x_i\|^2},c_\epsilon\sigma^2\epsilon^{-1},n\}+Kb}}\\
			\leq&\underset{\rm complexity~of~SPIDER-ADMM}
			{\underbrace{K\min\{c_\epsilon\sigma^2\epsilon^{-1},n\}+Kb=\mathcal{O}(n+n^{\frac{1}{2}}\epsilon^{-1})}}.
		\end{aligned}\]
		The proof is completed.
	\end{proof}

%% If you have bib database file and want bibtex to generate the
%% bibitems, please use
%%
%%  \bibliographystyle{elsarticle-harv} 
%%  \bibliography{<your bibdatabase>}

%% else use the following coding to input the bibitems directly in the
%% TeX file.

%% Refer following link for more details about bibliography and citations.
%% https://en.wikibooks.org/wiki/LaTeX/Bibliography_Management

\bibliographystyle{elsarticle-harv} 
\bibliography{references}

\end{document}